\documentclass{article}

\usepackage{amsmath}
\usepackage{amsthm}
\usepackage{thmtools}
\usepackage{amsfonts}
\usepackage{graphicx}
\usepackage{xspace}
\usepackage{parskip}
\usepackage[all]{xy}
\usepackage{color}
\usepackage{framed}
\usepackage{fullpage}
\graphicspath{{figures/}}

\makeatletter
\def\thm@space@setup{%
  \thm@preskip=\parskip \thm@postskip=0pt
}
\makeatother

\declaretheorem[parent=section]{lemma}
\declaretheorem[sibling=lemma, name=Proposition]{prop}
\declaretheorem[sibling=lemma]{theorem}
\declaretheorem[sibling=lemma]{claim}
\declaretheorem[sibling=lemma,style=definition]{remark}
\declaretheorem[sibling=lemma,style=definition]{definition}
\declaretheorem[sibling=lemma]{corollary}
\declaretheorem[sibling=lemma]{question}

\usepackage{xspace}

\newcommand{\bdry}{\partial}
\newcommand{\boundary}{\partial}
\newcommand{\set}[1]{\left\{#1\right\}}

\newcommand{\closure}[1]{\overline{#1}}
\newcommand{\scc}{simple closed curve\xspace}
\newcommand{\sccs}{simple closed curves\xspace}
\newcommand{\mobius}{M\"obius\xspace}

\newcommand{\interior}[1]{\operatorname{int} #1}
\newcommand{\co}{\colon}
\newcommand{\R}{\mathbb{R}}
\renewcommand{\L}{\mathcal{L}}

\newcommand{\N}{\mathbb{N}}
\newcommand{\Rhat}{\widehat{R}}
\newcommand{\Uhat}{\widehat{U}}

\newcommand{\barQ}{\overline{Q}}
\newcommand{\barF}{\overline{F}}
\newcommand{\algint}[2]{#1\cdot#2}
\newcommand{\geomint}[2]{\Delta(#1,#2)}
\newcommand{\ie}{\textit{i.e.}\xspace}
\newcommand{\abs}[1]{\left|#1\right|}
\newcommand{\ceil}[1]{\left\lceil#1\right\rceil}

\newcommand{\Z}{\mathbb{Z}}

\newcommand{\gwrap}{\hbar}

\newcommand{\defn}[1]{\emph{#1}}
\newcommand{\cf}{cf.\xspace}

\newcommand{\Pcal}{\mathcal{P}}

\newcommand{\Am}{A_{-}}
\newcommand{\Ap}{A_{+}}

\newcommand{\Lm}{L_{-}}
\newcommand{\Lp}{L_{+}}
\newcommand{\Tp}{T_+}
\newcommand{\Tm}{T_-}
\newcommand{\Tk}{T_K}

\newcommand{\cut}{\setminus}
\newcommand{\nbhd}{N}%{\mathcal{N}}

\begin{document}
\title{Boundary-reducing surgeries and bridge number}
\author{Kenneth L.\ Baker \and R.\ Sean Bowman\and John Luecke}
%\author{Kenneth L.\ Baker}
%\author{R.\ Sean Bowman} 
%\author{John Luecke}

\date{\today}
\maketitle

\begin{abstract}
    Let $M$ be a $3$--dimensional handlebody of genus $g$. This paper gives
    examples of hyperbolic knots in $M$ with arbitrarily large genus $g$ bridge
    number which admit Dehn surgeries which are boundary-reducible manifolds.
\end{abstract}

\section{Introduction}
A $3$--dimensional  handlebody of genus $g$ is a $3$--ball with $g$ one--handles
attached.  In general, when the manifolds before and after Dehn surgery on
a knot are homeomorphic, the surgery is called {\em cosmetic}.  This paper
studies cosmetic surgeries on a handlebody. Let $K$ be a knot in a genus $g$
handlebody $M$.  When does $K$ admit a cosmetic surgery, i.e.\ a Dehn surgery
which is again a genus $g$ handlebody?

When a knot $K$ in any $3$--manifold $M$ is isotopic into the boundary of $M$,
then $K$ admits infinitely many surgeries that are manifolds homeomorphic to
$M$ --- infinitely many cosmetic surgeries. In particular, when $M$ is
a handlebody of genus $g$, one obtains cosmetic surgeries on any knot parallel
into the boundary of $M$. When $g=0$, \cite{GordonLuecke89} shows that this is the only way
a cosmetic surgery arises. When $g=1$, Berge and Gabai (\cite{Berge91,Gabai89, Gabai90}) give
a classification of knots which admit cosmetic surgeries. Not all are isotopic
into the boundary, but, if not, they show that they have bridge number one in
the solid torus. Indeed, showing that the bridge number is one is an important
part of the classification of these knots. Berge and Gabai also give
(\cite{Berge, Gabai90}) examples of knots in handlebodies of genus $g \geq 2$ which 
admit cosmetic surgeries but which are not isotopic into the boundary of the
handlebody. These turn out also to have bridge number one in the the
handlebody. In \cite{Wu93} (see also \cite[Question 4.5]{Gabai90} and \cite{Berge} ), Wu conjectures 
that this is always the case ---  knots in
handlebodies admitting cosmetic surgeries are always either isotopic into the
boundary (i.e.\  bridge number zero) or  bridge number one.  
However, 
Bowman gives 
counterexamples to this conjecture in \cite{Bowman13},  exhibiting knots in a genus
$2$ handlebody that admit cosmetic surgeries but have bridge number greater
than one.  In Theorem~\ref{statement1}(1), we generalize Bowman's result in two
directions: $(a)$ we show that certain infinite subcollections of the knots in
\cite{Bowman13} in fact have arbitrarily large bridge number (in particular,
see Corollary~\ref{seansknots}), and $(b)$ we give such infinite families for
handlebodies of arbitrary genus. For a knot $K$ in a genus $g$ handlebody,
$b_g(K)$ denotes the genus $g$ bridge number of $K$.

A compact $3$--manifold is said to be {\em boundary-reducible} if it contains
a properly embedded disk whose boundary is essential in the boundary of the
manifold, and {\em boundary-irreducible} otherwise.  There has been much interest in the question: when does
a knot in a handlebody admit a Dehn surgery which is a boundary-reducible
manifold? (see \cite{CGLS}, \cite[Question 4.5]{Gabai90}, and \cite{Wu92,Wu93}). 
Note that a handlebody of
genus $g>0$ is boundary-reducible, so the cosmetic surgeries above are examples
of such surgeries. For genus $1$ handlebodies, the knots which admit such
surgeries are classified (\cite{Berge91,Gabai89, Gabai90}, and \cite{Scharlemann90} for the reducible case),
and in particular the only hyperbolic knots which admit a boundary-reducible
surgery are $1$--bridge in the solid torus. In Theorem~\ref{statement1}(2) we
give, for any genus $g>1$, infinite families of hyperbolic knots with arbitrarily
large bridge number that admit Dehn surgeries giving a Seifert fiber space with
$g-1$ one-handles attached (hence  a boundary-reducible manifold).

\begin{definition}\label{torusknot}
A {\em $D(p,q)$--Seifert space}  is a Seifert fiber space over the disk with two exceptional fibers of order $p,q$.
\end{definition}

\begin{theorem}\label{statement1}
  Let $M$ be a handlebody of genus $g>1$. For every positive integer $N$
the following hold.
  \begin{enumerate}
  \item There are infinitely many knots $K\subseteq M$ such that $K$ admits a nontrivial handlebody surgery and
      \[ b_g(K) \geq N. \]
 	Furthermore, the knots may be taken to have the same genus $g$
      bridge number.

  \item There are infinitely many pairs of relatively prime integers $p$ and $q$
      such that for each pair, there are infinitely many knots $K\subseteq M$
      admitting a surgery yielding a $D(p,q)$--Seifert space with $(g-1)$
      one-handles attached.  Furthermore, for each such $K$ 
	\[ b_g(K)\geq N. \]
      Finally, fixing $(p,q)$, the knots may be taken to have the same genus
      $g$ bridge number.
  \end{enumerate}

  The knots in each family above have exteriors in $M$ which are irreducible,
  boundary-irreducible, atoroidal, and anannular.  The slope of each surgery
  is longitudinal, that is, intersecting the meridian once, and there is
  \emph{only one} nontrivial surgery on each knot which is a boundary-reducible manifold.
\end{theorem}

This is Theorem~\ref{thm:main}, proven in section~\ref{sec:bridgenumberbounds}.

If one does not worry about establishing a lower bound on bridge number, the construction  for Theorem~\ref{statement1} generalizes to give more examples of knots in
a genus $g$ handlebody upon which Dehn surgery gives boundary-reducible manifolds.

\begin{theorem}
Let $M$ be a handlebody of genus $g>1$ and $p,q$ be  non-zero, relatively prime
integers.  There are infinitely many knots $K\subseteq M$ admitting a
longitudinal surgery yielding a $D(p,q)$--Seifert space with $(g-1)$
one-handles attached.  Furthermore,  the exterior in $M$ of each knot is  irreducible,
boundary-irreducible, atoroidal, and anannular. Finally, for fixed $p,q$,
these knots can be taken  to have the same genus $g$ bridge number in $M$.
\end{theorem}

This is Theorem~\ref{thm:main2}, proven in section~\ref{sec:bridgenumberbounds}.

Note that in Theorem~\ref{statement1}(1), we show that the cosmetic surgery on
these knots with large bridge number is unique. There are examples of knots in
handlebodies that admit more than one cosmetic surgery and whose exteriors are
irreducible, boundary-irreducible, atoroidal, and anannular, see (Berge \cite{Berge}, Frigerio et.\ al.\ \cite{frigerio}). 
Note that by \cite{Wu92}, any such knot will admit at most two
cosmetic surgeries.  Knots admitting more than one cosmetic surgery should be much more 
special, and in fact the
only examples known to date have bridge number one. So we ask the following:

\begin{question}
Let $K$ be a knot in a handlebody whose exterior
is  irreducible, boundary-irreducible, atoroidal, and anannular. If $K$
admits more than one cosmetic surgery, must it have bridge number one in the
handlebody?
\end{question}

The {\em surgery dual knot} to a Dehn surgery is the knot in the resulting manifold that is the core of the attached solid
torus. Theorem~\ref{statement1}(1) gives families of knots in a handlebody with large bridge number and that admit cosmetic surgeries.
The surgery dual knots to these knots are also knots in handlebodies admitting cosmetic surgeries, 
and we can ask about their bridge
number. Wu shows in Theorem 5 of \cite{Wu93} that a knot in a handlebody with a cosmetic surgery has 
bridge number at most one if and only if its surgery dual does too. Thus the surgery duals in the families given
above will have bridge number greater than one.

\begin{question}
Is there a family of knots in a handlebody  with the following properties?
\begin{enumerate}
\item Each knot admits a cosmetic surgery.
\item The bridge numbers for these knots and for their surgery duals grow without bound.
\item The exterior of each knot in the family is irreducible, boundary-irreducible, atoroidal, and anannular.
\end{enumerate}  
\end{question}

The paper is organized as follows. In section~\ref{section:little-handlebody}
we define a collection of knots in a genus $2$ handlebody, $H$. In
section~\ref{section:big-handlebody}, we glue this genus $2$ handlebody to
a genus $g$ handlebody, $H'$, along a $3$--punctured sphere that has nice
properties in both $H$ and $H'$. This produces a genus $g$ handlebody $M$, and
the knots defined for $H$ become the knots in $M$. In
section~\ref{sec:essentialsurfaces}, we show how twisting a knot along an
annulus can increase the hitting number of the knot with surfaces of fixed
Euler chararcteristic in $M$. In section~\ref{sec:wrapping}, we use twisting
along an annulus to produce infinitely many hyperbolic knots in $M$ which admit
handlebody and boundary-reducible surgeries. In
section~\ref{sec:bridgenumberbounds}, we use twisting along an annulus to find
infinite families of knots among those from  section~\ref{sec:wrapping} that
also have large bridge number. In section~\ref{section:genus2}, we clarify the
connection between the knots constructed in this paper and the knots
constructed in \cite{Bowman13}.

\section{The little handlebody}\label{section:little-handlebody}

In this section we construct a family of knots in the genus 2 handlebody
$P\times I$, where $P$ is a pair of pants (that is, a $2$--sphere
minus three disjoint open disks) and $I=[-1,1]$.  The surface
$P\times\set{1}$ is $\boundary$--incompressible in the complement of the knots,
but becomes $\boundary$--compressible after some nontrivial surgery.  We call
this handlebody the \defn{little handlebody} in contrast with the big
handlebody of Section~\ref{section:big-handlebody}.  The big handlebody is
constructed by gluing the little handlebody to another handlebody along
$P\times \{-1\}$.

We consider two knots in a $3$--manifold $M$ to be the same if there is
a homeomorphism of $M$ to itself taking one knot to the other.

\medskip

\begin{definition}\label{def:essentialannulus}
    Let $M$ be a $3$--manifold with boundary. A properly embedded disk in $M$ is
    said to be essential if its boundary does not bound a disk in $\partial M$.
    When $M$ is a handlebody, such a disk is called a \defn{meridian disk} of
    $M$.  A properly embedded annulus is said to be essential if it is
    incompressible and not parallel into $\partial M$.
\end{definition}

\medskip

\begin{definition}
    When $L$ is a $1$--manifold properly embedded in a $3$--manifold, $M$, we
    write $M_L$ to denote the exterior of $L$ in $M$, $\closure{M\setminus
    N(L)}$.
\end{definition}

\medskip

Fix a once-punctured torus $T$.  Let $\mu$ and $\lambda$ be two oriented \sccs
intersecting transversally in one point, so that the homology classes of
$\mu,\lambda$ form a right-handed  basis for $H_1(T)$.  If $\alpha$ and $\beta$
are oriented \sccs in $T$, let $\alpha+\beta$ denote the \scc homologous to
$[\alpha]+[\beta]$.  Define $\nu$ to be the \scc $\lambda+\mu$.

\begin{definition}
    Let $\tau$ be an unoriented,  non-trivial simple closed curve in $T$.  Then
    $\tau'$ is defined to be a properly embedded arc in $T$ whose complement is
    an annulus with core $\tau$.  We also write $\tau(p,q)$ when $\tau$ is
    homologous in $T$ to $\pm(p [\mu] + q [\lambda])$.
\end{definition}

Note that by proper isotopies, the arcs $\lambda'$, $\mu'$, and $\nu'$ may be
made mutually disjoint in $T$.  Also we may write $\mu = \tau(1,0)$, $\lambda
= \tau(0,1)$, and $\nu = \tau(1,1)$.  We will also consider the curves
$\lambda-\mu = \tau(-1,1)$, $\lambda+\nu = \tau(1,2)$, and $\mu+\nu
= \tau(2,1)$.

\begin{definition}\label{def:Rhat(tau)}
    Consider the genus two handlebody $H\cong T\times I$, with $I=[-1, 1]$.
    Let $\tau$ be an essential \scc in $T$.  Define the meridian disk of $H$,
    $D_{\tau} = \tau'\times I$.  Define annuli $\Am(\tau) = \tau\times [-1,
    -1/2]$, $\Rhat(\tau) = \tau\times [-1/2, 1/2]$, and $\Ap(\tau) = \tau\times
    [1/2, 1]$.  Define $K(\tau)$ to  be the knot $\tau  \times\set{0}$ in $H$.
    We refer to the components of $\boundary \Rhat(\tau)$ as $L_{\pm}(\tau)$ so
    that $\Lp(\tau)=\tau\times\set{1/2}$ and $\Lm(\tau)
    = \tau\times\set{-1/2}$.  Note that $D_{\tau}$ is disjoint from
    $L_{\pm}(\tau)$.
\end{definition}

\medskip

We will be looking at the knots above as obtained by twisting along an annulus.

\begin{definition}\label{def:twistingalongR}
    Let $K$ be a knot and $R$ be an annulus in the interior of $H$. Assume $K$
    intersects $R$ transversely. Let  $R \times [0,1]$ be a product
    neighborhood of $R$ in $H$. Let $h_n:R \times [0,1] \to R \times [0,1]$ be
    a homeomorphism gotten by $i$ complete twists along $R$, where $h_n$ is the
    identity on $R \times \{0,1\}$. Define the knot {\em K twisted $n$ times
    along $R$} as $[K-(R \times [0,1])] \cup h_n(K \cap (R \times [0,1]))$. See
    \cite{BGL} for a more detailed description.
\end{definition}

\begin{definition}\label{defn:twisting}
    Let $\tau,\kappa, \alpha$ be essential simple closed curves in $T$ with
    $\kappa$ and $\alpha$ not isotopic in $T$. If $\tau$ is the result of $n$
    positive Dehn twists of $\kappa$ along $\alpha$ we indicate this by writing
    $\tau(\kappa,\alpha,n)$ for $\tau$.   Let $\L(\kappa,\alpha)$ be the link
    $K(\kappa) \cup \Lp(\alpha) \cup \Lm(\alpha)$ in $H$. In the notation
    above,  $K(\tau(\kappa,\alpha,n))$ is the knot $K(\kappa)$ twisted $n$
    times along the annulus $\Rhat(\alpha)$, where the sign of the twisting is
    chosen to agree with a positive Dehn twist.
\end{definition}

In twisting $K(\kappa)$ along $\Rhat(\alpha)$, we will always assume that
$\kappa$ and $\alpha$ are not isotopic in $T$ and are taken to intersect
minimally in $T$.  Let $\L=\L(\kappa,\alpha)$,  and let $\Tp$, $\Tm$, and $\Tk$
be the boundary components of $H_{\L}$ corresponding to $\Lp(\alpha)$,
$\Lm(\alpha)$, and $K(\kappa)$, respectively.  By abuse of notation, we will
also refer to the part of $\Ap(\alpha)$ and $\Am(\alpha)$ lying in $H_{\L}$ as
$\Ap(\alpha)$ and $\Am(\alpha)$, respectively.  These are also annuli, and they
are essential since $\Lp$ and $\Lm$ are nontrivial in $H$.  Let $R$ be the
planar surface $\Rhat(\alpha) \cap H_{\L}$.  This surface has one longitudinal
boundary component on each of $\Tp$ and $\Tm$ and one or more coherently
oriented meridional boundary components on $\Tk$.  It is also incompressible in
$H_{\L}$ since compression would show that either $\Lp$ or $\Lm$ is trivial or
there is a non-separating sphere in $H$.

The knots obtained by twisting $K(\kappa)$ along $\Rhat(\alpha)$ are of the
type studied in~\cite{Bowman13}.
Lemmas~\ref{lem:primitivecore}, \ref{L-meets-boundary-compressing-disks}, 
\ref{lem:onearc}, \ref{lem:min-D-cap-P}, 
\ref{no-annuli-on-P}, \ref{X-irreducible}, and~\ref{lem:old3diskbusting} 
essentially appear there (in whole or in part) and are reproduced here for completeness.

\begin{lemma}\label{lem:coreofH}
    Let $\tau$ be an essential \scc in $T$. Then $K(\tau)$ is a core curve of
    the genus $2$ handlebody $H \cong T \times I$.
\end{lemma}
\begin{proof}
    Since $T \cut \tau'$ is an annulus with core curve $\tau$, we have that $H
    \cut D_\tau \cong (T \cut \tau') \times I$ is a solid torus with $K(\tau)$
    as its core.  Because $H$ is recovered by attaching a $1$--handle to this
    solid torus (reversing the chopping along $D_\tau$) we have our result.
\end{proof}

\begin{lemma}\label{lem:primitivecore}
    Let $W$ be the compression body that is the exterior of a core curve of
    a genus $2$ handlebody.   Up to isotopy, there is a unique non-separating
    compressing disk for $W$.  In particular, up to isotopy, $D_{\tau}$ is the
    unique non-separating meridian disk  for $H$ which does not meet $K(\tau)$.
\end{lemma}

\begin{proof}
    Let $\widehat{T}$ be the (closed) torus.  We may view $W$ as a thickened
    torus $\widehat{T} \times I$ with a $1$--handle attached to one side. Let
    $D$ be the cocore of this $1$--handle.    Among all non-separating
    compressing disks for $W$ that are not isotopic to $D$,  let $E$ be one
    that intersects $D$ minimally.

    If $E$ is disjoint from $D$, then we may isotop $E$ into $\widehat{T}
    \times I$ disjoint from the feet of the $1$--handle.  Since $\widehat{T}
    \times I$ is irreducible and $\bdry$--irreducible, the disk $E$ is
    $\bdry$--parallel and hence separating in $\widehat{T} \times I$.  In order
    for $E$ to be non-separating in $W$, the disk $\bdry E$ bounds in $\bdry
    (\widehat{T}\times I)$ must contain exactly one of the feet of the
    $1$--handle.   But then $E$ is isotopic to $D$ in $W$.

    Thus $E$ is not disjoint from $D$.  Let $c$ be an outermost arc of
    intersection in $D$ and let $D'$ be the subdisk it cuts off.  Surger $E$
    along $D'$ to produce disjoint disks $E_1$ and $E_2$ that intersect $D$
    fewer times than $E$.  Each must be either separating or isotopic to $D$.
    As $E_1 \cup E_2$ must also be non-separating, we may assume $E_1$ is
    isotopic to $D$ and $E_2$ is separating. If $E_2$ were $\bdry$--parallel
    then $E$ is isotopic to $D$, a contradiction.  Hence $\partial E_2$ must
    separate $\partial W$ into two punctured tori, one of which contains
    $\partial E_1$. But then banding $\partial E_2$ to $\partial E_1$ to give
    $\partial E$ shows that $\partial E$ is isotopic to $\partial E_1$ which in
    turn is isotopic to $\partial D$. Since $W$ is irreducible, $E$ must be
    isotopic to $D$, a contradiction.

    Lemma~\ref{lem:coreofH} then gives the final statement.
\end{proof}

Choose a subarc $\nu_0$ of $\boundary D_{\nu}$ such that $\nu_0\cap
(\mu\times\set{-1} \cup \lambda\times\set{1}) = \boundary\nu_0$
(there are two such choices).

Define the pair of pants $P$ to be a regular neighborhood of
$\mu\times\set{-1}$, $\lambda\times\set{1}$, and $\nu_0$.  See
Figure~\ref{fig:xyzV2}.  Then $H\cong P\times I$: cutting along the disjoint
disks $D_\mu, D_\lambda$ gives the product of a disk and  an interval --- which
reglues to give $P \times I$.  In particular $P$ is incompressible in $H$.  Let
$\boundary_- P$ be the boundary component of $P$ homotopic to
$\mu\times\set{-1}$ in $\boundary H$, $\boundary_+P$ the component homotopic to
$\lambda\times\set{1}$, and $\boundary_0 P$ the third component.  See
Figures~\ref{fig:xyzV2}(right) and \ref{fig:HandP2}(left).

\begin{figure}[h!tb]
  \begin{center}
\includegraphics[width=0.95\textwidth]{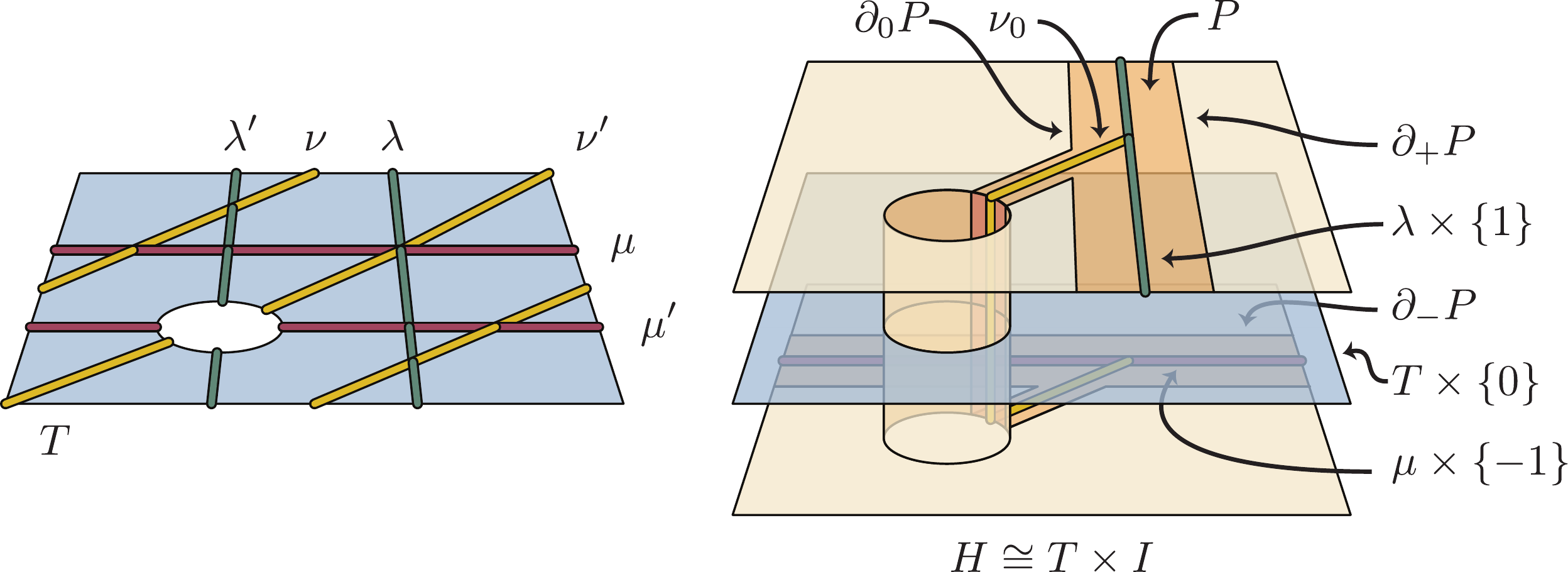}
  \end{center}

  \caption{(Left) The punctured torus $T$ with arcs $\mu'$, $\lambda'$, $\nu'$
  and closed curves $\mu$, $\lambda$, $\nu$. (Right) The handlebody $H\cong
  T\times I$ with the curves $\mu'\times\set{-1}$ and $\lambda'\times\set{1}$
  and arc $\nu_0$ in $\bdry H$.  Their regular neighborhood $P$ is also shown.}

  \label{fig:xyzV2}
\end{figure}

\begin{figure}[h!tb]
  \begin{center}
\includegraphics[width=0.95\textwidth]{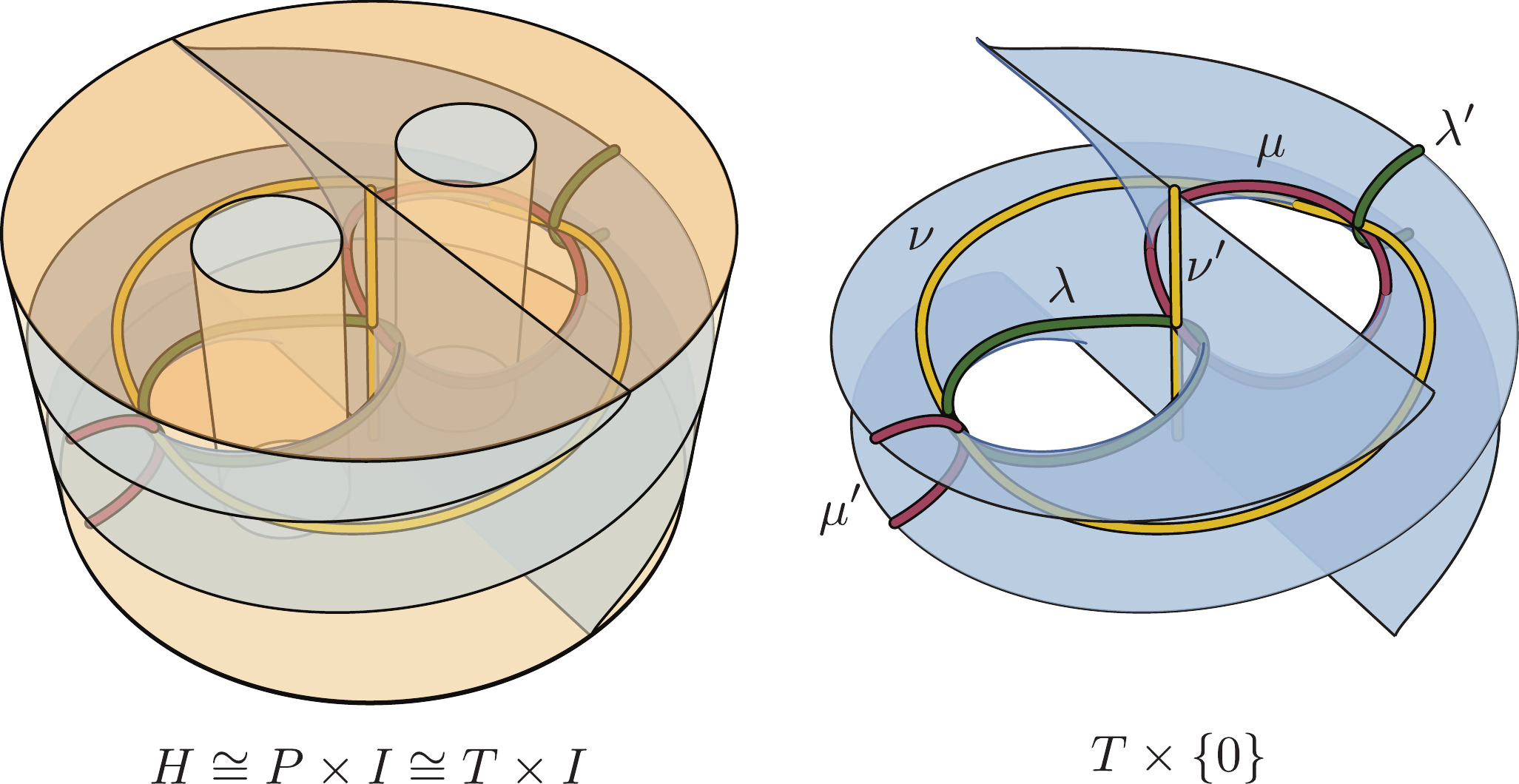}
  \end{center}

  \caption{(Left) A view of $H$ as $P\times I$. (Right) The surface $T\times
  \{0\} \subset H$ is isolated  with its basic curves and arcs.}

  \label{fig:HandP}
\end{figure}

\begin{figure}[h!tb]
  \begin{center}
\includegraphics[width=0.95\textwidth]{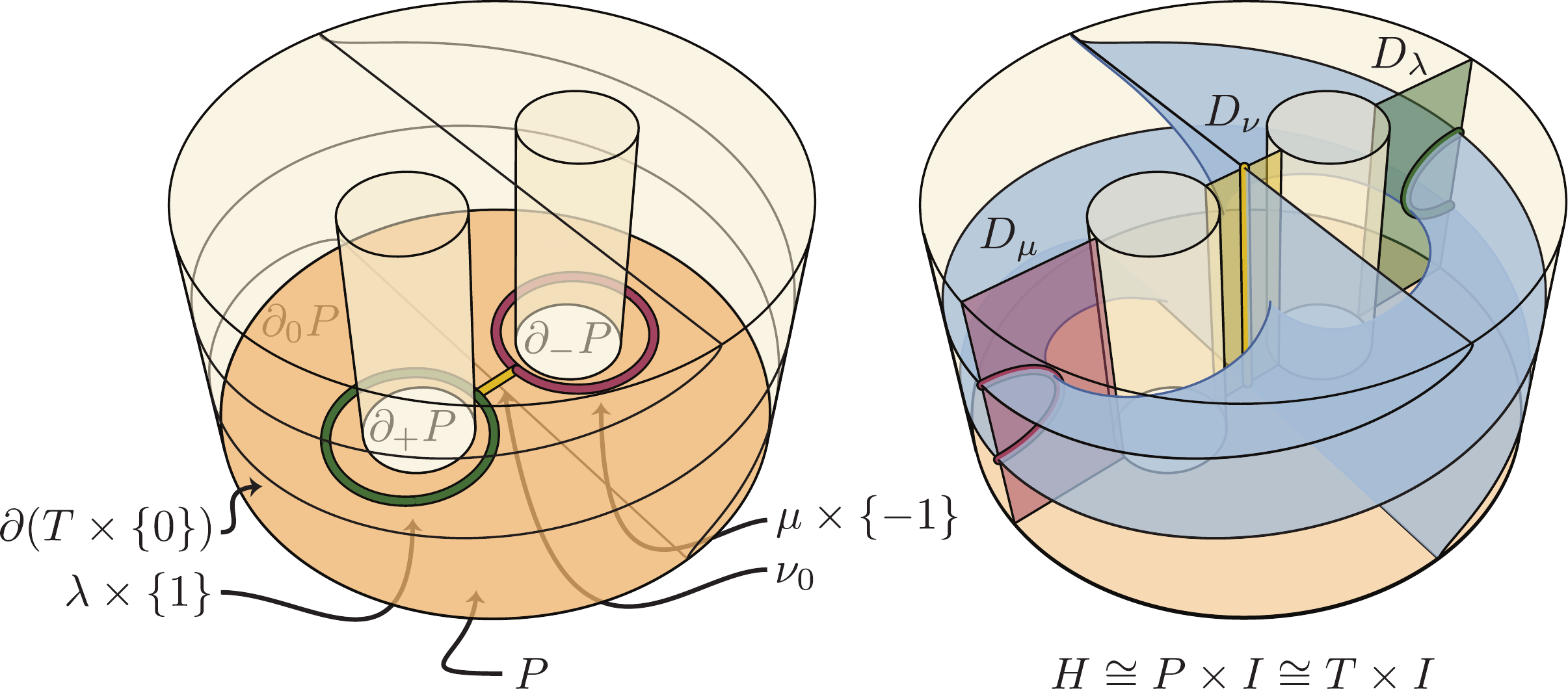}
  \end{center}
  \caption{(Left) Omitting the top and front of this view of $\bdry H$ reveals the surface $P$ and its defining curves from Figure~\ref{fig:xyzV2}.  (Right) The disks $D_\lambda$, $D_\mu$, and $D_\nu$ are also shown relative to $T\times \{0\}$ in $H$.  }
  \label{fig:HandP2}
\end{figure}

The disks $D_{\mu}$, $D_{\lambda}$, and $D_{\nu}$ are disjoint,  non-separating
$\boundary$--compressing disks for $P$ in $H$.  We give them alternate names
according to which component of $\boundary P$ they are disjoint from.  In this
way, we let $D_-=D_{\mu}$ so that $D_-\cap \boundary_-P = \emptyset$,
$D_{+}=D_{\lambda}$ so that $D_+\cap\boundary_+P = \emptyset$, and
$D_0=D_{\nu}$ so that $D_0\cap \boundary_0P=\emptyset$.  Since $H$ is
homeomorphic to the product $P\times I$, the disks $D_-$, $D_0$, and $D_+$ are
also product disks with respect to this product structure and, up to isotopy,
are the  unique non-separating $\boundary$--compressing disks for $P$ in $H$.

\begin{definition}\label{def:distance}
    For properly embedded, connected submanifolds $x,y$, we use the notation
    $\geomint{x}{y}$ for the geometric intersection number, the minimal number of
    intersections between $x$ and $y$ up to proper isotopy.  If $x$ or $y$ has
    multiple components, we take $\geomint{x}{y}$ to be maximum of this over their
    components. In the context of curves in a surface, especially tori,
    $\geomint{x}{y}$ is also referred to as the {\em distance} of $x$ and $y$. For
    oriented submanifolds $x,y$ of an oriented manifold we use $\algint{x}{y}$  for
    their algebraic intersection number. By isotoping $\tau$ in $T$ to intersect
    $\mu',\lambda',\nu'$ coherently we get the following.
\end{definition}

\begin{lemma}\label{L-meets-boundary-compressing-disks}
Let $\tau$ be a \scc in $T$.
  \begin{itemize}
  \item $\abs{\algint{K(\tau)}{D_{\lambda}}} = \abs{\algint{K(\tau)}{D_+}}
    = \geomint{\tau}{\lambda}$,
  \item $\abs{\algint{K(\tau)}{D_{\mu}}} = \abs{\algint{K(\tau)}{D_-}}
    = \geomint{\tau}{\mu}$, and
  \item $\abs{\algint{K(\tau)}{D_{\nu}}} = \geomint{\tau}{\nu}$.
  \end{itemize}
\end{lemma}

\begin{proof}
    The intersection numbers occur in $H \cong T \times I$.
\end{proof}

\begin{lemma}\label{lem:onearc}
    If $D$ is a compressing disk for $H_{K(\tau)}$ such that $D \cap P$ is
    a single arc, then $\tau$ is isotopic to $\mu$, $\lambda$, or $\nu$ in $T$.

    Equivalently, if $P$ is $\bdry$--compressible in $H_{K(\tau)}$, then $\tau$
    is isotopic to $\mu$, $\lambda$, or $\nu$ in $T$.
\end{lemma}

\begin{proof}
    Since both $P$ and $\bdry H \cut P$ are incompressible in $H \cong P \times
    I$, the arc of $D \cap P$ must be  essential in $P$.   Then, fixing this
    arc, we may isotop $D \cup K(\tau)$ so that $D$ is a product disk in $H
    \cong P\times I$.  Therefore $K(\tau)$ must be disjoint from one of the
    non-separating product disks $D_\mu$, $D_\lambda$, or $D_\nu$.   By
    Lemma~\ref{L-meets-boundary-compressing-disks}, $\tau$ may be isotoped in
    $T$ to be disjoint  from $\mu$, $\lambda$, or $\nu$ correspondingly.  Thus
    $\tau$ is isotopic to that curve.
\end{proof}

\begin{lemma}\label{lem:twoarcs}
    Assume $\tau$ is not isotopic to $\mu$, $\lambda$, or $\nu$ in $T$.  If $D$
    is a compressing disk for $H_{K(\tau)}$ such that $D \cap P$ is a pair of
    arcs $a_1 \cup a_2$, then there are meridian disks $D_1$, $D_2$ such that
    $D_i \cap P = a_i$ with a banding disjoint from $P$ that produces $D$.
\end{lemma}

\begin{proof}
    Let $D$ be a compressing disk for $H_{K(\tau)}$ such that $D \cap P$ is
    a pair of arcs $a_1 \cup a_2$.  By Lemma~\ref{lem:onearc}, these arcs must
    be essential in $P$.

    Let $b_1$ and $b_2$ be two arcs in $P$ parallel to $a_1$ so that each
    component of $P\setminus(b_1\cup b_2)$ contains exactly one of $a_1$ or
    $a_2$.  Let $F_1$ and $F_2$ denote the product disks $b_1\times I$ and
    $b_2\times I$, respectively, and choose $D$ to minimize $\abs{D\cap
    (F_1\cup F_2)}$ subject to the constraint that $D\cap P$ consists of the
    two arcs $a_1$ and $a_2$.  Isotop $K(\tau)$ to minimize $\abs{K(\tau)\cap
    (F_1\cup F_2)}$, keeping $K(\tau)$ disjoint from $D$.  Since $F_1\cup F_2$
    separates a ball $B$ from $H$ containing $a_1$ but not $a_2$, $D \cap (F_1
    \cup F_2) \neq \emptyset$.

    Let us assume, without loss of generality, that $D\cap F_1\neq\emptyset$.
    We may assume that there are no simple closed curves of intersection in
    $D\cap F_1$ since $H_{K(\tau)}$ is irreducible.  Let $c$ be an arc of
    $D\cap F_1$ outermost in $F_1$ away from $b_1$ which  cuts off a subdisk
    $F_1'$ of $F_1$ whose interior does not meet  $D$ or $b_1$.

    Assume $c$ cuts off a subdisk $D'$ of $D$ which is disjoint from $P$.  Then
    $F_1' \cup D'$ is a disk disjoint from $P$, and thus it must be
    $\bdry$--parallel in $H \cong P \times I$.  Hence there is an isotopy of $D
    \cup K(\tau)$ fixing $D \cap P$  that reduces $\abs{D\cap (F_1\cup F_2)}$.
    (We may surger $D$ along $F_1'$ and further push off of $F_1'$ to create
    a new disk that intersects $F_1 \cup F_2$ fewer times than $D$.  Since
    $F_1' \cup D'$ is boundary parallel, this surgery may be realized by an isotopy
    of $D'$ to $F_1'$.)

    Therefore $c$ divides $D$ into two halves each containing one of $a_1$ or
    $a_2$.  Surgering along $F_1'$ we obtain two parallel
    $\boundary$--compressing disks for $P$ in $H$ (they are both isotopic,
    relative to $P$, to product disks).  Thus we have shown that $D$ consists
    of $D_1$ and $D_2$, isotopic to product disks, tubed along an arc in
    $\boundary H\setminus P$.
\end{proof}

\begin{lemma}\label{lem:min-D-cap-P}
    If $\tau \not \in \{\lambda, \mu, \nu, \lambda-\mu, \lambda+\nu, \mu+\nu\}$,
    then the boundary of any compressing disk of $\partial H$ in $H_{K(\tau)}$
    meets $P$ in at least $3$ arcs.
\end{lemma}

\begin{proof}
    Let $D$ be a compressing disk in $H_{K(\tau)}$, isotoped to minimize
    $\abs{D\cap P}$.  If $\abs{D\cap P}=1$ then Lemma~\ref{lem:onearc} implies
    that $\tau$ is $\mu$, $\lambda$, or $\nu$ (in which case
    $\{\geomint{\tau}{\lambda},\geomint{\tau}{\mu}, \geomint{\tau}{\nu}\}
    = \{0,1,1\}$).   We now proceed to show that if $\abs{D\cap P}=2$, then
    $\{\geomint{\tau}{\lambda},\geomint{\tau}{\mu}, \geomint{\tau}{\nu}\}
    = \{1,1,2\}$ as multisets.  It is a simple exercise to see that this
    implies that $\tau$ is one of the curves $\lambda-\mu, \lambda+\nu$, or
    $\mu+\nu$.

    Assuming that $\abs{D\cap P}=2$, set $D \cap P = a_1 \cup a_2$.
    Lemma~\ref{lem:twoarcs} implies that $D$ is isotopic to the banding of
    product disks $a_1 \times I, a_2 \times I$ of $H\cong P \times I$ along an
    arc $\alpha$ disjoint from $P = P \times \{-1\}$.  We may isotop $\alpha$ to
    be an arc in the interior of $P \times \{1\}$ meeting each $a_1 \times
    \{1\}$ and $a_2 \times \{1\}$ in its endpoints.  Say such an arc $\alpha$
    is $\bdry$--parallel if $\alpha$ together with a subarc of each $a_1 \times
    \{1\}$, $a_2 \times \{1\}$, and $\bdry P \times \{1\}$ bounds a rectangular
    disk $R$ in $P \times \{1\}$; this rectangle guides an isotopy of the disk
    $D$ obtained by banding of $D_1$ and $D_2$ along $\alpha$ that reduces
    $\abs{D\cap P}$.  Hence in our situation, the arc $\alpha$ cannot be
    $\bdry$--parallel.

    \begin{figure}[h!tb]
        \begin{center}
            \includegraphics[width=0.9\textwidth]{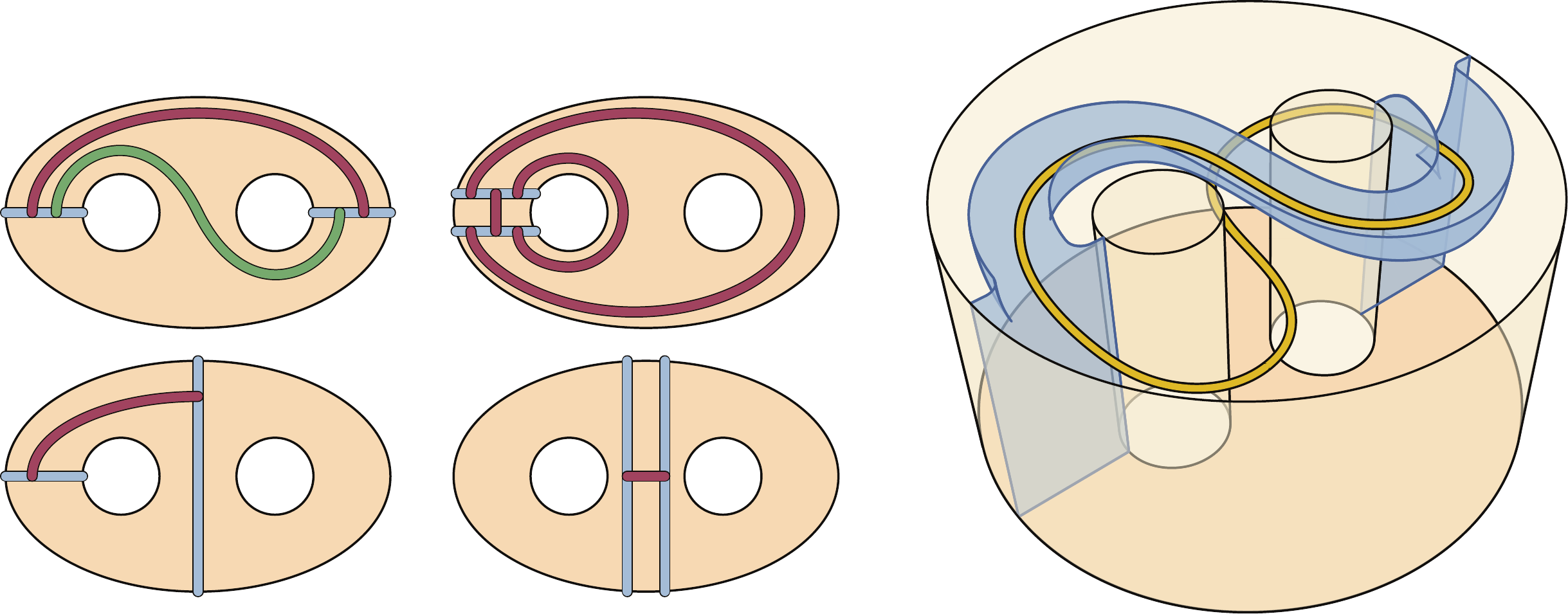}
        \end{center}
        \caption{(Left) The four possibilities for $(P, a_1 \cup a_2)$ up to
            homeomorphism.  Each shows the possible arcs $\alpha$, up to
            homeomorphism, that join the arcs $a_1$ and $a_2$.   The one which
            is not $\bdry$--parallel is in the top left. (Right)  The
            compressing disk in $H_{K(\tau)}$ up to homeomorphism.  The knot
            $K(\tau)$ intersects two of $D_\lambda$, $D_\mu$, and $D_\nu$ once
            and the other one twice.  ($D_\lambda$, $D_\mu$, and $D_\nu$ are
            the three non-separating product disks of $H \cong P\times I$.)
            }
        \label{fig:bandingsofdisks}
    \end{figure}

    Figure~\ref{fig:bandingsofdisks}(Left) exhibits the four homeomorphism
    types of  pairs of essential arcs $a_1, a_2$ embedded in $P$ along with the
    homeomorphism types of arcs $\alpha$ that connect $a_1$ and $a_2$.   By
    inspection one sees that there is a single homeomorphism type of $(P, a_1
    \cup a_2, \alpha)$ for which $\alpha$ is not $\bdry$--parallel; this is
    highlighted in Figure~\ref{fig:bandingsofdisks}(Left).  In particular $a_1$
    and $a_2$ are non-isotopic arcs in $P$ that are each non-separating.  Hence
    we must have that $D_1$ and $D_2$ are two of the three non-separating
    product disks $D_\lambda$, $D_\mu$, $D_\nu$ and $D$ is a non-separating
    disk obtained by banding them together in the manner indicated.  This gives
    six possibilities for $D$, all equivalent to
    Figure~\ref{fig:bandingsofdisks}(Right) by a level preserving homeomorphism
    of $P\times I$.  Regardless of which of these six actually is $D$, since
    $K(\tau)$ is isotopic to the core curve of the solid torus $H \cut D$ by
    Lemma~\ref{lem:primitivecore}, we can then see that $K(\tau)$ intersects
    two of the disks $D_\lambda$, $D_\mu$, $D_\nu$ just once and the third disk
    twice as claimed.
\end{proof}

%%%%%%%%%%%%%%

\section{The big handlebody}\label{section:big-handlebody}
In this section we construct knots in handlebodies which have nontrivial
surgeries yielding either handlebodies or Seifert fiber spaces with attached
$1$--handles. To get the originating handlebody, we glue the handlebody $H$ of
the previous section to another handlebody along the pair of pants $P$.
The knots of interest are those in $H$, of the preceding section, under this
embedding of $H$.

\begin{definition}\label{def:primitive}
    A \scc in the boundary of a handlebody is  \defn{primitive} if there is
    a meridian disk for the handlebody meeting the curve exactly once --- we
    refer to such a disk as a \defn{primitivizing disk} for the curve.
    Equivalently, a  \scc in the boundary of a handlebody is primitive if
    attaching a $2$--handle along it results in a handlebody.
\end{definition}

\begin{definition}\label{def:jointlyprimitive}
    Let $H'$ be a genus $g>1$ handlebody.
    We say that a set
    $\set{a_1,\dots,a_n}$ of disjoint \sccs is \defn{jointly primitive} if
    there exist disjoint disks $\set{D_1,\dots, D_n}$ in $H'$ such that
    $\abs{a_i\cap D_j}=\delta_{ij}$.  This is equivalent to the statement
    that attaching $2$--handles to $H'$ along any subset of curves in
    $\set{a_1,\dots, a_n}$ results in a handlebody~\cite{Gordon87,Wu92b}.
\end{definition}

\begin{definition}\label{def:boundaryincompressible}
    A surface, $P$, in the boundary of $3$--manifold $M$ will be said to be
    {\em $\boundary$--compressible} if there is a disk properly embedded in $M$
    that intersects $P$ is a single arc which is not parallel into $\partial P$.
    Otherwise $P$ is  {\em $\boundary$--incompressible}.
\end{definition}

\begin{definition}\label{def:handlebodytype}
    Let $a$, $b$, and $c$ be disjoint \sccs on genus $g>1$ handlebody
    $\boundary H'$ such that $a$ and $b$ are jointly primitive and  $a\cup
    b\cup c$ bounds a pair of pants $P'\subseteq \bdry H'$.  It will turn out
    that such a pair $(H', P')$ leads to a handlebody containing knots with
    nontrivial handlebody surgeries, so we say that such a pair is of
    \defn{handlebody type}.  If furthermore, there is no meridian disk or
    essential annulus of $H'$ disjoint from $c$ (i.e.\ $c$ is `disk-busting'
    and `annulus-busting', see section~\ref{sec:busting}), we say that
    $(H',P')$ is of \defn{strong handlebody type}.
\end{definition}

\begin{definition}\label{def:seiferttype}
    Let $a$, $b$, and $c$ be disjoint \sccs on $\boundary H'$ such that $a$ and
    $b$ are primitive, $a$ and $b$ are parallel, and $a\cup b\cup c$ bounds
    a pair of pants $P'\subseteq\boundary H'$ that does not lie in the
    parallelism between $a$ and $b$.  In other words, $P'$ is gotten by banding
    annular neighborhoods of $a$ and $b$ in the complement of the annulus
    cobounded by $a$ and $b$.  A pair $(H', P')$ satisfying these conditions
    leads to knots in handlebodies which have surgeries that are the union of
    $1$--handles with a Seifert fiber space over the disk with two
    exceptional fibers.  So we say that such a  pair is of \defn{Seifert type}.
    Furthermore, we say $(H',P')$ is  of \defn{strong Seifert type} if the
    following two conditions hold:

    \begin{itemize}
        \item  There is no meridian disk or essential annulus of $H'$ disjoint
            from $c$.

        \item  If $H'$ has genus 2, then $P'$ is $\boundary$--incompressible in
            $H'$.
    \end{itemize}
\end{definition}

\begin{lemma}\label{lem:old3diskbusting}
    Let $(H',P')$ be a genus $g\geq2$ handlebody and a pair of pants of either
    handlebody or Seifert type with $\boundary P'=a \cup b \cup c$ as above.
    Assume $c$ is disk-busting in $H'$. Then $P'$ and $\closure{\boundary
    H'\setminus P'}$  are incompressible in $H'$.  Furthermore, $P'$ and
    $\closure{\boundary H'\setminus P'}$ are $\boundary$--incompressible in
    $H'$ unless $g=2$, $(H',P')$ is of Seifert type, and $H'=T'\times I$ with
    $c = \bdry T'\times \{pt\}$ where $T'$ is a once-punctured torus.  In
    particular,  if $c$ intersects each  meridian disk of $H'$  at least three
    times ($c$ is $3$--disk-busting) and intersects each essential annulus of
    $H'$ ($c$ is annulus-busting), then
    $(H',P')$ is of strong handlebody or Seifert type.
\end{lemma}

\begin{proof}
    Assume $c$ is disk-busting. Clearly $P'$ and $\closure{\boundary H'\setminus
    P'}$  are incompressible in $H'$.  The
    $\boundary$--incompressibility of $P'$ is then equivalent to that of
    $\closure{\boundary H'\setminus P'}$, so we consider $P'$.  Assume $D$ is
    a $\bdry$--compressing disk for $P'$.   Note that $c$ cannot intersect
    a meridian disk of $H'$ just once, since then $c$ would be primitive and
    hence not disk-busting.  Therefore the $\boundary$--compressing disk $D$
    has to intersect $P'$ in a separating arc connecting $c$ to itself.   Then
    a primitivizing disk for $a$ that is disjoint from $D$ can be constructed
    from another primitivizing disk for $a$ by surgering away outermost
    intersections with $D$.   When $(H',P')$ is of handlebody type, this can be
    done with a primitivizing disk for $a$ disjoint from $b$ to produce a new
    disk that intersects $c$ once --- a contradiction.
    So we assume $(H',P')$ is of Seifert type.  Surger $H'$ along $D$. Then $P'
    \cut D$ is contained in an annulus in the boundary of the resulting genus
    $g-1$ handlebody, and $a$ and $b$ continue to be primitive. If $g>2$, then
    we can band a primitivizing disk for $a$ (and $b$) to itself to get
    a meridian disk disjoint from $c$ --- a contradiction. So we assume $g=2$.
    Let $T'$ be the once-punctured torus in $H'$, with boundary $c$, gotten by
    adding to $P'$ the parallelism between $a$ and $b$.  Now $a$ is primitive
    in this solid torus gotten by surgering $H'$ along $D$. Cutting again along
    a meridian disk, $D'$, for this solid torus we get the product of a disk
    and an interval.  Regluing along $D,D'$ gives the claimed product
    structure.

    The final statement follows immediately from the definition of strong
    handlebody or Seifert type given the above, once we note that $c$ being
    $3$--disk-busting guarantees that $P'$ is $\boundary$--incompressible in $H'$.
\end{proof}

\begin{definition}\label{def:bighandlebodyM}
    Let $(H',P')$ be either of handlebody or Seifert type, and $(H,P)$ be as in
    section~\ref{section:little-handlebody}.  Identify $P$ and $P'$ so that
    $\boundary_+P$ and $\boundary_-P$  are identified with $a$ and $b$.  Since
    $H$ is a product $P\times I$, the resulting space $M=H\cup H'$ is
    a handlebody containing a properly embedded pair of pants $P$
    separating $H$ and $H'$.  We call $M$ the \defn{big handlebody}. We refer
    to $M$ as having  \defn{(strong) handlebody} or \defn{(strong) Seifert
    type} according to the constituent $(H',P')$.  As in
    section~\ref{section:little-handlebody}, let $\tau$ be a \scc in $T$ where
    $H=T \times I$, and let $K(\tau)$ be the corresponding knot in $H$. Under
    the above identification of $H$ in $M$ we may then consider $K(\tau),
    K(\tau(p,q))$ or $K(\tau(\kappa,\alpha,n))$  as the corresponding  knot in
    $M$.
\end{definition}

\begin{prop}\label{k-has-handlebody-surgery}
    Let $\tau$ be an essential \scc in $T$. If $M$ is of handlebody type, then
    $K(\tau)$ has a longitudinal handlebody surgery.  If $M$ is of Seifert
    type, then $K(\tau)$ has a longitudinal surgery yielding a $D(p,q)$--Seifert space
    with $g - 1$ attached $1$--handles where $\tau=\tau(p,q)$ and
    $g$ is the genus of $M$.
\end{prop}
\begin{proof}
    We first show that $K=K(\tau) \subseteq H$ has a handlebody surgery under
    which $P$ becomes $\boundary$--compressible.  To see this, recall that $K$
    lies in the punctured torus $T\times\set{0}$.  Let $S$ be the
    $3$--punctured sphere $(T\times\set{0})\cap H_{K}$.  This surface defines
    a slope on $\boundary N(K)$, the unoriented isotopy class of any boundary
    component. Because the geometric intersection number of this slope and the
    meridian on $\partial N(K)$ is one, we say this slope is {\em
    longitudinal}. Since $K$ is isotopic to a primitive curve on $\boundary H$,
    surgery at this slope yields a genus two handlebody.  After surgery, $S$
    becomes a separating disk $D$ meeting $P$  in a single essential arc
    disjoint from $\boundary P_+$ and $\boundary P_{-}$ (see
    Figure~\ref{fig:HandP2}).  Call the surgery slope $\gamma$, and denote by
    $H(\gamma)$ the surgered handlebody.

    If $M$ has handlebody type, perform the $\boundary$--compression along $D$
    and glue the resulting two solid tori to $H'$, along annuli whose cores are
    $a$ and $b$.  The cores of the gluing annuli are jointly primitive in $H'$,
    so the result is a handlebody.  Reversing the $\boundary$--compression does
    not change this and reconstructs $M$.

    When $M$ has Seifert type,  with $\tau=\tau(p,q)$, we need to give
    coordinates to the $\gamma$ surgery on $P\subseteq H$.  As before,
    $H(\gamma)$ contains a separating disk $D$ meeting $P$ in a single arc.
    The disk $D$ separates $H(\gamma)$ into two solid tori $J_+$ and $J_-$,
    with $\boundary_+ P\subseteq J_+$ and $\boundary_- P\subseteq J_-$.

    Recall that as oriented curves $\mu$ and $\lambda$ form an oriented basis
    of $H_1(T)$.  Let $H[\boundary_-P]$ and $H[\boundary_+P]$  be the solid
    tori gotten by attaching a $2$--handle to $H$ along the curves
    $\boundary_-P=\mu\times\{-1\}$ and $\boundary_+P=\lambda\times\{+1\}$
    respectively. Thinking of $H \cong T\times I$, we use meridian/longitude
    coordinates on $H[\boundary_-P]$ and $H[\boundary_+P]$ given by the pairs
    of oriented curves $(\mu\times\set{1}, \lambda\times\set{1})$ in the first
    case and $(\lambda\times\set{-1}, \mu\times\set{-1})$ in the second.  
    (One may care to refer to Figure~\ref{fig:xyzV2}(Right).) 
    In $H[\boundary_-P]$, $K(\tau(p,q))$ becomes a $(p,q)$ curve, with surface
    slope $\gamma$.  Therefore the result of attaching a $2$--handle to
    $H(\gamma)$ along $\boundary_-P$ is the connect sum of the lens space
    $L(p,q)$ with a solid torus (\cf~\cite[Proposition
    5.2]{Bowman13}).  Similarly, the result of attaching a $2$--handle to
    $H(\gamma)$ along $\boundary_+P$ is the connect sum of the lens space
    $L(q,p)$ with a solid torus.  It follows that $\boundary_-P$ is a $(p,q)$
    curve in $J_-$ and $\boundary_+P$ is a $(q,p)$ curve in $J_+$.

    When $M$ has Seifert type, the curves $\boundary_{\pm}P$ are primitive and
    isotopic in $\boundary H'$.  Therefore we may think of the space
    $M(\gamma)$ as the space obtained by identifying an annulus neighborhood of
    $\boundary_-P$ in $J_-$ with an annulus neighborhood of $\boundary_+P$ in
    $J_+$ and attaching $(g-1)$ $1$--handles.  This space is a 
    $D(p,q)$--Seifert space with $g-1$ attached $1$--handles, as claimed.
\end{proof}

We will need the following technical lemma later:

\begin{lemma}\label{no-annuli-on-P}
    Let $(H,P)$ be as in section~\ref{section:little-handlebody}, let $(H',P')$
    be of strong handlebody or strong Seifert type, and let $(\widehat{H},
    \widehat{P})$ denote either $(H,P)$ or $(H',P')$.  Let $A$ be an annulus
    properly embedded in $\widehat{H}$ with $\bdry A \subset \widehat{P}$.  If
    $\boundary A$ bounds an annulus $B\subseteq \widehat{P}$, then $A$ may be
    isotoped to lie entirely in $\widehat{P}$.  If $\boundary A$ does not bound
    an annulus in $\widehat{P}$, then $(\widehat{H},\widehat{P})=(H',P')$ of
    strong Seifert type, and $A$ may be isotoped to the annulus of $H' \cut P'$
    with boundary $a \cup b$ while keeping  $\boundary A \subset P'$.
\end{lemma}
\begin{proof}
    Let $A$ be an annulus properly embedded in $\widehat{H}$ with $\boundary
    A\subseteq \widehat{P}$.  First assume $\bdry A$ cobounds an annulus $B$ in
    $\widehat{P}$.  The surface $A\cup B$ is a torus $S$.  Isotop $S$ slightly
    inside $\widehat{H}$ to obtain an embedded torus.  There is an annulus $C$
    so that one boundary component of $C$ lies on $\widehat{P}$ and the other
    is a core of $B$.

    A torus in a handlebody compresses to one side.  If $S$ were compressible
    to the side containing $C$, then using an innermost disk/outermost arc
    argument we could find a compressing disk $D$ for $S$ not meeting $C$.
    Therefore we could isotop $D$ so that $\boundary D\subseteq \widehat{P}$,
    contradicting that $\widehat{P}$ is incompressible.  So $S$ must be
    compressible to the side not containing $C$.  From the irreducibility of
    $\widehat{H}$, we conclude that $S$ bounds a solid torus $N$ on this side.

    If the core of $B$ is not longitudinal in $N$, then we obtain a reducible
    manifold after attaching a $2$--handle to $\widehat{H}$ along this curve.
    When $\widehat{H}=H$, or when $\widehat{H}=H'$ and the attaching curve is
    parallel to one of the primitive components of $\boundary P'$, it is clear
    that this does not happen.  We wish to show that it does not happen when
    $\widehat{H}=H'$ and the attaching curve is parallel to component $c$ of
    $\boundary P'$.  Because $c'$ meets every meridian disk of $H'$, and hence
    its complement in $\boundary H'$ is incompressible, the Handle Addition
    Lemma \cite[Lemma 2.1.1]{CGLS} gives a contradiction.

    We have shown that the core of $B$ is longitudinal in $N$, and so we may
    isotop $A$ to $B$ through $S$ as claimed.

    Finally, assume the components of $\boundary A$ are not parallel in $\widehat{P}$. Then $A$
    cannot lie in $H$ since the boundary components of $P$ represent different
    homology classes in $H$. So $A$ is in $H'$. As there is a meridian disk of $H'$ disjoint
    from $A$, neither component of $\partial A$ can be isotopic to the
    curve $c$ of $\boundary P'$ (since this curve is disk-busting).
    Hence $A$ may be isotoped keeping
    $\bdry A \subset P'$ so that $\bdry A = a \cup b$.  Since $c$ intersects
    every essential annulus, $A$ must be $\bdry$--parallel.  Therefore
    $(H',P')$ must be of Seifert type and $A$ isotopic to the annulus component
    of $H' \cut P'$.
\end{proof}

Assume $M$ is of handlebody or Seifert type.  Under the embedding of $H$ in $M$
when $\tau=\tau(\kappa,\alpha,n)$ then the knot $K(\tau)$ in $M$ can be
considered, as in Definition~\ref{defn:twisting} of
section~\ref{section:little-handlebody}, as obtained from the knot $K(\kappa)$
by twisting $n$ times along the annulus $\Rhat(\alpha)$. As in the preceding
section, let $\partial \Rhat(\alpha)=\Lp \cup \Lm$ and
$\L(\kappa,\alpha)=K(\kappa) \cup \Lp \cup \Lm$ in $M$.

\begin{lemma}\label{X-irreducible}
    Assume $M$ is of strong handlebody or strong Seifert type.  Let
    $\tau=\tau(\kappa,\alpha,n)$. Let $\L=\L(\kappa,\alpha)$. The space $M_{\L}$
    is irreducible.
\end{lemma}
\begin{proof}
    Because $M$ is of strong handlebody or Seifert type,  $P$ is incompressible
    in $M$. So it suffices to show that $H_{\L}$ is irreducible.  Assume
    $S\subseteq H_{\L}$ is an embedded sphere not bounding a ball.  Let
    $A_+,A_{-},R$ be the restriction to $H_{\L}$ of $A_+(\alpha),
    A_{-}(\alpha), R(\alpha)$ of $H$ (Definition~\ref{def:Rhat(tau)}). If we
    choose $S$ so that $|S\cap (A_+\cup A_{-}\cup R)|$ is minimal, an innermost
    disk argument shows that we may take $S\cap (A_+\cup A_{-}\cup
    R)=\emptyset$.

    The result of cutting $H_{\L}$ along $A_+\cup A_{-}\cup R$ is homeomorphic
    to the space we get by cutting $H$ along the annulus $\alpha\times I$ and
    removing tubular neighborhoods of a collection of arcs $a_i$ whose ends
    meet to form $K$.  The annulus $\kappa\times [0, 1/2]\subseteq H_{\L}$
    becomes a collection of disks showing that each arc $a_i$ is unknotted, and
    so the resulting space is a handlebody.  The conclusion follows from the
    fact that handlebodies are irreducible.
\end{proof}

\begin{lemma}\label{lem:no-spanning-annulus1}
    Assume $M$ is of strong handlebody or strong Seifert type.  Let
    $\tau=\tau(\kappa,\alpha,n)$. Let $\L=\L(\kappa,\alpha)$.  Let $\Tp,\Tm$ be
    the components of $\bdry M_{\L}$ corresponding to $\Lp,\Lm$. There is no
    essential annulus in $M_{\L}$ with one boundary component on $\Tp$ and the
    other on $\Tm$.
\end{lemma}
\begin{proof}
    Assume $A$ is such an essential annulus.  First we show that we may assume
    that $A$ lies in $H_{\L}$.  Choose $A$ to minimize $|A\cap P|$.  Since $A$
    and $P$ are incompressible in $M$, there are no \sccs of intersection which
    are trivial in $A$ or $P$.  If there are essential \sccs of intersection,
    Lemma~\ref{no-annuli-on-P} shows that there are sub-annuli $A'_+, A'_{-}$
    of $A$, properly embedded in $H$,  that have one boundary component on
    $T_+,T_{-}$ (resp.) and the other isotopic to $\boundary_+P$ or to
    $\boundary_{-}P$, with $\boundary A'_+$ not parallel to $\boundary A'_{-}$
    on $P$. As $L_+,L_{-}$ are isotopic in $H$, this says that a power of
    $\boundary_+P$ is equal to a power of $\boundary_-P$ as homotopy classes in
    $H$. But $\boundary_+P, \boundary_-P$ generate the fundamental group of
    $H$, a contradiction.   Hence $A \cap P = \emptyset$.

    Thus we take $A$ properly embedded in $H_{\L}$. Let $A_+, A_-$ be the
    restriction of $A_+(\alpha), A_{-}(\alpha)$ to $H_\L$.  Isotop $\partial A$
    and $A_+ \cup A_{-}$ to intersect minimally on $T_+ \cup T_{-}$.  There can
    be no arcs of intersection of $A\cap (A_+\cup A_{-})$ as the signs of
    intersections are consistent along each boundary component of $H_{\L}$.
    (In particular, $\bdry A$ must be disjoint from $A_+ \cup A_-$.)

    Suppose then that $A$ meets $A_+ \cup A_{-}$ minimally.  We may choose
    a \scc of intersection $\gamma$ which is outermost in $A_+ \cup A_{-}$ in
    the sense that it bounds an annulus $A' \subset A_+ \cup A_{-}$ with one
    boundary component on $T_+ \cup T_{-}$ so that $\interior A'\cap
    A =\emptyset$.  By cutting and pasting along $\gamma$ we get a new annulus
    with one boundary component on $T_i$ and one on $T_j$ which meets $A_+\cup
    A_{-}$ fewer times.  This is impossible by minimality of $|A\cap (A_+\cup
    A_{-})|$, so we may assume that $A\cap (A_+\cup A_{-})=\emptyset$.

    The surface $A_+\cup A\cup A_{-}$ in $H_\L$ then extends to an annulus $B$
    in $H$ with $\boundary B = \alpha \times \{-1, +1\}$ and  $B\cap
    K(\kappa)=\emptyset$.  Then, carrying along $K(\kappa)$, we may isotop $B$
    in $H$ to be vertical with respect to the product structure $T \times I$ of
    $H$.   The image of $B$ under the projection $H=T \times I \to T$ is the
    curve $\alpha$, and the image of $K(\kappa)$ is homotopic to $\kappa$.
    Since $B$ is disjoint from $K(\kappa)$, the algebraic intersection number
    of $\alpha$ and $\kappa$ in $T$ should be zero.   But by definition,
    $\alpha$ and $\kappa$ are not isotopic in $T$, a contradiction.
\end{proof}

\begin{prop}\label{prop:exterior-hyperbolic}
    Assume $M$ is of strong handlebody or strong Seifert type.
    Let $\tau$ be an essential curve in $T$ and consider $K(\tau)$ in $M$.
    \begin{itemize}
        \item The space $M_{K(\tau)}$ is irreducible and atoroidal.

        \item When $\tau \not \in \{\lambda, \mu, \nu\}$,
            $M_{K(\tau)}$ is $\boundary$--irreducible.

        \item When $\tau \not \in \{\lambda, \mu, \nu, \lambda-\mu,
            \lambda+\nu, \mu+\nu\}$,
            there is no essential annulus in $M_{K(\tau)}$.
    \end{itemize}
\end{prop}

\begin{proof}
    The fact that $M_{K(\tau)}$ is irreducible is implied by the irreducibility
    of $H_{K(\tau)}$ because $P$ is incompressible in $H$ and $H'$.  The space
    $H_{K(\tau)}$ is irreducible since there are meridian disks of $H$ that
    have nonzero algebraic intersection with $K(\tau)$ by our hypothesis that
    $\tau$ is an essential curve in $T$.

    \smallskip
    Suppose that $S\subseteq M_{K(\tau)}$ is an essential torus chosen to
    minimize $\abs{S\cap P}$.  If $\abs{S\cap P} >0$, these \sccs must be
    essential in both $S$ and $P$ (because both $S$ and $P$ are
    incompressible).   By Lemma~\ref{no-annuli-on-P} then, each component of $S
    \cap H$ is an annulus parallel into $P$. By the same Lemma, and the
    minimality of intersection, we conclude that $(H',P')$ has Seifert type and
    each component of $S\cap H'$ is an annulus whose boundary is a pair of
    curves that are not parallel in $P$.  Thus there must be annular components
    of $S \cap H$ parallel into disjoint annuli in $P$. One of these
    parallelisms must be disjoint from $K(\tau)$ --- contradicting the
    minimality of $|S \cap P|$.   Thus $S \cap P = \emptyset$.  Since the
    handlebody $H'$ and compression body $H_{K(\tau)}$ are atoroidal, no such
    essential torus can exist.

    \smallskip
    Assume $\tau \not \in \{\lambda, \mu, \nu\}$.
    Suppose that $D$ is a $\boundary$--reducing disk for $M_{K(\tau)}$ chosen
    to minimize $\abs{D\cap P}$.  An arc of intersection $D\cap P$, outermost
    in $D$, cuts off a subdisk $D'\subseteq D$ which is
    a $\boundary$--compressing disk for $P\subseteq M$.  There are no such
    disks in $H'$ since $(H',P')$ is of strong handlebody or Seifert type. By
    Lemma~\ref{lem:onearc}, if there is such disk in $H$ then $\tau = \lambda$,
    $\mu$, or $\nu$ contrary to hypothesis.  Hence $D\cap P=\emptyset$. But $P,
    \boundary H \cut P$  are incompressible in $H$ by Lemma~\ref{lem:onearc},
    and $P',  \boundary H' \cut P'$ are incompressible in $H'$ since $(H',P')$
    is a strong type.

    \smallskip
    Finally, assume $\tau \not \in \{\lambda, \mu, \nu, \lambda-\mu,
    \lambda+\nu, \mu+\nu\}$.  Suppose that $A$ is an essential annulus properly
    embedded in $M_{K(\tau)}$, chosen to minimize $\abs{A\cap P}$.  Because
    $M_{K(\tau)}$ contains no essential tori, we cannot have $\boundary
    A\subseteq \boundary N(K(\tau))$.  So suppose that one component of
    $\boundary A$ lies in $\boundary N(K(\tau))$ and the other lies in
    $\boundary M$.  By the minimality of $\abs{A\cap P}$, we may assume that
    there are no arcs of intersection.  Therefore there is a subannulus
    $A'\subseteq A$ with one boundary component on $\boundary N(K(\tau))$ and
    the other on $\boundary H$, either in $P$ or $\bdry H \cut P$.  This
    annulus shows that a power of $K(\tau)$ is homotopic to a component of
    $\boundary P$.  Since both $K(\tau)$ and each component of $\boundary P$ is
    primitive, no component of $\boundary P$ is homotopic to a proper power of
    $K(\tau)$.  On the other hand, $K(\tau)$ is not homotopic to a component of
    $\boundary P$ since $\tau \neq \lambda$, $\mu$ or $\nu$.

    Suppose then that $\boundary A\subseteq M$. By minimality of $\abs{A\cap
    P}$, there can be no arcs of intersection which are trivial in either $A$
    or $P$ (Lemma~\ref{lem:onearc} and $\partial$--irreducibility).  If there
    are nontrivial arcs of intersection, we may choose two which bound a disk
    in $A$ whose interior lies in $H$.  With our restriction on $\tau$, this
    contradicts Lemma~\ref{lem:min-D-cap-P}.  If there are \sccs of
    intersection, there is a subannulus $A'\subseteq A$ lying entirely in $H$.
    We may isotop $A'$ so that both components of $\boundary A'$ lie in $P$. By
    Lemma~\ref{no-annuli-on-P} and the restriction on $\tau$, this annulus must
    be boundary parallel into $P$.  Hence $A$ may be isotoped to further reduce
    $\abs{A\cap P}$ contrary to the minimality.  Thus $A$ is disjoint from $P$
    and lies entirely in $H$ or $H'$.  Again, Lemma~\ref{no-annuli-on-P} (and
    the restriction on $\tau$ in case $A \subset H$) or the fact that $c$ is
    annulus-busting ($P'$ is of strong type) implies that $A$ is not essential.
\end{proof}

\section{Essential surfaces}\label{sec:essentialsurfaces}

In section~\ref{sec:wrapping}, we will study a family of knots $K_i$ in
a handlebody $M$ obtained by twisting a knot $K=K_0$ along an annulus $\Uhat$
with one boundary component $\gamma$ in $\bdry M$ and the other a knot $L$ in
the interior of $M$.   We want to show that, generically, these knots are
distinct and hyperbolic.   The difference in knot type will come from showing
that the minimum number of times $K_i$ intersects a meridian disk of $M$ (the
``disk hitting number'') increases with $|i|$,
Proposition~\ref{prop:wrappingno}(4).  The hyperbolicity will come from showing
that the exteriors of these knots are irreducible, atoroidal,
$\bdry$--irreducible, and anannular, Proposition~\ref{prop:strongandhyp} by way
of Lemma~\ref{lem:no-spanning-annulus2} and
Proposition~\ref{prop:wrappingno}(5).

In this section, we establish Theorem~\ref{thm:essentialsurfaces} below which
is applied in Proposition~\ref{prop:wrappingno} to show that for large twisting
numbers $|i|$, $K_i$ must intersect any meridian disk or essential annulus of
$M$ many times. In application, we find a {\em catching surface} $Q$ and use it
to generate a lower bound on the intersection number of $K_i$ with an essential
surface (such as a meridian disk or essential annulus) in $M$ which restricts
to a surface $F$ properly embedded in the exterior of $K_i$ and $L$.  As the
pair $(M,K_i)$ is homeomorphic to the pair $(M_L(-1/i), K)$, where $M_L(-1/i)$
is the manifold obtained by doing $-1/i$--surgery on $L$ in $M$, we may view
$F$ as a surface in the exterior of $K \cup L$ in $M$ whose boundary components
have framing $-1/i$ on $L$. The catching surface will be chosen so that its
boundary has framing on $L$ whose intersection number on $\partial \nbhd(L)$
with $\partial F$ increases with $|i|$   (that is, $\Delta_L$ below increases
with $|i|$).

Theorem~\ref{thm:essentialsurfaces} however is more general. In particular, it
does not require that we be in the context of twisting along an annulus
$\Uhat$. The argument mimics that of Lemma~2.8 of \cite{BGL}.  It shows that,
for an essential surface $F$ and a catching surface $Q$ in $M_{K \cup L}$, if
the boundary slopes of the surfaces $F$ and $Q$ along the knot $L$  intersect
more than a certain measure of complexity of $F$ and $Q$ then there must be
certain kinds of annuli or \mobius bands in $M_{K \cup L}$.

\medskip
The remainder of this section is devoted to the proof of
Theorem~\ref{thm:essentialsurfaces}.  We first establish the necessary
notation.

Let $M$ be a compact, connected oriented $3$--manifold with knots $K$ and $L$.
Let $X$ be the exterior of $K$ and $L$ in $M$. Let $T_K, T_L$ be the torus
components of $\bdry X$ corresponding to $K,L$.  Note that the rest of the
boundary of $X$ is the boundary of $M$, i.e.\  $\bdry X - (T_K \cup T_L) =\bdry
M$.

Assume $F$ is a compact, connected, orientable surface properly embedded in
$X$.  For $j=K,L$ let $f_j = |\bdry F \cap T_j|$  and let $\alpha_j$ be the
slope of the curves $\bdry F \cap T_j$ in $T_j$. Also let $f_M = |\bdry F \cap
\bdry M|$. We assume throughout this section that $f_L \geq 1$. Define $f'_K
= \max(f_K,1)$ and $f'_M = \max(f_M,1)$.

\begin{definition}[Catching surface and associated distances]\label{def:catching1}
    Let $Q$ be an oriented surface with $\bdry Q \cap T_L$ being a non-empty
    set of coherently oriented curves in $T_L$ (with orientation from $Q$).  We
    also let $\beta_j$ denote the slope of the curves $\bdry Q \cap T_j$ for
    $j=K,L$.  Then we define $\Delta_j$ to be $\geomint{\alpha_j}{\beta_j}$ on
    $T_j$ for $j=K,L$ (see Definition~\ref{def:distance}), where $\Delta_K=0$
    when $f_K=0$.   Also set $\Delta'_K = \max(\Delta_K, 1)$.   Similarly
    define $\Delta_M$ to be the maximum geometric intersection number between
    a component of $\bdry Q \cap \bdry M$ and a component of $\bdry F \cap
    \bdry M$; we also set $\Delta_M=0$ when $f_M=0$.  If $\Delta_L>0$ then $Q$
    is called a {\em catching surface for $(F,K,L)$} .
\end{definition}

\begin{theorem}\label{thm:essentialsurfaces}
    Let $K$ and $L$ be knots in $M$ and $X$ be their exterior, with $T_L$ the
    component of $\partial X$ corresponding to $L$.  Let $F$ be a compact,
    connected, orientable, properly embedded surface in $X$ which is boundary
    incompressible at $T_L$ (and $|F \cap T_L| \neq \emptyset$), and  let $Q$
    be a catching surface for $(F,K,L)$ (Definition~\ref{def:catching1}). If \[
    \Delta_L > 6f'_M\max(-6\chi(Q),2) (f'_K \Delta'_K
    + f_M-\chi(\widehat{F})+2)\] then one of the following is true:
    \begin{enumerate}
        \item There exists a Mobius band properly embedded in $X$ with boundary
            isotopic to $\alpha_L$ (the slope of $\partial F$)  in $T_L$.

        \item There exists an annulus properly embedded in $X$ whose boundary
            is an essential curve in $T_L$ and a curve in $\bdry M$.
    \end{enumerate}
\end{theorem}

\begin{remark}
    Note that if $F$ is incompressible, then it is automatically $\boundary$--incompressible 
    at $L$ unless $F$ is an annulus parallel into $T_L$.
\end{remark}

\begin{remark}
    Theorem~\ref{thm:essentialsurfaces} holds more generally when $K$ is a link
    of multiple components.  Just let $T_K$ be the union of the corresponding
    torus components of $\bdry X$ and define $\Delta_K$ to be the maximum of
    $\geomint{\alpha}{\beta}$ for slopes $\alpha$  of $\bdry F$ and $\beta$ of
    $\bdry Q$ among components of $T_K$, and then $\Delta'_K = \max(\Delta_K,
    1)$.  The proof below carries through unchanged.
\end{remark}

\begin{proof}
    Let $F$ and $Q$ be as above, and take $j\in\{K,L\}$. Fixing an orientation
    on $F$ ($Q$) , we call two components of $\partial F \cap T_j$ ($\partial
    Q \cap T_j$, resp.) {\em parallel} if they inherit coherent orientations on
    $T_j$. They are {\em anti-parallel} otherwise.  For example, by definition,
    all components of $\partial Q \cap T_L$ are parallel. Label the components
    of $\partial F \cap T_j$ from $1$ to $|\partial F \cap T_j|$ in sequence
    along $T_j$.  Similarly label the components of $\partial Q \cap T_j$.
    Label the components of $\partial F \cap \partial M$ ($\partial Q \cap
    \partial M$) arbitrarily. Assume that $F$ and $Q$ have been isotoped to
    intersect transversally and minimally.  Abstractly cap off the boundary
    components of these surfaces with disks to form the respective closed
    surfaces $\widehat{F}$ and $\widehat{Q}$.  (Note that $\chi(\widehat{F})
    = \chi(F) + f_K + f_L + f_M$.)   Regarding these capping disks as fat
    vertices and the arcs of the intersection $F \cap Q$ as edges we create the
    {\em fat-vertexed graphs} $G_F$ and $G_Q$ in the respective closed surfaces
    $\widehat{F}$ and $\widehat{Q}$.  {\em Label} the endpoint of an edge in
    one graph with the vertex of the other graph whose boundary contains the
    endpoint.

    Figure~\ref{fig:intersection} gives an example of how the labelled graphs
    $G_Q,G_F$ arise.

    \begin{figure}
        \centering
        \includegraphics[width=0.5\textwidth]{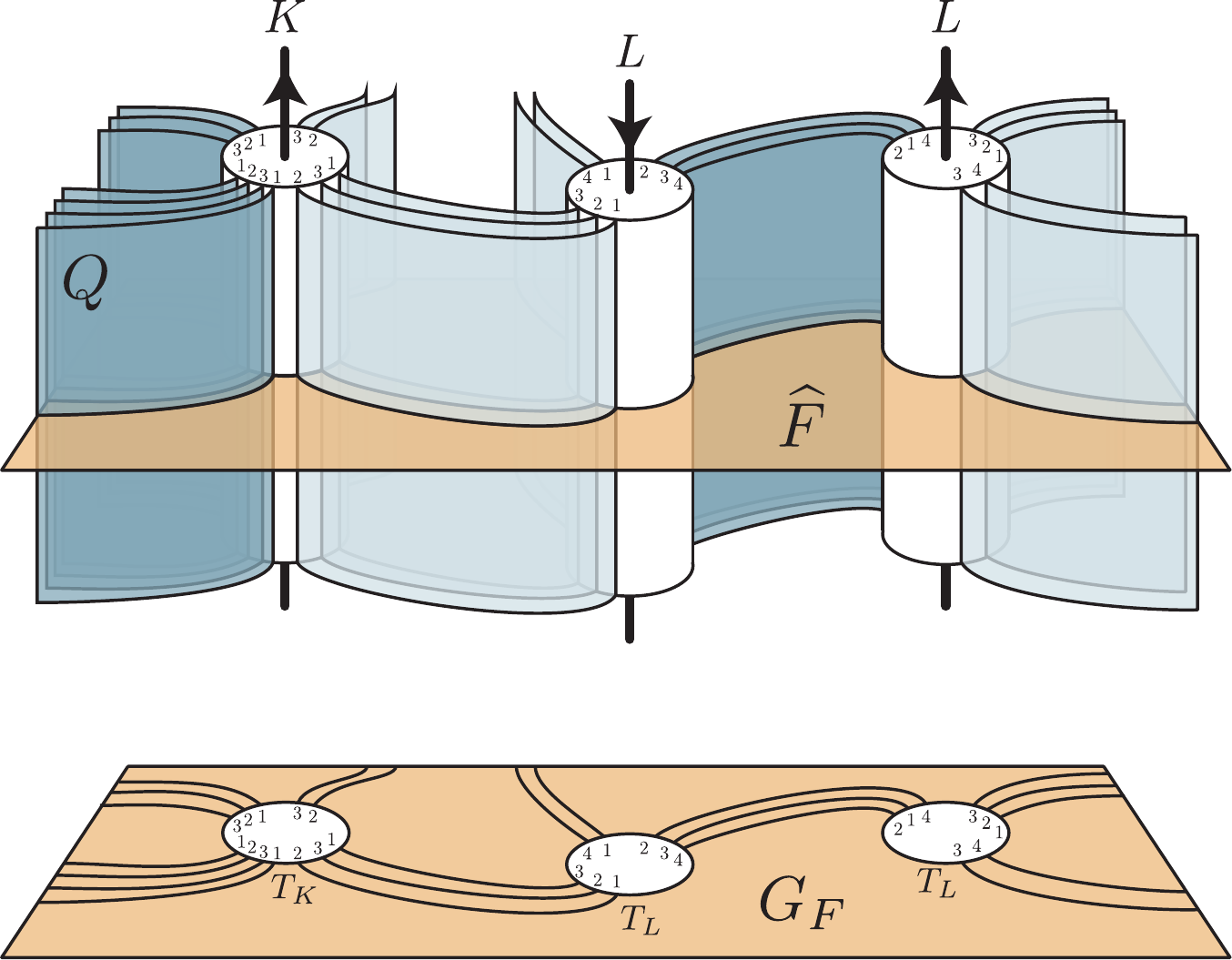}
        \caption{An example of a portion of the labelled graph $G_F$ arising
            from the intersection of $Q$ with $F$ and, say, $|\bdry Q \cap
            T_K|=3$  with $\Delta_K=4$ and $|\bdry Q \cap T_L|=4$ with
            $\Delta_L=2$.  Labels of $G_F$ are given by the  corresponding
            component of $\partial Q \cap \bdry X$.}
        \label{fig:intersection}
    \end{figure}

    Two vertices on $G_F$ ($G_Q$) are {\em parallel} or {\em anti-parallel} if
    the corresponding boundary components of $F$ ($Q$, resp.) are parallel or
    anti-parallel (in particular, for vertices to be parallel or anti-parallel
    they must corresond to boundary components which are both on $T_K$ or
    both on $T_L$).  The orientability of $F$, $Q$, and $X$ give the

    \noindent {\em Parity Rule:} An edge connecting parallel vertices on one
    graph must connect anti-parallel vertices on the other graph.

    Observe that a vertex of $G_Q$ corresponding to $T_j$ has valence $f_j
    \Delta_j$.   Let $V_j$ be the set of vertices in $G_Q$ that corresponds to
    $T_j$ for $j=K,L$, and let $V_M$ be the set of vertices corresponding to
    $\bdry M$.  Recall that because $Q$ is catching, $V_L$ is non-empty. Since
    $F$ is $\bdry$--incompressible along $T_L$, the Parity Rule guarantees that
    $G_Q$ contains no monogons (i.e.\ $1$--sided faces) at any vertex of $V_L$.

    Assume
    \[\frac{ \Delta_L}{ \max(-6\chi(Q),2)} > 6 f'_M
      (f'_K \Delta'_K + f_M-\chi(\widehat{F}) +2) > 0 \tag{$*$}\]
    and note that the second inequality does indeed hold true since
    $f'_K \Delta'_K \geq 1$.

    An edge of a graph is called {\em trivial} if it bounds a monogon face. Two
    edges are {\em parallel} if there is a sequence of bigon faces between
    them.

    \begin{claim}[cf.\  Claim 2.9 \cite{BGL}]\label{claim:parallel}
        The graph $G_Q$ contains at most $\max(-3 \chi(Q),1)$ parallelism
        classes of non-trivial edges. \qed
    \end{claim}

    \begin{proof}
        Pick one edge from each parallelism class of non-trivial edges in $G_Q$
        and consider it as an arc properly embedded in the surface $Q$. Let $E$
        be the collection of all such arcs. Then $|E|$ is the number of
        parallelism classes in $G_Q$.  First note that $|E| \leq 1$ when
        $\chi(Q) \geq 0$.  So we assume $\chi(Q)<0$.  Because the edges in $E$
        are not boundary parallel in $Q$, they can be completed to an ideal
        triangulation of (the interior of) $Q$ by adding more arcs between the
        components of $\bdry Q$ as needed.  If $E'$ is this resulting number of
        edges and $F$ is the number of ideal triangles, then we have both $3F
        = 2E'$ and $\chi(Q) = -E'+F$.  Thus $E \leq E'=-3\chi(Q)$.
    \end{proof}

    Since any vertex of $V_L$ has valence $f_L \Delta_L$ (and there are no
    monogons of $G_Q$ at vertices of $V_L$), Claim~\ref{claim:parallel} shows
    that there exists a set $\mathcal{E}$ of at least $\ceil{\frac{(f_L
    \Delta_L/2)} { \max(-3\chi(Q),1)}}=\ceil{ \frac{f_L \Delta_L }{
    \max(-6\chi(Q),2) }}$ mutually parallel edges in $G_Q$ with an end point at
    a vertex of $V_L$, where $\ceil{x}$ is the smallest integer not less than
    $x$.   There are three cases for the other end point of these edges: either
    it is at a vertex of  (a) $V_L$ (perhaps the same vertex); (b) $V_M$; or
    (c) $V_K$.   By the assumed inequality ($*$) above, $|\mathcal{E}| \geq
    \frac{f_L \Delta_L }{ \max(-6\chi(Q),2) }> f_K \Delta_K$, so there are more
    edges in $\mathcal{E}$ than the valence of any vertex in $V_K$; hence the
    other end point of these edges cannot be in $V_K$ and $(c)$ does not arise.
    Also note  that the inequality $(*)$ implies that $\mathcal{E}$ contains at
    least two edges, more than $f_L$ edges, and more than $f_M$ edges.

    Let $\widehat{F}_{L}$ be the subsurface of $\widehat{F}$ obtained by
    capping off the boundary components of $F$  in $T_L$ and
    $\widehat{F}_{L,M}$ be obtained by further capping off any boundary
    components of $\widehat{F}_{L}$ in $\bdry M$.

    For case (a), form the subgraph $G_{F_L}(\mathcal{E})$ of $G_F$ in the
    surface $\widehat{F}_{L}$ that consists of the edges $\mathcal{E}$ and all
    the $f_L$ vertices of $\bdry F \cap T_L$.

    For case (b),  form the subgraph $G_{F_{L,M}}(\mathcal{E})$ of $G_F$ in the
    surface $\widehat{F}_{L,M}$  that consists of the edges $\mathcal{E}$ and
    all the $f_L$ vertices of $\bdry F \cap T_L$ and all $f_M$ vertices of
    $\bdry F \cap \bdry M$.

    Now we have
    \begin{align*}
        |\mathcal{E}| &\geq \frac{f_L \Delta_L }{ \max(-6\chi(Q),2) }\\
                      &\geq6f_Lf_M'(f'_K \Delta'_K + f_M -\chi(\widehat{F})+2)\\
                      &\geq 6f_Lf_M'( f'_K +f_M - \chi(\widehat{F})+1) + 6f_Lf_M'
    \end{align*}

    In particular, we have for case (a)
    \[|\mathcal{E}| \geq 3f_L[\max(1-\chi(\widehat{F}_{L}),0)] + 3f_L \tag{$**$}\]

    Using that $2f_Lf_M'=2f_Lf_M \geq (f_L+f_M)$ in case (b), we get
    \[|\mathcal{E}| \geq 3(f_L+f_M)(f_M+1-\chi(\widehat{F}_{L,M})) +3(f_L+f_M) \tag{$***$}\]

    Observe that neither  $G_{F_L}(\mathcal{E})$ nor $G_{F_{L,M}}(\mathcal{E})$
    has any monogons.   In the former, case (a), this is because the vertices
    of $V_L$ are parallel and the Parity Rule.  In the latter, case (b),  this
    is because the edges connect vertices corresponding to different components
    of $\bdry X$.

    \begin{claim}[cf.\  Claim~2.10 \cite{BGL}]\label{clm:parallelP}
        Let $G$ be a graph in a surface $S$. Let $V$ be the number of vertices
        and $E$ the number of edges of $G$, and let $\chi(S)$ be the Euler
        characteristic of $S$. If $G$ has no monogons and $E>3V
        \max(1-\chi(S),1)$, then $G$ has parallel edges.

        Further assume either $\chi(S)>0$ or $\bdry S \neq \emptyset$. If $G$
        has no monogons and $E\geq3V \max(1-\chi(S),1)$, then $G$ has parallel
        edges. That is, equality also guarantees parallel edges in these cases.
    \end{claim}
    \begin{proof}

        Assume there are no parallel edges in $G$, and let $V,E$ be the number
        of vertices, edges.  Assume $E>3V \max(1-\chi(S),1)$.  Then we may add
        edges to $G$ so that all faces are either $m$--gons with $m\geq 3$ or
        annuli with one boundary component being a component of $\bdry S$ and
        the other consisting of a single edge and vertex of $G$.
        Since this doesn't change the number of vertices and only increases the
        number of edges, we still have $E>3V \max(1-\chi(S),1)$.  Now
        $\chi(S)=V-E+F$ where $F$ is the number of disk faces of $G$.  Because
        every edge of $G$ is on the boundary of the faces (including the
        annular faces) twice, $2E \geq 3F + |\bdry S|$.  Let $C=3
        \max(1-\chi(S),1) \geq 3$.  Our assumption that $E > CV $ shows
        both that $V <E/C$ and $E>C$.

        Therefore $\chi(S) = V-E+F < E/C - E + 2E/3 - |\bdry S|/3$.  Hence  $C
        \chi(S) < E (1-C/3)-|\partial S|C/3$. Since $C \geq 3$, this implies
        that $\chi(S)<0$. Then $C=3(1-\chi(S))$. So $C \chi(S)
        < E (\chi(S))-|\partial S|(1-\chi(S))$. Consequently, $C
        > E + (1-1/\chi(S)) |\bdry S| \geq E$.  This contradicts that $C<E$.

        Now assume either $\chi(S)>0$ or $\bdry S \neq \emptyset$.  Change the
        strict inequalities $<$ and $>$ above to $\leq$ and $\geq$.  Then $C
        \chi(S) \leq E (1-C/3)-|\partial S|C/3$ implies that $\chi(S) \leq 0$,
        hence $|\bdry S| >0$. A second application of the preceding inequality
        says that in fact $\chi(S)<0$. As above we conclude that  $C \geq
        E + (1-1/\chi(S)) |\bdry S| > E$, contradicting that $C \leq E$.
    \end{proof}

    Applying Claim~\ref{clm:parallelP},  $G_{F_L}(\mathcal{E})$ in case (a) has
    parallel edges (with $S=\widehat{F}_L, V=f_L, E=|\mathcal{E}|$, using $(**)$),
    and $G_{F_{L,M}}(\mathcal{E})$ in case (b) has parallel edges (with
        $S=\widehat{F}_{L,M} ,V=f_L+f_M, E=|\mathcal{E}|$, using $(***)$ and
    $f_M>0$). That is, there exist edges $e,e' \in \mathcal{E}$ bounding
    rectangles $D_Q$ in $G_Q$ and $D_F$ in  $G_{F_L}(\mathcal{E})$ or
    $G_{F_{L,M}}(\mathcal{E})$ such that $D_Q \cap D_F= \{e \cup e'\}$ (after
        possibly surgering away simple closed curves of intersection in the
    interiors of these disks).

    \begin{remark}
        Note that Claim~\ref{clm:parallelP} allows for vertices with valence
        $0$, which may occur in its application to $G_{F_{L,M}}$, at vertices
        corresponding to $\partial M$.
    \end{remark}

    In case (a), $D_Q \cup D_F$ is a \mobius band in $X$ with boundary on $T_L$
    of slope $\alpha_L$. This follows from the proof of Lemma~2.1 of
    \cite{Gordon98}. Following that proof, the boundary of the \mobius band has
    the slope of $\partial F$ since the rectangle $D_F$ connects anti-parallel
    vertices in $F$  ($D_Q$ connects parallel vertices in $Q$, hence the Parity
    Rule guarantees that $D_F$ connects anti-parallel vertices). This is
    conclusion $(1)$ of the Theorem.

    In case (b), $D_Q \cup D_F$ is an annulus in $X$ with a boundary component
    on each of $T_L$ and $\partial M$.   The boundary component of this annulus
    must intersect a component of $\bdry F$ and of $\partial Q$ algebraically
    a non-zero number of times on $T_L$. Thus the boundary component of this
    annulus is essential on $T_L$ (and isotopic to neither a component of
    $\bdry F$ nor $\bdry Q$). This is conclusion $(2)$ of the Theorem.
\end{proof}

\section{Constructing $M$ of strong handlebody or Seifert type}\label{sec:wrapping}

\subsection{Disk-busting and annulus-busting curves}\label{sec:busting}
We show how to construct disk-busting and annulus-busting curves in the
boundary of a handlebody.

\begin{definition}\label{def:diskbusting}
    Let $M$ be an orientable 3-manifold and $\gamma$ a collection of simple
    closed curves in $\partial M$.  We say that $\gamma$ is
    \defn{$n$--disk-busting} for a positive integer $n$ if any properly
    embedded disk in $M$ that intersects $\gamma$ fewer than $n$ times is
    boundary parallel.  We simply say $\gamma$ is \defn{disk-busting} if it is
    $1$--disk-busting.  We say $\gamma$ is \defn{annulus-busting} in $M$ if it
    is disk-busting and if any annulus properly embedded in $M$ and disjoint
    from $\gamma$ is compressible or boundary parallel in $M$.
\end{definition}

In the discussion below, we then focus on essential disks and annuli in
$M$.

\begin{definition}\label{def:essential}
    A disk properly embedded in a $3$--manifold $M$ is called {\it essential} if
    it is not parallel into the boundary of $M$. An annulus properly embedded in
    $M$ is said to be {\it essential} if it is incompressible and not parallel
    into the boundary of $M$.
\end{definition}

\begin{lemma}\label{lem:1-disk-busting}
    Let $\gamma$ be a collection of curves in a component of $\partial M$ of
    genus $g$.  If $\gamma$ is $1$--disk-busting, then either $g=1$ or $\gamma$
    is $2$--disk-busting.
\end{lemma}

\begin{proof}
Assume there is a disk $D$ in $M$ that $\gamma$ intersects once. Then using
a component of $\gamma$ to band two copies of $D$ to itself gives an essential
disk disjoint from $\gamma$. But $\gamma$ is disk-busting.
\end{proof}

\begin{definition}
Recall that a {\em pair of pants} is a surface homeomorphic to a $2$--sphere
minus three disjoint open disks.   For a pair of pants $P$, a \defn{seam} is an
essential properly embedded arc in $P$ with endpoints on distinct components of
$\bdry P$.  Observe that a pair of pants $P$ has three isotopy classes of seams
and they have mutually disjoint representatives.
\end{definition}

\begin{definition}
Let $\Sigma$ be a connected, closed, compact surface of genus at least two.
A \defn{pants decomposition} of $\Sigma$ is a collection $\Pcal$ of \sccs in
the surface such that the complement of the curves in the surface is
a collection of (interiors of) pairs of pants.   
\begin{itemize}
\item 
Given a pants decomposition
$\Pcal$ of $\Sigma$, a collection of curves  $\gamma$ embedded in $\Sigma$ is
called \defn{$k$--seamed with respect to $\Pcal$} for an integer $k \geq 0$ if,
for each pair of pants $P \in \Sigma - \Pcal$, the intersection $\gamma \cap P$
is a collection of seams with at least $k$ members in each isotopy class.
\item 
For a handlebody $H$, a pants decomposition $\Pcal$ of $\boundary H$ is
\defn{compatible with $H$} if every component of $\Pcal$ bounds a disk in $H$.
\end{itemize}
\end{definition}

\begin{lemma}[Cf.\ Lemma~2.10, \cite{Yoshizawa}] \label{lem:seamedisbusting}

    Let $H$ be a handlebody and let $\Pcal$ be a pants decomposition of $\bdry
    H$ that is compatible with $H$.  If a collection of curves $\gamma$ in
    $\bdry H$ is $k$--seamed with respect to $\Pcal$, then $\gamma$ is
    $k$--disk-busting in $H$.

\end{lemma}

\begin{proof}
    Let $D$ be an essential disk of $H$ and isotope it so that $|\partial
    D \cap \gamma|$ is minimal. Under this restriction, isotope $D$ to
    intersect $\Pcal$ minimally. If $D$ is disjoint from $\Pcal$, then clearly
    the result holds. Let $\mathcal{D}(\Pcal)$ be the collection of disks in
    $H$ bounded by $\Pcal$. We may assume that there are no simple closed
    curves of intersection between $D$ and $\mathcal{D}(\Pcal)$.  An arc of
    intersection $c$ that is outermost on $D$ cuts off a disk $\Delta$ whose
    boundary is $c \cup e$ where $e \subset \partial D$ is properly embedded in
    a pair pants coming from $\Pcal$. The arc $e$ cannot be parallel into
    $\Pcal $, else we could isotope $D$ to reduce its intersection with
    $\gamma$  or $\Pcal$. Thus $e$ is an essential separating arc in this pair
    of pants and hence must intersect at least  $k$ seams.
\end{proof}

\bigskip

\begin{definition}\label{def:plumbing}
    Let $\gamma_i$ be a collection of simple closed curves in the boundary of
    the 3-manifold $M_i$, $i \in \{1,2\}$.  Let $b_i$ be a band in $\partial
    M_i$ running from $\gamma_i$ to itself. That is, $b_i=[0,1] \times [0,1]
    $ is embedded in $\partial M_i$ such that $\gamma_i \cap b_i = \{0,1\}
    \times [0,1]$. The \defn{boundary plumbing} of the pairs $(M_1, \gamma_1)$
    and $(M_2, \gamma_2)$ along $b_1,b_2$ is the pair $(M,\gamma)$ where

    \begin{itemize}
        \item the 3-manifold $M$ is the gluing of $M_1$ and $M_2$ via
            a homeomorphism $h:b_1 \to b_2$ that identifies $\{0,1\} \times
            [0,1]$ of $b_1$ with $[0,1] \times \{0,1\}$ of $b_2$; and

        \item the collection of curves $\gamma$ in $\partial M$ is the closure
            of  $(\gamma_1 - b_1)  \cup  (\gamma_2 - b_2)$.
    \end{itemize}

    In the plumbing construction, we refer to $b_i$ as a \defn{plumbing band},
    and say it is \defn{non-trivial} if the core of the band, $[0,1] \times
    \{1/2\}$, is not isotopic rel boundary into $\gamma_i$.

    Note that $b_1=b_2$ becomes a properly embedded disk, $D$, in $M$ which
    intersects $\gamma$ four times.  We refer to $D$ as the \defn{decomposing
    disk} for the boundary plumbing $(M,\gamma)$.
\end{definition}

\bigskip

\begin{lemma}\label{lem:3diskbusting}
    Let $\gamma_i$ be a collection of essential simple closed curves in
    $\partial M_i$  which are $3$--disk-busting in $M_i$, $i \in \{1,2\}$.  Let
    $(M,\gamma)$ be  a boundary plumbing of $(M_1,\gamma_1)$ with
    $(M_2,\gamma_2)$ where the plumbing bands are non-trivial. Then

    \begin{enumerate}
        \item No component of $\gamma$ is trivial and $\gamma$ is
            $3$--disk-busting in $M$.

        \item If $A$ is a properly embedded annulus in $M$ disjoint from
            $\gamma$ and which is neither compressible nor boundary parallel in
            $M$, then $A$ can be isotoped in $M- \gamma$ so that it is disjoint
            from $\gamma \cup D$ where $D$ is the decomposing disk of the
            boundary plumbing.

        \item If each $\gamma_i$ is annulus-busting in $M_i$, then $\gamma$ is
            annulus-busting in $M$.
    \end{enumerate}
\end{lemma}

\begin{proof}
    Let $D$ be the decomposing disk of the boundary plumbing. To see that no
    component of $\gamma$ is trivial we note that any such would have to
    intersect $D$. Let $E$ be the disk bounded by an innermost trivial
    component. Then an arc of $\partial D \cap E$ that is outermost on $D$
    shows that either a component of $\gamma_i$ or the plumbing band of
    $\gamma_i$ is trivial in $\partial M_i$.

    Note that the assumption that no component of $\gamma_i$ is trivial implies
    that the components of $\partial M_i$ in which they lie cannot be
    $2$-spheres.

    First we prove that $\gamma$ is $3$--disk-busting. Assume not, that there
    is  a essential disk $E$ in $M$ that intersects $\gamma$ minimally and at
    most twice. We take $E$ to intersect $D$ minimally among such disks. Then
    it has no simple closed curves of intersection with $D$ (an innermost such
    on $E$ must be boundary parallel in some $M_i$) and any arc of intersection
    must separate the four points $\gamma \cap D$ in $\partial D$. It cannot be
    disjoint from $D$ else, since $\gamma_i$ is 3-disk-busting, $\partial E$
    and $\partial D$ would have to cobound an annulus in some $\partial M_i$,
    contradicting that the band of the plumbing is non-trivial. Consider an arc
    $\tau$ of $E \cap D$ that is outermost in $E$ and let $\delta$ be the
    outermost disk that it bounds. We may assume that $\delta$ lies in $M_1$.
    Then $\delta$ cannot be disjoint from $\gamma$, else surgering $D$ along
    $\delta$ would give a disk in $M_1$ intersecting $\gamma_1$ at most twice.
    Such a disk would be trivial in $M_1$ and imply that either the plumbing
    band or some component of $\gamma_1$ is trivial. Thus we may assume
    $\delta$ intersects $\gamma$ once. Note that this means that $E$ intersects
    $\gamma$ twice, all the arcs of $E \cap D$ separate these points of
    intersection on $\partial E$, and besides $\tau$ there is exactly one other
    outermost arc, $\tau'$.  We can view $D$ as a rectangle with corners on
    $\gamma$ and with one pair of opposing sides in $\gamma_1$ and the other in
    $\gamma_2$.  Using $\delta$ to surger $D$ we again get a contradiction
    (e.g. to the minimality of intersection with $E$), unless $\tau$, viewed in
    this rectangle $D$, intersects $\gamma_1$ twice and is disjoint from
    $\gamma_2$. Letting $\delta'$ be the subdisk of $E$ bounded by $\tau'$, the
    same argument applied to $\delta'$, shows that $\tau'$ is disjoint from
    $\gamma_2$ and that $\delta'$ also lies in $M_1$. Let $\delta''$ be the
    disk of $E-D$ contiguous to $\delta$. Then $\delta''$ lies in $M_2$ and
    intersects $\gamma_2$ at most once. Thus $\delta''$ is boundary parallel in
    $M_2$. But this contradicts that each component of $\gamma_2$ is essential
    or the minimality of $E \cap  D$.  Thus we conclude that $\gamma$ is
    $3$--disk-busting in $M$.

    To prove part $(2)$, assume $A$ is a properly embedded annulus in $M$ which
    is disjoint from $\gamma$ and neither compressible nor boundary parallel.
    Isotop $A$ to intersect $D$ minimally and assume that $A$ is not disjoint
    from $D$.  There are no simple closed curves of intersection (each such
        would have to bound a disk in $A$ and an innermost such on $A$ will
    have to give a boundary parallel disk in some $M_i$). Furthermore, an arc
    of $A \cap D$ must run between different boundary components of $A$, else
    surgering $D$ along an outermost disk in $A$ gives, as argued above,
    a trivial disk in some $M_i$ which would imply that some component of
    $\gamma_i$ or the plumbing band in $M_i$ is trivial. Let $\tau$ be an arc
    of $E \cap D$ that is outermost on $D$ and let $\delta$ be the
    corresponding outermost subdisk of $D$.  Then $\boundary$--compressing $A$
    along $\delta$ gives a disk which is disjoint from $A$ and which intersects
    $\gamma$ at most twice. Since $\gamma$ is $3$--disk-busting, this disk must
    be boundary parallel in $M$. But this implies that $A$ is either
    compressible or boundary parallel.

    To see that $\gamma$ is annulus-busting in $M$, note that by $(1)$ it is
    disk-busting. Now let $A$ be an annulus disjoint from $\gamma$ in $M$ and
    assume it is incompressible and not boundary parallel. By $(2)$ it can be
    isotoped to be disjoint from the decomposing disk $D$.  But this contradicts
    the annulus-busting assumption on each $\gamma_i$.
\end{proof}

\begin{lemma}\label{lem:constructingbusting}
    For a genus $g>0$ handlebody, $H_g$, there is a non-separating \scc
    $\eta_g$ which is $3$--disk-busting and annulus-busting.
\end{lemma}

\begin{proof}
    We prove the Lemma by induction on $g$.

    Let $H_1$ be a solid torus and $\eta_1$ be a \scc in $\partial H_1$ with
    winding number $3$ in $H_1$. Then $\eta_1$ is clearly $3$--disk-busting and
    annulus-busting in $H_1$ (an incompressible annulus in a solid torus is
    boundary parallel). This proves the Lemma when $g=1$.  Let $b_1 \subset
    \partial H_1$ be a band running between opposite sides of $\eta_1$.

    Let $\eta$ be two copies of $\eta_1$ in $H_1$. Then $\eta$ is also
    $3$--disk-busting and annulus busting in $H_1$. Let $b$ be a band
    connecting the two components of $\eta$. Let $(H_2,\eta_2)$ be the boundary
    plumbing of $(H_1,\eta_1)$ and $(H_1,\eta)$ gotten by identifying $b_1$ and
    $b$.  By Lemma~\ref{lem:3diskbusting}, $(H_2,\eta_2)$ is $3$--disk-busting
    and annulus-busting.  $H_2$ is a genus $2$ handlebody and $\eta_2$ is
    connected and non-separating ($b_1$ can be thought of as gotten from
    a curve $c_1$ intersecting $\eta_1$ once in $\partial H_1$, this curve
    perturbed slightly shows $\eta_2$ to be non-separating). Thus we have verified
    the Lemma when $g=2$.

    Assume $\eta_g$ in $H_g$ satisfies the Lemma, for some $g>1$.  Let $c_g$ be
    a \scc in $\partial H_g$ intersecting $\eta_g$ once. A neighborhood of
    $c_g$ broken at $\eta_g$ becomes an embedded band $b_g$  that runs from one
    side of $\eta_g$ to the other. Let $(H_{g+1},\eta_{g+1})$ be the boundary
    plumbing of $(H_1,\eta)$ to $(H_g,\eta)$ along the bands $b$ and $b_g$.  By
    Lemma~\ref{lem:3diskbusting}, $(H_{g+1},\eta_{g+1})$ is $3$--disk-busting
    and annulus-busting.  $H_{g+1}$ is a genus $g+1$ handlebody and
    $\eta_{g+1}$ is connected and non-separating (as above, use a perturbed
    $c_g$ to verify that $\eta_{g+1}$ is non-separating).  This verifies the
    Lemma for $g+1$ and the induction.  \end{proof}

\subsection{Another $3$--disk-busting and annulus-busting curve.}

Figure~\ref{fig:3diskbusting}(left) shows a curve $\gamma_2$ in the boundary of
a genus $2$ handlebody $H_2$ that is disjoint from the curves $a(2,*,0),
b(2,*,0)$ in each of Figure~\ref{fig:typesofpants}(right) and (left).   We show
in Lemma~\ref{lem:ourcurveisbusting} that this curve $\gamma_2$ is both
$3$--disk-busting and annulus-busting.  In section~\ref{sec:defKtau}, for
$g\geq3$ we will boundary plumb $(H_2,\gamma_2)$  to $(H_{g-2},\eta_{g-2})$
along the bands $b_\gamma$ (also shown in Figure~\ref{fig:typesofpants}) and
$b_{g-2}$ as in the construction of Lemma~\ref{lem:constructingbusting} to
produce a new non-separating $3$--disk-busting and annulus-busting curve
$\gamma_g$ in the genus $g$ handlebody $H_g$ that is disjoint from curves
$a(g,*,0), b(g,*,0)$ induced from $a(2,*,0), b(2,*,0)$.

\begin{figure}[h!tb]
  \begin{center}
\includegraphics[width=0.95\textwidth]{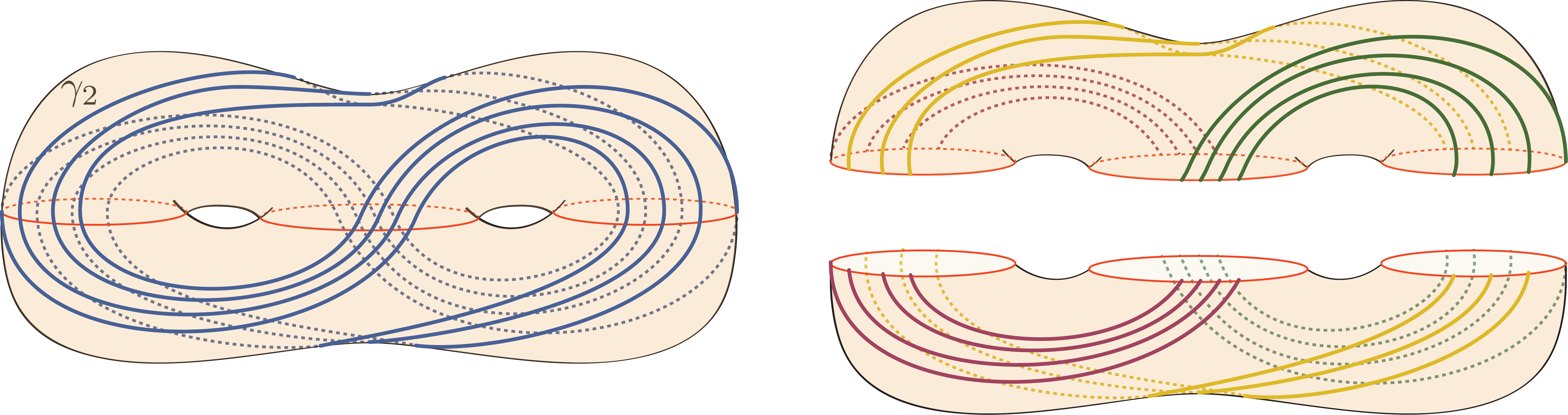}
  \end{center}
  \caption{The curve $\gamma_2$ in the boundary of the genus $2$ handlebody is
  shown with a compatible pants decomposition $\Pcal$ (left).   The handlebody
  is divided along $\Pcal$, and the arcs of intersection of $\gamma_2$ with
  each pair of pants is colored according to isotopy classes.}
  \label{fig:3diskbusting}
\end{figure}

\begin{figure}[h!tb]
  \begin{center}
\includegraphics[width=0.95\textwidth]{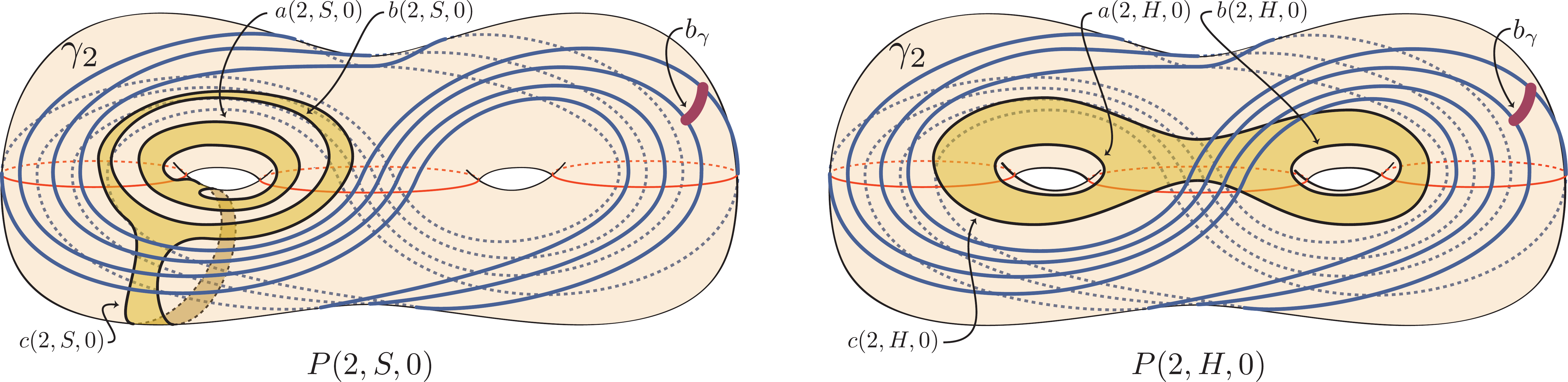}
  \end{center}
  \caption{Seifert type and handlebody type pairs of pants in the boundary of
  $H_2$ along with the curve $\gamma_2$.   A non-trivial band $b_\gamma$
  connecting $\gamma_2$ from just one side is also shown.}
  \label{fig:typesofpants}
\end{figure}

\begin{lemma}\label{lem:handleaddition}
Let $\gamma$ be a simple closed curve in the boundary of a genus $2$ handlebody
$H$.  If $\gamma$ is disk-busting but not annulus-busting then $H[\gamma]$
contains an essential annulus.
\end{lemma}

\begin{proof}
    Assume $A$ is a properly embedded essential annulus in $H$ that is disjoint
    from $\gamma$ and becomes inessential in $H[\gamma]$.  Because $\gamma$ is
    disk-busting, the Handle Addition Lemma \cite[Lemma 2.1.1]{CGLS}, shows
    that $H[\gamma]$ is irreducible and boundary-irreducible.

    First we show that $A$ must be separating in $H$. If $A$ were compressible
    in $H[\gamma]$, each component of $\partial A$ would have to be trivial in
    $\partial H[\gamma]$.  This would then imply that $A$ is separating in $H$,
    since the boundary components of a non-separating annulus in $H$ are
    individually and jointly non-separating.   On the other hand, if $A$ is
    boundary-parallel in $H[\gamma]$, $A$ must be separating in $H$. Thus $A$
    is separating.

    A $\bdry$--compression of $A$ in $H$ shows that $A$ may be viewed as the
    result of banding an essential separating disk to itself on one side.  Thus
    $A$ divides $H$ into a solid torus and another genus $2$ handlebody $H'$.
    Since $A$ is not boundary-parallel in $H$, it must wind around this solid
    torus more than once.    So if $\gamma$ were on the solid torus side of
    $A$, then $H[\gamma]$ would contain a lens space summand; but this cannot
    occur since $H[\gamma]$ is irreducible. Thus $\gamma$ must lie in $\bdry
    H'$. Then  irreducibility similarly shows that $A$ cannot be compressible
    in $H[\gamma]$.    So $A$ must be boundary-parallel in $H[\gamma]$. That
    is, $H'[\gamma]$ must become a solid torus in which $A$ winds once.  Hence
    $H[\gamma]$ is also a solid torus.  But $H[\gamma]$ is boundary-irreducible.
\end{proof}

\begin{lemma}\label{lem:hypmakesbusting}
    Let $\gamma$ be an essential simple closed curve in the boundary of a genus
    $2$ handlebody $H$.  If (the interior of) $H[\gamma]$ is a hyperbolic
    $3$--manifold, then $\gamma$ is disk-busting and annulus-busting.
\end{lemma}

\begin{proof}
    Assume the interior of $H[\gamma]$ is hyperbolic.  Then it contains no
    properly embedded essential disks or annuli.  If $\gamma$ were not
    disk-busting, then there would be a non-separating disk in $H$ disjoint
    from $\gamma$ that would become a non-separating disk in $H[\gamma]$.  Thus
    $\gamma$ is disk-busting.  Lemma~\ref{lem:handleaddition} then implies that
    $\gamma$ must also be annulus-busting.
\end{proof}

\begin{lemma}\label{lem:ourcurveisbusting}
    The curve $\gamma_2$ in the boundary of the genus $2$ handlebody $H$ as
    shown in Figure~\ref{fig:3diskbusting}(left)  is $3$--disk-busting and
    annulus-busting.
\end{lemma}

\begin{proof}
    Figure~\ref{fig:3diskbusting}(right) demonstrates that $\gamma_2$ is
    $3$--seamed with respect to a compatible pair of pants decomposition of
    $H$.  By Lemma~\ref{lem:seamedisbusting}, this means $\gamma_2$ is
    $3$--disk-busting.  Figure~\ref{fig:hyperbolic-handle} shows that
    $H[\gamma_2]$ is the exterior of the $5_2$ knot.  Since this manifold is
    hyperbolic, Lemma~\ref{lem:hypmakesbusting} implies that $\gamma_2$ is
    annulus-busting.
\end{proof}

\begin{figure}[h!tb]
  \begin{center}
      \includegraphics[width=0.95\textwidth]{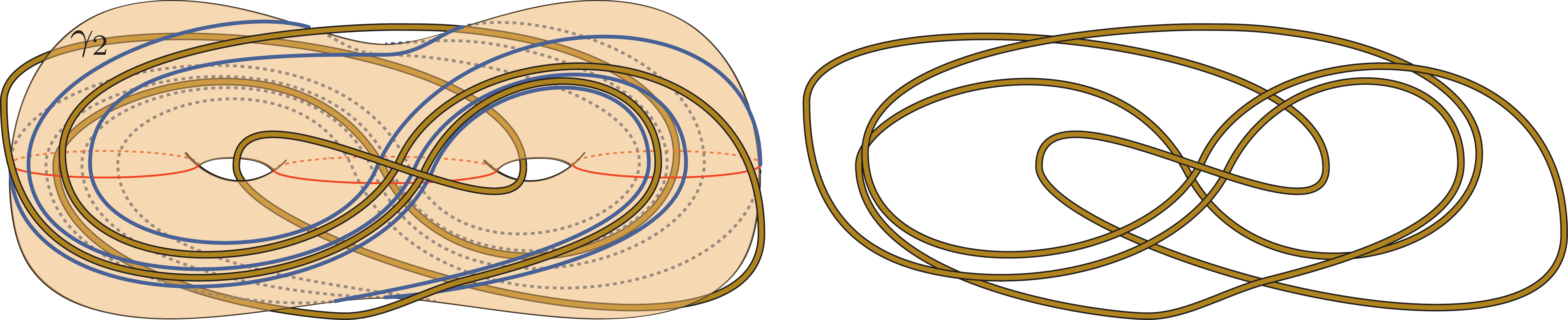}
  \end{center}
  \caption{Attaching a $2$--handle to the genus 2 handlebody $H_2$ along the
  curve $\gamma_2$ produces the exterior of the  $5_2$ knot, a hyperbolic
  manifold.  Shown (with and without $H_2$) is a knot disjoint from the
  standardly embedded $H_2$ and from a $2$--handle attached to $\gamma_2$.  One
  easily confirms this  is the $5_2$ knot.}
  \label{fig:hyperbolic-handle}
\end{figure}

\subsection{Disk and annulus hitting numbers}

Here we quantify the extent to which a knot is disk-busting or annulus-busting.

\begin{definition}{\em (Disk Hitting Number)}\label{def:wrappingno}
    Let $M$  be a compact, connected, orientable manifold with boundary.  Let
    $\mathcal{D}$ be the \defn{disk set} of $M$, i.e.\ the set of  properly
    embedded disks in $M$ that are not boundary-parallel.
    Assuming $\mathcal{D} \neq \emptyset$, the \defn{disk hitting number of $K$
    in $M$}, $\gwrap_D(K)$, is the smallest number of times $K$ transversely
    intersects an essential disk in $M$.  That is, \[ \gwrap_D(K) = \min_{D\in
    \mathcal{D}}\,|K \cap D|. \]
\end{definition}

The disk hitting number is analogous to the wrapping number  of a link in
a solid torus.  In a similar way, one can also define a hitting number for
a knot in a $3$--manifold relative to other homeomorphism types of surfaces.
For our purposes, we only need the annular case.

\begin{definition}{\em (Annulus Hitting Number)}\label{def:annhitno}
    Let $M$ be a compact, connected, orientable manifold with boundary.  Let
    $\mathcal{A}$ be the \defn{annulus set} of $M$, i.e.\ the set of
    incompressible, properly embedded annuli in $M$ that are not boundary-parallel.  
    Assuming $\mathcal{A} \neq \emptyset$, the \defn{annular hitting
    number of $K$ in $M$}, $\gwrap_{A}(K)$, is the smallest number of times
    that $K$ transversely intersects an essential annulus in $M$.  That is, \[
    \gwrap_{A}(K) = \min_{A\in \mathcal{A}}\, |K \cap A|. \]
\end{definition}

\subsection{Handlebody, knot pairs with large hitting numbers}
\label{sec:defKtau}

The goal of this section is Proposition~\ref{prop:wrappingno} which gives
infinitely many knots in a genus $g$ handlebody admitting non-trivial
handlebody or Seifert-type Dehn surgeries, distinguished by their hitting
numbers.  Proposition~\ref{prop:strongandhyp} shows that generically the
exteriors of these knots are irreducible, $\partial$--irreducible, atoroidal,
and anannular.

Figure~\ref{fig:typesofpants} depicts two pairs of a genus $2$ handlebody $H_2$
with a pair of pants in the boundary of $H_2$; on the left the pair $(H_2,
P(2,S,0))$ is of Seifert type and on right the pair $(H_2,P(2,H,0))$ is of
handlebody type (see Definitions~\ref{def:handlebodytype} and
~\ref{def:seiferttype}).  Also shown in each is a curve $\gamma_2 \subset \bdry
H_2$ that is $3$--disk-busting and annulus-busting according to
Lemma~\ref{lem:ourcurveisbusting}.

\bigskip

\begin{definition}[{\em Of $\gamma_g$, $P(g,*,i)$, and $M(g,*,i)$}]\label{def:gammag}
    Let $b_\gamma$ be the band connecting $\gamma_2$ to the same side of itself
    in $\bdry H_2$ as shown in Figure~\ref{fig:typesofpants}.  Then, for $g
    \geq 3$, let $\gamma_g$ be the curve in the boundary of the genus $g$
    handlebody $H_g$ obtained by boundary plumbing $(H_{g-2}, \eta_{g-2})$ and
    $(H_2, \gamma_2)$ together along the bands $b_{g-2}$ and $b_\gamma$.   (The
    triple $(H_{g-2},\eta_{g-2},b_{g-2})$ are constructed in
    Lemma~\ref{lem:constructingbusting}.) By Lemma~\ref{lem:3diskbusting}
    $\gamma_g$ is $3$--disk-busting and annulus-busting.

    Since $b_\gamma$ is disjoint from $P(2,*,0)$ for $* \in \{S,H\}$, this
    boundary plumbing induces the pair of pants $P(g,*,0)$ in the boundary of
    $H_g$.  For each integer $i$, let $P(g,*,i)$ be the pair of pants obtained
    from $P(g,*,0)$ by $i$ positive Dehn twists along $\gamma_g$. $P(g,*,i)$
    has two boundary components which are disjoint from $\gamma_g$, label these
    $a(g,*,i),b(g,*,i)$. Label the third component $c(g,*,i)$.  Then
    $a(g,S,i),b(g,S,i)$ are parallel and primitive in $H_g$;
    $a(g,H,i),b(g,H,i)$ are jointly primitive in $H_g$.  Furthermore,
    $c(g,*,i)$ is gotten from $c(g,*,0)$ by Dehn-twisting along $\gamma_g$ in
    $\bdry H_g$.

    Recall from Section~\ref{section:little-handlebody} the identifications of
    the genus $2$ handlebody $H$ as a product $T \times [-1,1]$ for
    a once-punctured torus $T$ and also as a product $P \times [-1,1]$ for
    a pair of pants $P$. The boundary components of $P$ are labelled
    $\bdry_{-}P, \bdry_{+}P, \bdry_0P$.  Recall that  $K(\tau)$ is a knot in $T
    \times \{0\} \subset H = T \times [-1,1]$ obtained from an essential \scc
    $\tau \subset T$.
    Fix a homeomorphism $\phi(*) \colon P \to P(2,*,0)$ that identifies
    $\bdry_{-}P, \bdry_{+}P$ with $a(2,*,0), b(2,*,0)$  in $P(2,*,0)$.  Let
    $\phi(g,*,i)\colon P \to P(g,*,i)$ be the homeomorphism gotten by $\phi(*)$
    followed by twisting along $\gamma_g$. To $(H_g,P(g,*,i))$ constructed
    above, attach $H$ using $\phi(g,*,i)$ to identify $P \times \{-1\}$ with
    $P(g,*,i)$.  Call the resulting genus $g$ handlebody $M(g,*,i)$. Let
    $K(\tau,g,*,i)$ be the knot in $M(g,*,i)$ that is the image of $K(\tau)
    \subset H$ under inclusion.
\end{definition}
\bigskip

We will apply Theorem~\ref{thm:essentialsurfaces}  to show
that for fixed $\tau (\neq \mu,\lambda)$ and $g\geq 2$
\begin{enumerate}
    \item the disk hitting numbers of the knots
        $K(\tau,g,*,i) \subseteq M(g,*,i)$ go to infinity with $i$, and
    \item the annular hitting numbers of the knots
        $K(\tau,g,*,i) \subseteq M(g,*,i)$ go to infinity with $i$.
\end{enumerate}

To do this, let $\Sigma\times [0,1]$ be a collar neighborhood of $\boundary
M(g,*,0)$, so that $\Sigma\times\set{0}=\boundary M(g,*,0)$.  Let $J_g$ be the
curve $\gamma_g\times\set{1}\subseteq\Sigma\times [0,1]$.  There is an annulus
that runs from $J_g$ to $\bdry M(g,*,0)$, $\Uhat_g=\gamma_g\times [0,1]$.  We
identify $H\cong P\times I$ in $\Sigma\times [0,1]$ as $P(g,*,0)\times [0,
1/2]$.  Then the knots $K(\tau,g,*,i)\subseteq M(g,*,i)$ are gotten from the
knots  $K(\tau,g,*,0)\subseteq M(g,*,0)$ by twisting $i$ times along the
annulus $\Uhat_g$ (under the identification $M(g,*,i) \cong H_g \cong
M(g,*,0)$).   This twisting is basically the same as the construction described
in \cite{BGL} and discussed in Definition~\ref{defn:twisting} except that one
boundary component of the annulus $\Uhat_g$ is in the boundary of the
$3$--manifold $M(g,*,0)$. Specifically,

\begin{definition}\label{def:twistingalongU}
    Let $N$ be a submanifold of $M(g,*,0)$ that is disjoint from $J_g$. Let
    $\Uhat_g \times [0,1]$ be a product neighborhood of $\Uhat_g$ in
    $M(g,*,0)$.  Let $h_i:\Uhat_g \times [0,1] \to \Uhat_g \times [0,1]$ be
    a homeomorphism gotten by $i$ complete twists along $\Uhat_g$, where $h_i$
    is the identity on $\Uhat \times \{0,1\}$. Define the submanifold {\em
    N twisted $i$ times along $\Uhat_g$} as $[N-(\Uhat_g \times [0,1])] \cup
    h_i(N \cap (\Uhat_g \times [0,1]))$. On the boundary of $M(g,*,0)$, $h_i$
    induces $i$ Dehn twists along $\gamma_g$. The sign of the twist along
    $\Uhat$, $h_i$, is taken so that the induced twist on the boundary is
    a positive Dehn twist along $\tau_g$.
\end{definition}

Above, we are twisting $N=K(\tau,g,*,0)$ along $\Uhat_g$ to get $K(\tau,g,*,i)$.

Our application of Theorem~\ref{thm:essentialsurfaces} in
Proposition~\ref{prop:wrappingno} requires the following constraint on \mobius
bands and essential annuli in the exterior of $J_g$.

\begin{lemma}\label{lem:no-spanning-annulus2}
    Let $(M, K)=(M(g,*,0), K(\tau,g,*,0))$ and $J_g$ be as above.  Then
    $M_{J_g}$ is irreducible and boundary-irreducible. Assume that $\tau$ is
    not parallel to $\mu$ or $\lambda$ in $T$.  Let $T_J$ be the torus
    component of $\bdry M_{J_g}$.  There is no incompressible annulus in
    $M_{J_g}$ disjoint from $K$ with one boundary component on $T_J$ and the
    other on $\bdry M$. Furthermore, there is no \mobius band properly embedded
    in $M_{J_g}$  with boundary on $T_J$.
\end{lemma}

\begin{proof}
    Assume for contradiction that $U'$ is such an annulus, and let $U_g$ be the
    restriction of $\Uhat_g$ to $M_{J_g}$.  Because $\gamma_g$ is disk-busting,
    $M_{J_g}$ is irreducible and  boundary-irreducible.  This implies that $U'$
    and $U_g$ can be isotoped in $M_{J_g \cup K}$ to have disjoint boundaries
    on $T_J$ and, in fact, isotoped in $M_{J_g}$ to be disjoint.  Let $A$ be
    the properly embedded annulus in $M$ gotten by taking the union of
    $\Uhat_g$ and (an extension of) $U'$ along $J_g$. $A$ must be parallel into
    $\bdry M$, otherwise there is an essential disk in $M$ disjoint from $A$ --
    contradicting that $\gamma_g$ is disk-busting. But this implies that $U_g$
    and $U'$ are properly isotopic in $M_{J_g}$.
    Thus, by an isotopy in a neighborhood of $\partial M$, we may take
    $\partial U'=\gamma_g$. In particular, $\partial U'$ is disjoint from
    $a(g,*,0) \cup b(g,*,0)$ and no arcs of $\partial U' - c(g,*,0)$ are
    parallel into $c(g,*,0)$.  Let $P'=(P(g,*,0)\times \{1/2\}) \cup (\partial
    P(g,*,0) \times [0,1/2])$. Then $P'$ is incompressible in $M_{J_g}$ (since
    $P'$ is isotopic into $\partial M_{J_g}$ and $M_{J_g}$ is boundary-irreducible). 
    So we may assume that $U' \cap P'$ contains no simple closed
    curves which are trivial in either $U'$ or $P'$.  The arcs of intersection of
    $U'$ with $P'$ identify (in $U'$) a $\bdry$--compressing disk $D$ for $P'$ in
    $M$ that intersects $\partial M$ in a component of $\gamma_g - c(g,*,0)$.  One
    checks that $D$ must then lie in $H$ and $D \cap \partial M$ is an essential
    arc in $P(g,*,0)$ running from $c(g,*,0)$ to itself.  As $D$ comes from $U'$,
    it is disjoint from $K$.  As $\nu$ is isotopic to $c(g,*,0)$ in $H$, $\tau$
    then cannot be isotopic to $\nu$ in $T$.
    Lemma~\ref{lem:onearc} implies that $\tau$ must be isotopic to $\mu$ or
    $\lambda$ in $T$ --- contradicting our assumptions.

    Now assume there were a \mobius band, $B$, in $M_{J_g}$.  Note that $B$
    must be $\bdry$--incompressible in $M_{J_g}$ since $T_J$ does not bound
    a disk in $M_{J_g}$.  Then $U_g$ can be isotoped in $M_{J_g}$ such that its
    boundary is disjoint from that of $B$. In fact, we may take $B$ to be
    disjoint from $U_g$ (they must intersect in simple closed curves parallel
    to the boundary in $B$).  But then we can amalgamate $\Uhat_g$ with $B$
    along $J_g$ to obtain a \mobius band, $\widehat{B}$, properly embedded in
    $M$. There must be a meridian disk in $M$ disjoint from $\widehat{B}$ and
    in particular from $\bdry \Uhat_g$ in $M$.  But this contradicts that
    $\gamma_g$ is disk-busting.
\end{proof}

The main result of this section is:

\begin{prop}\label{prop:wrappingno}
    Assume $g\geq 2$ and $\tau$ is not isotopic to $ \mu$ or $\lambda$ in $T$.
    Consider the knot $K(\tau,g,*,i)$ in the handlebody $M(g,*,i)$ where $* = H,S$.
    \begin{enumerate}
        \item $K(\tau,g,H,i)$ admits a longitudinal surgery that is
            a handlebody.

        \item $K(\tau,g,S,i)$ admits a longitudinal surgery which is
            a $D(p,q)$--Seifert space in union with $g-1$ one-handles when
            $\tau=\tau(p,q)$.

        \item There is a function $N(\tau)$, i.e. independent of $g,*,i$, such
            that $b_g(K(\tau,g,*,i)) < N(\tau)$.

        \item There is a constant $N_D(\tau,*)>0$ such that
            $\max(\gwrap_D(K(\tau,g,*,i)),1)\geq |i| \cdot N_D(\tau,*) -1$, and
            hence $\gwrap_D(K(\tau,g,*,i)) \to \infty$ as $i \to \infty$.

        \item There is a constant $N_A(\tau,*)>0$ such that
            $\max(\gwrap_A(K(\tau,g,*,i)),1)\geq |i| \cdot N_A(\tau,*) -2$, and
    \end{enumerate}
\end{prop}

\begin{proof}
    Statements (1) and (2) follow directly from
    Proposition~\ref{k-has-handlebody-surgery}.
    %Lemma~\ref{lem:surgeryontwisted}.

    Fix $\tau,g$ and let $(M, K)=(M(g,*,0), K(\tau,g,*,0))$. Let
    $L=J_g=\gamma_g\times\set{1}\subseteq\Sigma\times [0,1], \Uhat=\Uhat_g$ be
    as above (preceding Lemma~\ref{lem:no-spanning-annulus2}).

    {\bf Statement (3).}
    Put $\tau \subset H=P\times I$ in bridge position with respect to the
    height function of this product structure.  Let $N(\tau)$ be the bridge
    number of this presentation.  When included into $M$, $H$ can be regarded
    as residing within a collar neighborhood of $\bdry M$ so that this product
    structure of $H$ is compatible with the product structure of the collar.
    Hence this puts $K$  in bridge position entirely within a collar
    neighborhood of $M$ that is disjoint from $L$, so the bridge number of $K$
    in $M$ is at most $N(\tau)$. As the knots $\{K(\tau,g,*,i), i \in \Z\}$ are
    obtained from $K$ by twisting along the vertical annulus $\Uhat$, which
    preserves the height in the collar, the bridge numbers of these knots is
    also at most $N(\tau)$. This is statement $(3)$ of the Proposition, noting
    that the bound is independent of $g,*,i$.

    For statements (4) and (5) we need a catching surface in order to apply
    Theorem~\ref{thm:essentialsurfaces}.  Let $\barQ_2$ be the meridian disk of
    $H_2$ whose boundary is the leftmost curve of the pants decomposition
    $\Pcal$ in Figure~\ref{fig:3diskbusting}.  Observe that $\gamma_2$
    intersects $\barQ_2$ algebraically once (with appropriate orientation) but
    geometrically seven times. Let $\barQ_g$ be the corresponding disk gotten
    by the inclusion of $H_2$ into $H_g$.   We also see algebraic intersection
    number one (with appropriate orientations) and geometric intersection
    number seven between $\gamma_g$ and $\barQ_g$.

    It follows that in $M$,  $L$ and $\barQ_g$ also algebraically intersect
    once but geometrically seven times.  Keeping $L$ fixed, tube oppositely
    oriented intersections of $L$ and $\barQ_g$ together to produce an oriented
    surface that $L$ intersects coherently and then isotop $K$ in the
    complement of $L$ to intersect this surface transversally. Let $Q$ be the
    restriction of the tubed $\barQ_g$ to the exterior of $K \cup L$ in $M$,
    $X=M_{K \cup L}$.  Note that $\chi(Q) = \chi(\barQ_g -\nbhd(K \cup L))\leq
    \chi(\barQ_g-\nbhd(L)) = -6$.  Let $T_K$ be the component of $\bdry X$
    corresponding to $K$ and $T_L$ be that corresponding to $L$.  Observe that,
    by construction, $Q$ intersects $T_L$ in a single component, a meridian of
    $L$. $Q$ will serve then as a catching surface for $(F,K,L)$ where $F$ is
    a properly embedded surface in $X$ whose boundary on $T_L$ is not
    meridional, see  Definition~\ref{def:catching1}. Note that being a catching
    surface does not require $Q$ to be incompressible or
    $\bdry$--incompressible in $X$.
    The surface $F$ will differ in $(4)$ and $(5)$.

    {\bf Statement (4).}
    Recall from the discussion preceding Lemma~\ref{lem:no-spanning-annulus2},
    that we may think of the knots $K(\tau,g,*,i)$ as gotten from
    $K=K(\tau,g,*,0)$ by twisting $i$ times along the annulus $\Uhat_g$ in $M$.
    Let $U_g$ be the restriction of this annulus to $X$. Let $\barF_i$ be
    a meridian disk of $M(g,*,i)$ that, among all meridian disks of $M(g,*,i)$,
    intersects $K(\tau,g,*,i)$ minimally.  This geometric intersection number
    is $\gwrap_D(K(\tau,g,*,i))$.  Furthermore, subject to this minimality
    assume that $\barF_i$ intersects $L$ minimally, a non-empty intersection
    since $\gamma_g$ is disk-busting in $M$.  Let $F_i$ be the restriction of
    $\barF_i$ to the exterior, $X_i$, of $K(\tau,g,*,i) \cup L$ in $M(g,*,i)$.
    Then $F_i$ is incompressible in $X_i$. Note that twisting along the annulus
    $\Uhat_g$ induces a homeomorphism between $X$ and $X_i$. Under this
    homeomorphism we view $F_i$ as a surface properly embedded in $X$. We are
    now in the context of Section~\ref{sec:essentialsurfaces} with the surfaces
    $F=F_i$ and $Q$ properly embedded in $X$. In that notation $\alpha_K,
    \beta_K$ are meridians of $K$ and consequently $\Delta_K=0$ so that
    $\Delta_K'=1$. As $\beta_L$ is a meridian of $L$ and $\alpha_L$ is obtained
    from a meridian of $L$ by Dehn-twisting longitudinally $-i$ times in $T_L$
    along $\partial U_g$, we have $\Delta_L=|i|$. In particular, when $i\neq0$
    then $\alpha_L$ is not a meridian of $L$, so $Q$ is a catching surface for
    $(F,K,L)$. Here $f_M=f_M'=1, \chi(\widehat{F})=2, \chi(Q)<0$ and
    $f_K=\gwrap_D(K(\tau,g,*,i))$.  (Recall that $\widehat{F}$ is the surface
    $F$ with all of its boundary components capped off with disks.)  Finally,
    $F$ is incompressible in $X$, hence it is $\boundary$--incompressible at $T_L$,
    and $|F \cap T_L| \neq 0$.  Then Theorem~\ref{thm:essentialsurfaces} along
    with Lemma~\ref{lem:no-spanning-annulus2} implies that

    \[ \max(\gwrap_D(K(\tau,g,*,i),1) \geq \frac{|i|}{36 |\chi(Q)|}-1. \]

    When $i=0$, note that the above holds trivially.  Observe that $\chi(Q)$
    depends on both $\tau$ and $* \in \{H,S\}$, but is independent of $g$ and
    $i$.  For statement (4) we then set $N_D(\tau,*) =  \frac{1}{36 |\chi(Q)|}
    > 0$.  Then, as $|i|\to\infty$, $\max(\gwrap_D(K(\tau,g,*,i)),1) \geq |i|
    N_D(\tau,*)-1$ implies $\gwrap_D (K(\tau,g,*,i)) \to \infty$.

    {\bf Statement (5).}
    Repeat the argument above taking $\barF_i$ to be an essential annulus.
    More specifically, let $\barF_i$ be a properly embedded essential annulus
    in $M(g,*,i)$ that, among all essential annuli in $M(g,*,i)$ intersects
    $K(\tau,g,*,i)$ minimally.  This intersection number is
    $\gwrap_{A}(K(\tau,g,*,i))$. Furthermore, subject to this minimality assume
    that $\barF_i$ intersects $L$ minimally, a non-empty intersection since
    $\gamma_g$ is annulus-busting. As above, let $F_i$ be the restriction of
    $\barF_i$ to $X_i$ and, via the homeomorphism between $X$ and $X_i$, regard
    $F_i$ as an incompressible surface $F$ properly embedded in $X$.  Again, we
    have $\Delta_K=0$ so that $\Delta'_K=1$, $\Delta_L=|i|$, and
    $\chi(\widehat{F})=2$, but now $f_M=f'_M=2$, and $f_K
    =\gwrap_{A}(K(\tau,g,*,i))$.  When $i \neq 0$, then $Q$ is a catching
    surface for $(F,K,L)$.

    Then Theorem~\ref{thm:essentialsurfaces} along with
    Lemma~\ref{lem:no-spanning-annulus2} implies that

    \[ \max(\gwrap_{A}(K(\tau,g,*,i)),1) \geq \frac{|i|}{72 |\chi(Q)|}-2. \]

    For statement (5) we then set $N_A(\tau,*) =  \frac{1}{72 |\chi(Q)|}$.
    Then, as $|i|\to\infty$, $\max(\gwrap_A(K(\tau,g,*,i)),1) \geq |i|
    N_A(\tau,*)-2$ implies $\gwrap_A(K(\tau,g,*,i)) \to \infty$.
\end{proof}

Section~\ref{section:big-handlebody} talks about when a pair of pants in the
boundary of a handlebody is of strong Seifert type or strong handlebody type,
a notion which allows us to get better control over the exteriors of the  knots
constructed.

\begin{prop}\label{prop:strongandhyp}
    For $g\geq 2$, there is a constant $N_{strong}$ such that when $|i|
    > N_{strong}$,
    \begin{itemize}
        \item $M(g,*,i)$ is of strong Seifert type or strong handlebody type, and

        \item when $\tau\not\in\{\lambda, \mu, \nu, \lambda-\mu, \lambda+\nu,
            \mu+\nu\}$, the exterior of $K(\tau,g,*,i)$ in $M(g,*,i)$ is
            irreducible, atoroidal, $\bdry$--irreducible, and anannular.
    \end{itemize}
\end{prop}
\begin{proof}
    The pair $(H_g,c(g,*,i))$ is homeomorphic to the pair
    $(M(g,*,i),K(\nu,g,*,i))$ since $K(\nu)$  is isotopic in $H$ to
    $\partial_0 P=c(g,*,i)$.  By Proposition~\ref{prop:wrappingno}(4), for any
    $n$  there is an $N_D(\nu,n)$ such that if $|i|>N_D(\nu,n)$,
    $\gwrap_DK(\nu,g,*,i)>n$ in $M(g,*,i)$ (taking the larger of the constants
    over $*=H,S$).  Thus for  $|i|>N_D(\nu,n)$, $c(g,*,i)$ is $n$--disk-busting
    in $H_g$.  Similarly, by Proposition~\ref{prop:wrappingno}(5), there is
    a constant $N_{A}(\nu, n)$ independent of $g$ such that for any $n\in \N$,
    if $|i|>N_{A}(\nu,n)$, $\gwrap_AK(\nu,g,*,i)>n$ in $M(g,*,i)$.  Thus for
    $|i|>N_A(\nu,n)$, $c(g,*,i)$ is $n$--annulus-busting in $H_g$.

    Set $N_{strong} = \max(N_D(\nu,3),N_A(\nu,1))$ and assume $|i|
    > N_{strong}$.  Then $c(g,*,i)$ is $3$--disk-busting and annulus-busting.
    Lemma~\ref{lem:old3diskbusting} then implies that $(H_g,P(g,*,i))$, and
    hence $M(g,*,i)$, is of strong Seifert type or strong handelbody type.
    Proposition~\ref{prop:exterior-hyperbolic} completes the proof.
\end{proof}

%%%%%%%%%%%%%%%%%%%%%%%%%%%%%%%%%%%%
%%%%%%%%%%%%%%%%%%%%%%%%%%%%%%%%%%%%

\section{Bridge number bounds}\label{sec:bridgenumberbounds}

In this section we compute a lower bound on  $b_g(K(\tau,g,*,i))$ in $M(g,*,i)$
for each $i$, Proposition~\ref{prop:bridge-number-bound}. We then combine this
with Propositions~\ref{prop:wrappingno} and \ref{prop:strongandhyp} to find the
infinite families of knots for our main result Theorem~\ref{thm:main}.  Recall
that in Section~\ref{section:little-handlebody}, $\tau(\kappa,\alpha,n)$ means
that the curve $\tau$ in $T \times \{0\}$ ($H=T \times [-1,1]$) is gotten from
$\kappa$ by Dehn twisting $n$ times along $\alpha$. Thinking of $\tau$ as
a knot, $K(\tau)$,  in $H$, we alternatively think of $K(\tau)$ as gotten by
twisting $n$ times the knot $K(\kappa)$ along the annulus $\Rhat(\alpha)$ in
$H$, see Definition~\ref{defn:twisting}. Recall that $\Rhat(\alpha)$ is the
product $\alpha \times [-1/2,1/2]$ in $T \times [-1,1]$, and  $\Lm,\Lp$ are the
components of $\bdry \Rhat(\alpha)$. When the curve $\alpha$ is clear, we will
only write $\Rhat$ for $\Rhat(\alpha)$.

\begin{lemma}\label{lem:boundary-slope-1-1}
    When $\alpha$ is parallel to $\nu$, $\Lp$ and $\Lm$ are isotopic in $H$ to
    $\boundary_0P$.  The slope that $\Rhat(\alpha)$ determines on $\bdry
    N(\Lp)$ and on $\bdry N(\Lm)$ is distance one
    (Definition~\ref{def:distance}) from the slope describing the parallelism
    to $\boundary_0P$.
\end{lemma}
\begin{proof}
    The first statement is obvious from the definition of $\Lp$ and
    $\Lm$.  By direct inspection, the slope of $\boundary R$ on
    $\boundary N(\Lp)$, say, is distance one from the surface slope of
    $\boundary M$; refer to Figure~\ref{fig:HandP}.
\end{proof}

Recall from Definition~\ref{def:gammag} that
$M(g,*,i)=(H,P)\cup_P(H_g,P(g,*,i))$. Below we will refer to $P$ as the
properly embedded copy of $P$ separating $H$ from $H_g$. Under the inclusion of
$H$ in $M(g,*,i)$, we consider $K(\tau), K(\kappa), \Rhat(\alpha)$ as living in
$M(g,*,i)$.  Following~\cite{BGL}, a genus $g$ splitting of the genus $g$
handlebody $M(g,*,i)$ is the decomposition of $M(g.*,i)$ into a union
$H_1\cup_{\Sigma} H_2$ where $H_1$ is a genus $g$ handlebody, $\Sigma$ is
a surface of genus $g$, and $H_2\cong \Sigma\times I$.  This is the
\defn{standard splitting} of $M(g,*,i)$.

\begin{lemma}\label{R-not-isotopic}
    Assume that $M=M(g,*,i)$ is of strong handlebody or strong Seifert type.
    If $\alpha$ is not parallel to $\mu$ or $\lambda$, then $\Rhat(\alpha)$ is
    not isotopic into the standard splitting of $M$.
\end{lemma}

\begin{proof}
    That the type is strong guarantees that $P$ is incompressible in $M$,
    $\boundary$--incompressible in $M-H$ (for the handlebody type see
    Lemma~\ref{lem:old3diskbusting}), and that $\partial_0 P$ is disk-busting
    in $M$.   Suppose first that $\alpha$ is not parallel to $\mu$, $\lambda$,
    or $\nu$.  We will show that the core of $\Rhat$ is not isotopic into
    $\boundary M$.  Suppose that it is, so that there is an annulus $A$ with
    one boundary component, $\boundary_1 A$, on the core of $\Rhat$ and the
    other, $\boundary_2A$, lying on $\boundary M$.  Choose $A$ to minimize
    $|A\cap P|$ and note that $A$ is incompressible in $M$ since $L_+$ and
    $L_-$ are nontrivial knots.  Because $P$ is also incompressible, there are
    no trivial \sccs of intersection of $A$ with $P$.  Among arcs of
    intersection of $A\cap P$, choose $a$ to be outermost in $A$ so that it
    cuts off a disk $D\subseteq A$ with $\interior D\cap P = \emptyset$,
    $\boundary D=a\cup b$, and $b\subseteq\boundary M$.  If $b$ is trivial in
    $\boundary M-P$ we can isotop $A$ to reduce $|A\cap P|$.  By the
    incompressibility of $\partial M - P$, Lemma~\ref{lem:old3diskbusting}, $a$
    is not trivial in $P$. Therefore $D$ must be a $\boundary$--compressing
    disk for $P$ in $M$ that is disjoint from $K(\alpha)$, the core of $\Rhat$.
    As $P$ is $\boundary$--incompressible in $M-H$, $D$ must lie in $H$. But by
    Lemma~\ref{lem:onearc} and our restrictions on $\alpha$, no such disks
    exist.  Therefore $A\cap P=\emptyset$, which implies that $L_+$ and $L_-$
    are isotopic to a component of $\boundary P$.  Since $H \cong T \times I$,
    $\alpha$ is isotopic in $T$ to $\mu$, $\lambda$, or $\nu$.

    Now suppose that $\alpha$ is parallel to $\nu$ and $\Rhat$ is isotopic into
    the standard splitting surface $\Sigma$ for $M$.  After an isotopy of
    $\Rhat$ into $\Sigma$, observe that $\Sigma-L_+$ is incompressible because
    $L_+$ is disk-busting in $M$  (it is isotopic to $\nu$ and hence to $c$).
    Let $\gamma$ be the framing on $L_+$ induced by $\Sigma$, which is the same
    as that from $\Rhat$.  Since $M \cut \Sigma$ is the product between
    $\Sigma$ and $\bdry M$ and the handlebody bounded by $\Sigma$,
    $M_{L_+}(\gamma)$ can be viewed as the union of these two pieces each with
    a $2$--handle attached along the imprint of $L_+$.  By the Handle Addition
    Lemma \cite[Lemma 2.1.1]{CGLS}, the resulting boundary $\Sigma'$ of the
    handlebody with a $2$--handle attached along $L_+$ is incompressible into
    that side of $\Sigma'$. The other side of $\Sigma'$ is a compression body
    in which $\Sigma'$ is also incompressible.  Thus $\Sigma'$ is an
    incompressible surface in $M_{L_+}(\gamma)$.  Recall that when $\alpha$ is
    parallel to $\nu$, $\Lp$ is isotopic to $\boundary_0P$.  However, by
    Lemma~\ref{lem:boundary-slope-1-1}, the framing, $\gamma$, of $\Rhat$ on
    $\Lp$ is distance one from the surface framing by $\boundary M$.
    Therefore, surgery on $\Lp$ at the slope $\gamma$ yields a handlebody.
    Handlebodies contain no closed incompressible surfaces, and so $\Rhat$ is
    not isotopic into $\Sigma$.
\end{proof}

Let $\tau = \tau(\kappa,\alpha,n)$ and recall the handlebody/knot pair
$(M(g,*,i), K(\tau,g,*,i))$ of section~\ref{sec:defKtau}. Under the inclusion
of $H$ in $M(g,*,i)$, we consider $K(\kappa),K(\alpha)$ as well as the annulus
$\Rhat(\alpha)$ and its boundary $\Lm,\Lp$ as embedded in $M(g,*,i)$. Then the
knot $K(\tau,g,*,i)$ is gotten from $K(\kappa,g,*,i)$ by twisting $n$ times
along the annulus $\Rhat(\alpha)$.  Write $\L$ for the link $K(\kappa) \cup \Lm
\cup \Lp$ in $M(g,*,i)$ and $M(g,*,i)_{\L}$ for its exterior. Denote by
$T_+,T_-,T_{\kappa}$ the components of $\partial M(g,*,i)_{\L}$ corresponding
to $\Lp,\Lm,K(\kappa)$.  We want to find a  catching surface, $Q$, as defined
in \cite[Definition 2.1]{BGL}, for the pair $(\Rhat(\alpha),K(\kappa))$ in
$M(g,*,i)$.
We repeat that definition in this context.

\begin{definition}\label{def:catches}
    Let $Q$ be an oriented surface properly embedded in $M(g,*,i)_{\L}$. We say
    that $Q$ {\em catches} the pair $(\Rhat(\alpha),K(\kappa))$ if
    \begin{itemize}
        \item $\bdry Q \cap T_i$ is a non-empty collection of coherently
            oriented parallel curves on $T_i$ for each $i \in \{+,-\}$; and

        \item $\partial Q$ intersects both $T_+$ and $T_-$ in slopes different
            than $\partial R(\alpha)$.
    \end{itemize}
\end{definition}
%%%%added

\begin{lemma}\label{lem:catching} Let $\tau=\tau(\kappa,\alpha,n)$. The
    annulus/knot pair $(\Rhat(\alpha),K(\kappa))$ in $M(g,*,i)$ is caught by
    a surface $Q$ with $\chi(Q) < 0$.
    \begin{enumerate}
        \item Assume $\alpha$ is not parallel to $\mu, \lambda$ or $\nu$ in $T$
            and $g \geq 3$. Then $Q$ intersects $T_+,T_-$ in meridians, and if
            $Q$ intersects $T_{\kappa}$ at all it does so in meridians.

        \item When $\alpha$ is not parallel to $\mu, \lambda$ or $\nu$ and
            $g=2$, $Q$ can be taken so that $\chi(Q)$ depends on
            $\kappa,\alpha$ but not on $i$.  Then $Q$ intersects $T_+,T_-$ in
            meridians, and if $Q$ intersects $T_{\kappa}$ at all it does so in
            meridians.

        \item When $\alpha=\nu$ and $g \geq 2$, $Q$ can be taken so that
            $\chi(Q)$ depends on $\kappa,\alpha$ but not on $i$.  Furthermore,
            the intersection of $Q$ with $T_-$ is in meridians, with $T_\kappa$
            in meridians (if non-empty), and with $T_+$ in a single curve which
            is distance one from both the meridional slope and the slope
            induced by $\Rhat$.
    \end{enumerate}
\end{lemma}

\begin{proof}
    Recall from section~\ref{sec:defKtau} that there is an annulus $\Uhat_g$ in
    $M(g,*,0)$, gotten from the trace of an isotopy of $\gamma_g$ into the
    interior of $M(g,*,0)$, such that the pair $(M(g,*,i),K(\tau,g,*,i))$ is
    gotten from the pair $(M(g,*,0),K(\tau,g,*,0))$ by twisting along
    $\Uhat_g$. Similarly, the pair $(\Rhat(\alpha),K(\kappa))$ in $M(g,*,i)$ is
    obtained from the pair $(\Rhat(\alpha),K(\kappa))$ in $M(g,*,0)$ by
    twisting along $\Uhat_g$.

    We first assume that $\alpha$ is not parallel to $\nu$ (or by
    assumption to either $\mu$ or $\lambda$).  Let $D$ be the disk in $M(g,*,i)$ whose
    boundary is the middle curve of Figure~\ref{fig:3diskbusting}(left) where
    $D \subset H_2 \subset H_g \cong M(g,*,i)$.

    \begin{claim}\label{clm:intwithD}
        The algebraic intersection number of $D$ with each of
        $\Lp$ and $\Lm$ in $M(g,*,i)$ is non-zero.
    \end{claim}

    \begin{proof}
        In the language of section~\ref{section:little-handlebody}, we may take
        $D \cap H = D_{+} \cup D_{-} \cup tD_s$ where $D_s$ is a separating
        product disk in $H$ in $M(g,*,i)$ and $tD_s$ means $t$ parallel copies of $D_s$ for a non-negative integer $t$.  We may take $t=32i$ when $*=H$ and $t=64i$ when $*=S$.  (Note: for $g=2$, $P(2,*,0)$ is pictured in Figure~\ref{fig:typesofpants}, and $P(2,*,i)$ is gotten by twisting along $\Uhat_2$).  
        The algebraic intersection number of $\Lp$ or $\Lm$ with $D$ then
        depends only on its intersections with $D_{-} \cup D_{+}$ with the
        appropriate orientation.  This is $\pm \geomint{\alpha}{\mu} \pm
        \geomint{\alpha}{\lambda}$ by
        Lemma~\ref{L-meets-boundary-compressing-disks} --- which is non-zero as
        long as $\alpha \neq \tau(1,1), \tau(1,-1)$. By inspection one sees
        that the algebraic intersection number of $D$ with
        $\tau(1,1)=\nu=c(g,*,i)$ is zero.  The algebraic intersection number of
        $\tau(1,-1)$ with the oriented $D_{-} \cup D_{+}= D_{\mu} \cup
        D_{\lambda}$ is different than that of $\tau(1,1)$, hence non-zero.
        This gives the conclusion.
    \end{proof}

    Tube oppositely oriented intersections of $D$ along $\Lp$ ($\Lm$) so that
    all intersections with $\Lp$ ($\Lm$) have the same sign. Let $Q$ be the
    intersection of this surface with the exterior, $M(g,*,i)_{\L}$,  of
    $\L=K(\kappa) \cup \Lp \cup \Lm$ in $M(g,*,i)$. Orienting $Q$, we see that
    $\bdry Q$ is a non-empty, coherently oriented collection of meridians on
    the  components, $T_+,T_-$, of $\partial M(g,*,i)_{\L}$.  Note that
    a catching surface does not need to be incompressible or
    $\bdry$--incompressible.  Thus $Q$ becomes a catching surface for the
    annulus/knot pair $(\Rhat,K(\kappa))$ in $M(g,*,i)$. Note that
    $\chi(Q)<0$ and depends on $\kappa, \alpha, i$. This is conclusion (1)
    of the Lemma.

    We now modify the surface $Q$  to make its Euler characteristic independent
    of $i$ when $g=2$. First we consider $(\Rhat(\alpha),K(\kappa))$ in
    $M(2,*,0)$. As above, we assume that $\alpha \neq \mu,\lambda,\nu$.
    $P(2,*,0)$ is given in Figure~\ref{fig:typesofpants}.  Let $D$ be the disk
    whose boundary is the middle curve of Figure~\ref{fig:3diskbusting}(left),
    and attach four bands to $D$ to obtain $\barQ$, properly embedded in
    $M(2,*,0)$, whose boundary does not meet the disk-busting curve $\gamma_2$.
    See Figure~\ref{fig:hbody-general-catching}.

    \begin{claim}
        The algebraic intersection number of $\barQ$ with each of $\Lp$ and
        $\Lm$ in $M(2,*,0)$ is non-zero.
    \end{claim}

    \begin{proof}
        In the language of section~\ref{section:little-handlebody}, $\barQ \cap
        H = D_{+} \cup D_{-} \cup tD_s$ where $D_s$ is a separating disk in $H$
        and where we may take $t=4$ when $*=H$ and $t=8$ when $*=S$. The
        algebraic intersection number of $\Lp$ or $\Lm$ with $\barQ$ is then
        $\pm \geomint{\alpha}{\mu} \pm \geomint{\alpha}{\lambda}$ by
        Lemma~\ref{L-meets-boundary-compressing-disks} --- which is non-zero as
        long as $\alpha \neq \tau(1,1),\tau(1,-1)$. As in the proof of
        Claim~\ref{clm:intwithD}, one sees that the algebraic intersection
        number of $\barQ$ with $\tau(1,1)=\nu$ is zero, hence with $\tau(1,-1)$
        is non-zero.
    \end{proof}

    Tube oppositely oriented intersections of $\barQ$ along $\Lp$ ($\Lm$),
    within $H$ of $M(2,*,0)$, so that all intersections with $\Lp$ ($\Lm$) have
    the same sign. Let $Q$ be the intersection of this tubed $\barQ$ with the
    exterior, $M(2,*,0)_{\L}$,  of $\L=K(\kappa) \cup \Lp \cup \Lm$ in
    $M(2,*,0)$. Orienting $Q$, we see that $\bdry Q$ is a non-empty, coherently
    oriented collection of meridians on the  components, $T_+,T_-$, of
    $\partial M(2,*,0)_{\L}$, and any intersections with $T_{\kappa}$ will be
    in meridians.  Thus $Q$ becomes a catching surface for the annulus/knot
    pair $(\Rhat,K(\kappa))$ in $M(2,*,0)$. Note that $\chi(\barQ)<0$.  As
    $\barQ$ is disjoint from $\gamma_2$, we may take $J_2$ disjoint from $Q$.
    After twisting along $\Uhat_2$, we see the image of $Q$, which we also
    refer to as $Q$, as a catching surface for $(\Rhat,K(\kappa))$ in
    $M(2,*,i)$.  The boundary of $Q$ on each of $T_{-},T_{+},T_{\kappa}$  is
    still a collection of meridians, and $\chi(Q)$ is independent of $i$ as
    desired when $g=2$.  This is conclusion $(2)$ of the Lemma.

    \begin{figure}[h!tb]
        \begin{center}
            \includegraphics[width=0.95\textwidth]{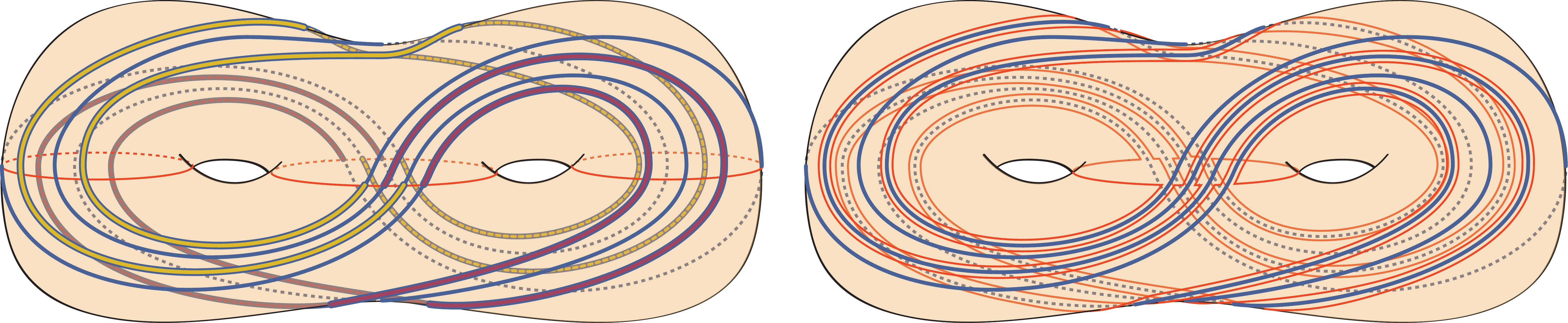}
        \end{center}
        \caption{(Left) Four subarcs of $\gamma_2$ that connect the central
        meridian disk to itself from the same side are highlighted. (Right)
        The boundary of the resulting orientable surface $\barQ$ is shown to be
        disjoint from $\gamma_2$.}
        \label{fig:hbody-general-catching}
    \end{figure}

    \begin{figure}[h!tb]
        \begin{center}
            \includegraphics[width=0.75\textwidth]{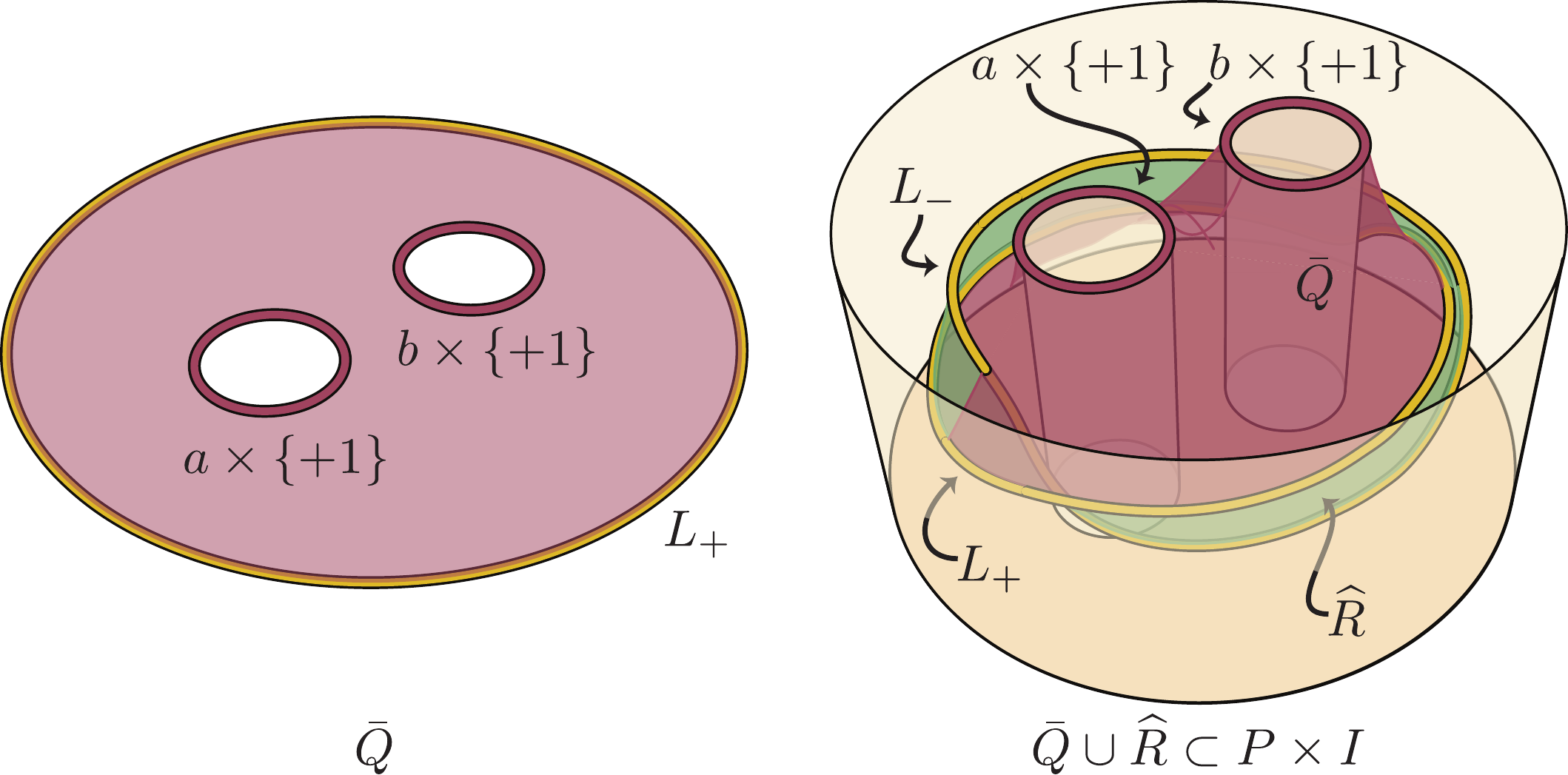}
        \end{center}

        \caption{The $3$--punctured sphere $\barQ$ resides in $P\times I$ with
        $\widehat{R}$ as shown.  Compare with $\nu \subset T\times \{0\}$ in
    Figure~\ref{fig:HandP}. }

        \label{fig:Qforalphaisnuinlittlehandlebody}
    \end{figure}

    We are left to consider the case that $\alpha$ is parallel to $\nu$ in $T$
    in any genus $g$. We will first exhibit a $3$--punctured sphere $\barQ$ in
    the little handlebody $H=P\times I$ with boundary $a \times \{1\} \cup
    b \times \{1\} \cup L_+$ that is intersected once by $L_-$ and some number
    of times by $K(\kappa)$.  Since $L_+$ is isotopic to $\nu$ which is
    isotopic to $c \times \{1\}$, we may take $\barQ$ to be the result of
    isotopy of $P\times \{1\}$ that fixes $a \times \{1\} \cup b \times \{1\}$
    and pulls the rest into the interior of $P\times I$.  See
    Figure~\ref{fig:Qforalphaisnuinlittlehandlebody}.  Because the framing that
    $T\times \{0\}$ gives $\nu$  differs from the framing that $P \times \{1\}$
    gives $c$ by $1$, we see that $L_-$ intersects $\barQ$ just once. Then we
    take $Q$ to be the surface $\barQ - \nbhd(L_+ \cup L_- \cup K(\kappa))$.
    Now we may attach $P\times I$ (along $P \times \{-1\}$) to $P(g,*,i)$  so
    that $Q$ becomes a catching surface for $(\Rhat(\alpha), K(\kappa))$ in
    $M(g,*,i)$. Then $\chi(Q)$ is independent of $i$. Furthermore, $Q$
    intersects $T_{-}$ in a single meridian, intersects $T_{\kappa}$ in
    meridians (if at all), and intersects $T_{+}$ in a single component whose
    distance from the meridian and from $\partial \Rhat$ is one. This is
    conclusion $(3)$.
\end{proof}

We can now give lower bounds on the bridge numbers of the knots
$K(\tau,g,*,i)$. These will be independent of $i$ when $g=2$ or $\alpha=\nu$.

\begin{prop}\label{prop:bridge-number-bound}
    Fix a genus $g$, and let $N_{strong}$ be as in
    Proposition~\ref{prop:strongandhyp}.  Let $\kappa, \alpha$ be curves in $T$
    such that $\alpha$ is not parallel to $\mu$ or $\lambda$ in $T$.  Let
    $\tau=\tau(\kappa,\alpha,n)$.  For  $|i| > N_{strong}$, there is
    a positive constant $C=C(g,\alpha,\kappa,i)$ such that \[
    b_g(K(\tau,g,*,i)) \geq \frac{1}{2}\left(|n| C -2g\right). \] When either
    $g=2$ or $\alpha=\nu$, the constant $C$ can be chosen to be independent of
    $i$ (given $|i| >N_{strong}$).
\end{prop}

\begin{proof}
    This follows by applying~\cite[Theorem 1.3]{BGL} to the annulus/knot pair
    $(\Rhat(\alpha),\kappa)$ in $M=M(g,*,i)$. Assume $|i|>N_{strong}$.  By
    Proposition~\ref{prop:strongandhyp} and Lemma~\ref{R-not-isotopic},
    $\Rhat(\alpha)$ is not isotopic into the genus $g$ splitting of $M$.
    Furthermore,  by Lemma~\ref{lem:no-spanning-annulus1} there is no essential
    annulus in $M_{\L}$ with one boundary component on each of $T_+$ and $T_-$
    (recall the standing assumption, see Definition~\ref{defn:twisting}, that
    $\kappa$ and  $\alpha$ are not isotopic in $T$.).  Let $Q$ be the
    catching surface of Lemma~\ref{lem:catching} for the annulus/knot pair
    $(\Rhat(\alpha),\kappa)$ in $M$. Since the components of $Q$ are
    meridional on $T_{\kappa}$, and either meridional or longitudinal and
    distance one from the slope of the framing induced by $\Rhat(\alpha)$
    on $ T_+$ and $T_-$, then, in the notation of Theorem 1.3 of
    \cite{BGL}, $\Delta_K$ is zero, $|p_1|,|p_2|=1$, and $|q_1|,|q_2| \leq
    1$. As $\chi(Q)$ is negative we conclude that

    \[ b_g(K(\tau,g,*,i)))\geq \frac{1}{2}\left(\frac{|n|}{-36\chi(Q)}-2g\right). \]

    Noting that $\chi(Q)$ depends only on $\alpha,\kappa,g, i$, and that, when
    $g=2$ or $\alpha=\nu$, $\chi(Q)$ depends only on $\alpha,\kappa,g$, this
    becomes the statement of the Proposition.
\end{proof}

Let the curves $\kappa$ and $\alpha$ be $(r,s)$ and $(t,v)$ curves with respect
to the basis on $T$ given by $\mu$ and $\lambda$. The result of $n$ Dehn twists
of $\kappa$ along $\alpha$ is a $(r + n\geomint{\kappa}{\alpha}t,
s + n\geomint{\kappa}{\alpha}v)$ curve,  where as in
Definition~\ref{def:distance}, $\Delta$ represents the geometric intersection
number in $T$.  That is, in the notation of
section~\ref{section:little-handlebody} $\tau(\alpha,*,n)=\tau(r
+ n\geomint{\kappa}{\alpha}t, s + n\geomint{\kappa}{\alpha}v)$.

We are now in a position to prove the main result

\begin{theorem}\label{thm:main}
    Let $M$ be a handlebody of genus $g>1$. For every positive integer $N$
    the following hold.
    \begin{enumerate}
        \item There are infinitely many knots $K\subseteq M$ such that $K$
            admits a nontrivial handlebody surgery and
            \[ b_g(K) \geq N. \]
		Furthermore, the knots may be taken to have the same
            genus $g$ bridge number.

        \item There are infinitely many pairs of relatively prime integers $p$
            and $q$ such that for each pair, there are infinitely many knots
            $K\subseteq M$ admitting a surgery yielding a $D(p,q)$--Seifert space in 
            union with $(g-1)$ $1$--handles.  Furthermore, for each such
            $K$
            \[ b_g(K)\geq N \]
            Finally, fixing $(p,q)$, the knots may be taken to have the same
            genus $g$ bridge number.
    \end{enumerate}

    The knots in each family above have exteriors in $M$ which are irreducible,
    boundary-irreducible, atoroidal, and anannular. The slope of each
    surgery is longitudinal, that is, intersecting the meridian once, and each
    knot admits \emph{exactly one} nontrivial $\boundary$--reducing surgery.
\end{theorem}

\begin{remark}
    The common bridge number for each family above is not prescribed, it is
    only known to be  bigger than $N$.
\end{remark}

\begin{proof}
    Let $N_{strong}$ be the constant of Proposition~\ref{prop:strongandhyp} for
    the given $g$.  Let $N$ be any given integer.  Fix a relatively prime pair of
    nonnegative integers $r,s$ with $r>s$, and let $\kappa=\tau(r,s)$ and
    $\alpha=\nu=\tau(1,1)$.  Let $C(g,\alpha,\kappa)$ be as in
    Proposition~\ref{prop:bridge-number-bound} and take $n$ large enough so that
    this proposition guarantees that $b_g(K(\tau(\kappa,\alpha,n),g,*,i)) \geq
    N$ for any $i>N_{strong}$. Note that because $\alpha=\nu$,
    $C(g,\alpha,\kappa)$ is independent of $i$ and this bridge number bound
    holds for any $i>N_{strong}$. In addition, ensure that $n$ is large enough
    so that $\tau(\kappa,\alpha,n)=\tau(r + n\geomint{\kappa}{\alpha},
    s + n\geomint{\kappa}{\alpha})=\tau((n+1)r-ns,nr-(n-1)s)$ is not isotopic
    in $T$ to any of $\{\lambda, \mu, \nu, \lambda-\mu, \lambda+\nu, \mu+\nu\}$
    (e.g.\  by ensuring that $\max(|(n+1)r-ns|,|nr-(n-1)s|)>2$).  As the
    disk-hitting number is a knot invariant,
    Proposition~\ref{prop:wrappingno}(4) shows that, fixing $n$ as above, as
    $i$ grows there are infinitely many different knots
    $K(\tau(\kappa,\alpha,n),g,*,i)$ in the genus $g$ handlebody $M \cong
    M(g,*,i)$. Proposition~\ref{prop:wrappingno}(3) shows that all of the knots
    $K(\tau(\kappa,\alpha,n),g,*,i)$ have a common upper bound on their bridge
    number, hence we may take an infinite subcollection that has  the same
    bridge number. Finally, as long as we take $|i|>N_{strong}$,
    Proposition~\ref{prop:strongandhyp} shows that the exteriors of these knots
    are irreducible, boundary-irreducible, atoroidal, and anannular.

    When $*=H$, Proposition~\ref{prop:wrappingno}(1) shows that these knots admit longitudinal handlebody surgeries.  
When $*=S$, Proposition~\ref{prop:wrappingno}(2) shows that these knots admit
a longitudinal surgery which is a $D((n+1)r-ns,nr-(n-1)s)$--Seifert space
with $(g-1)$ $1$--handles attached. 
(Note that in the case $*=S$,
we can arrange different $D((n+1)r-ns,nr-(n-1)s)$--Seifert spaces
with $(g-1)$ $1$--handles attached to be obtained as the result of Dehn surgery on knots in $M$
with bridge number bigger than $N$ as claimed 
by taking different nonnegative integers $r,s$ or larger integers $n$.)

Finally, we show that each such knot constructed admits exactly one nontrivial
boundary-reducing surgery.  This argument is adapted from~\cite[Lemma 5.6
and Proposition 5.7]{Bowman13}.  
Assume the above  longitudinal, $\bdry$--reducing surgery on $K=K(\tau(\kappa, \alpha,n),g,*,i)$ occurs along the slope $\sigma$.
Suppose that there is a second nontrivial
boundary-reducible surgery on $K$ in $M$ with slope $\rho \neq \sigma$, so that
$M(\rho)\cong H'\cup_P H(\rho)$.  
The first paragraph of the proof of Proposition~\ref{prop:strongandhyp} shows that, by our choice of $P \subset \bdry H'$ (with $|i|>N_{strong}$), the curve $\bdry_0 P$ is $3$--disk-busting in $H'$.  Hence any $\bdry$--reducing disk for $M(\rho)$ is either contained in $H(\rho)$ (and disjoint from $P$) or contains a $\bdry$--compressing disk for $P$ among the components of its intersection with $H(\rho)$.   Thus we may focus our attention upon $H(\rho)$.   Therefore we restrict attention to $\rho$--surgery upon $K \subset H$ and identify $K$ with $\tau = \tau(\kappa,\alpha,n)$.

By Proposition~\ref{prop:exterior-hyperbolic}
and a result of Wu~\cite{Wu92}, since $\tau \not \in \{\lambda, \mu, \nu\}$, this third surgery slope $\rho$ must have distance one
from $\sigma$ and from the meridional slope of $K$.  
Therefore, recalling Figure~\ref{fig:xyzV2}, since $H \cong T \times I$  and $K$ lies in $T \times\{0\}$ with framing $\sigma$, the effect of $\rho$--surgery on $K$ in $H$ may be viewed as altering the gluing between $T\times [-1,0]$ and $T\times[0,1]$ by a twist along $K$. In particular, we may identify $H(\rho)$ with $H=T\times[-1,1]$ except where $P \subset \bdry (T \times [-1,1])$ has been altered by a Dehn twist along $K \times \set{1}$.  Thus any essential simple closed curve in $T\times \{\pm1\} \subset \bdry H$ continues to be a primitive curve in $H(\rho)$.  Notably, the curves $\bdry P_\pm$ are each primitive in $H(\rho)$.

By a result of Ni~\cite{Ni}, if $\rho$--surgery on $K$ in $H = P \times I$ preserves the product structure 
(i.e.\ $P$ continues to be a fiber in $H(\rho) \cong P \times I$),
 then $K$ meets every non-separating $\bdry$--compressing disk for $P$ in $H$ at most twice. Indeed this would imply $\tau \in \{\lambda, \mu, \nu, \lambda-\mu, \lambda+\nu, \mu+\nu\}$, cf.\ Lemmas~\ref{L-meets-boundary-compressing-disks} and \ref{lem:min-D-cap-P}.  Above, we have ensured that $\tau(\kappa,\alpha,n)$ does not belong to this set.  Hence $H(\rho)$ does not have a product structure with $P$ as a fiber.

If $P$ is $\bdry$--compressible in $H(\rho)$, then since $H(\rho)$ is not a product with fiber $P$ and the curves $\bdry_\pm P$ are primitive curves, \cite[Lemma 5.6]{Bowman13} implies that $\bdry_0 P$ is not primitive and every $\bdry$--compressing disk for $P$ in $H(\rho)$ is disjoint from $\bdry_0 P$.  Since $\bdry_0 P$ would be isotopic to a boundary component of an annulus resulting from such a $\bdry$--compression of $P$, \cite[Proposition 5.5.]{Bowman13} implies that $\bdry_0 P$ would have to be primitive in either $H'$ or $H(\rho)$.  However, neither of these can occur since it is $3$--disk-busting in $H'$ and non-primitive in $H(\rho)$.  Thus $P$ is not $\bdry$--compressible.

Finally, if the complementary pair of pants $\bdry H(\rho)-P$ compresses in $H(\rho)$, then since the curves $\bdry_\pm P$ are primitive, it must be that $\bdry_0 P$ bounds a compressing disk.  However, in the expression of $H(\rho)$ as $T \times I$ with $P$ Dehn twisted once along $K \times \{1\}$,  $\bdry_0 P$ traverses the annulus $(\bdry T) \times I$ twice.  Hence the compressing disk with boundary $\bdry_0 P$ is a product disk and thus the curves $\bdry_\pm P$ in $T \times \{\pm1\}$ project to the same curve in $T$.  Yet since $\bdry_\pm P$ in $H$ projected to curves $\lambda$ and $\mu$ in $T$, this implies that $K$ must be isotopic to $\nu=\lambda+\mu$ or $\lambda-\mu$.

 This finishes the proof that there is exactly one nontrivial
boundary-reducible surgery on $K\subset M$.
\end{proof}

If we are not concerned about determining a lower bound on bridge number, the following points out 
that for any $(p,q)$, we can obtain the boundary
connected sum of a $D(p,q)$--Seifert space with a genus  $g-1$--handlebody
by Dehn surgery on infinitely many knots in a genus $g$ handlebody.

\begin{theorem}\label{thm:main2}
    Let $M$ be a handlebody of genus $g>1$ and $p,q$ be  non-zero, relatively prime
    integers.  There are infinitely many knots $K\subseteq M$ admitting
    a longitudinal surgery yielding a  $D(p,q)$--Seifert space with $(g-1)$
    $1$--handles attached.  Furthermore,  the exterior in $M$ of each knot is  irreducible,
    boundary-irreducible, atoroidal, and anannular. Finally, for fixed $p,q$,
    these knots can be taken  to have the same genus $g$ bridge number in $M$.
\end{theorem}

\begin{proof}
    If either $|p|=1$ or $|q|=1$ then the $D(p,q)$--Seifert space is a solid
    torus and this becomes a special case of Theorem~\ref{thm:main}$(1)$. So we
    assume $|p|,|q|>1$. Then $\tau(p,q) \not\in\{\lambda, \mu, \nu, \lambda-\mu,
    \lambda+\nu, \mu+\nu\}$. Proposition~\ref{prop:wrappingno}(2) shows that
    $K(\tau(p,q),g,S,i)$ admits a surgery which is a $D(p,q)$--Seifert space
    union $g-1$ $1$--handles. As the disk-hitting number is a knot invariant,
    Proposition~\ref{prop:wrappingno}(4) shows that as $i$ grows there are
    infinitely many such knots. Proposition~\ref{prop:wrappingno}(3) shows that all
    of these have a common upper bound on their bridge number, hence can be taken
    to have the same bridge number. Finally, as long as we take $|i|>N_{strong}$,
    Proposition~\ref{prop:strongandhyp} shows that the exteriors of these knots are
    irreducible, boundary-irreducible, atoroidal, and anannuluar.
\end{proof}

\section{The genus two case}\label{section:genus2}
In this section we clarify the connection between the knots in genus $2$ handlebodies
constructed here and those of~\cite{Bowman13}.
Proposition~\ref{prop:bridge-number-bound} considerably strengthens the results
of~\cite{Bowman13} in certain cases. It identifies  new infinite subcollections
of knots constructed in \cite{Bowman13} that have bridge number bigger than
one.  In fact, it shows these bridge numbers can be taken to be arbitrarily
large.  This is the content of Corollary~\ref{seansknots} below. We finish by exploring a
technique for producing handlebody/pants pairs that are of strong 
handlebody or strong Seifert type.  By Definition~\ref{def:bighandlebodyM}, and Propositions~\ref{k-has-handlebody-surgery}, ~\ref{prop:exterior-hyperbolic},  these handlebody/pants pairs 
give rise to lots of knots in a genus $2$ handlebody with 
boundary-reducing surgeries whose exteriors are irreducible, boundary-irreducible, atoroidal, and anannular. 
Note that subfamilies of these knots can be generated by annulus twisting, and the results of \cite{BGL} can
often be applied
to establish lower bounds on bridge number -- as is done with the knots 
$K(\tau(\kappa,\alpha,n),2,*,i)$ in
section~\ref{sec:bridgenumberbounds}.

We first show how $2$--bridge knot or link exteriors give rise to pairs
$(H',P')$ of handlebody or Seifert type 
(see Definitions~\ref{def:handlebodytype},~\ref{def:seiferttype}).

\begin{definition}\label{associatedpair}
Let $J$ be a non-trivial $2$--bridge knot or link in canonical
$2$--bridge position with respect to some height function
$h\co S^3\to \R$.  Then $J$ has
an unknotting tunnel connecting its two maxima, \ie an arc $t$ such that
$H_J'=S^3_{J\cup t}$ is a handlebody.  Furthermore, choose $t$ so that it
has one maximum and no minimum under $h$ in its interior.  The link exterior
$S^3_J$ is obtained from $H_J'$ by attaching a $2$--handle along an attaching
curve $c\subseteq\boundary H'$.  Let $P_J'$ be the pair of pants embedded
in $\boundary H_J'$ as shown in Figure~\ref{fig:Hprime-P}, so that $\boundary P_J' = a\cup b\cup c$ 
where $a$ and $b$ are
meridians of one component of $J$.  
\end{definition}

\begin{figure}[h!tb]
  \begin{center}
  \includegraphics[width=0.5\textwidth]{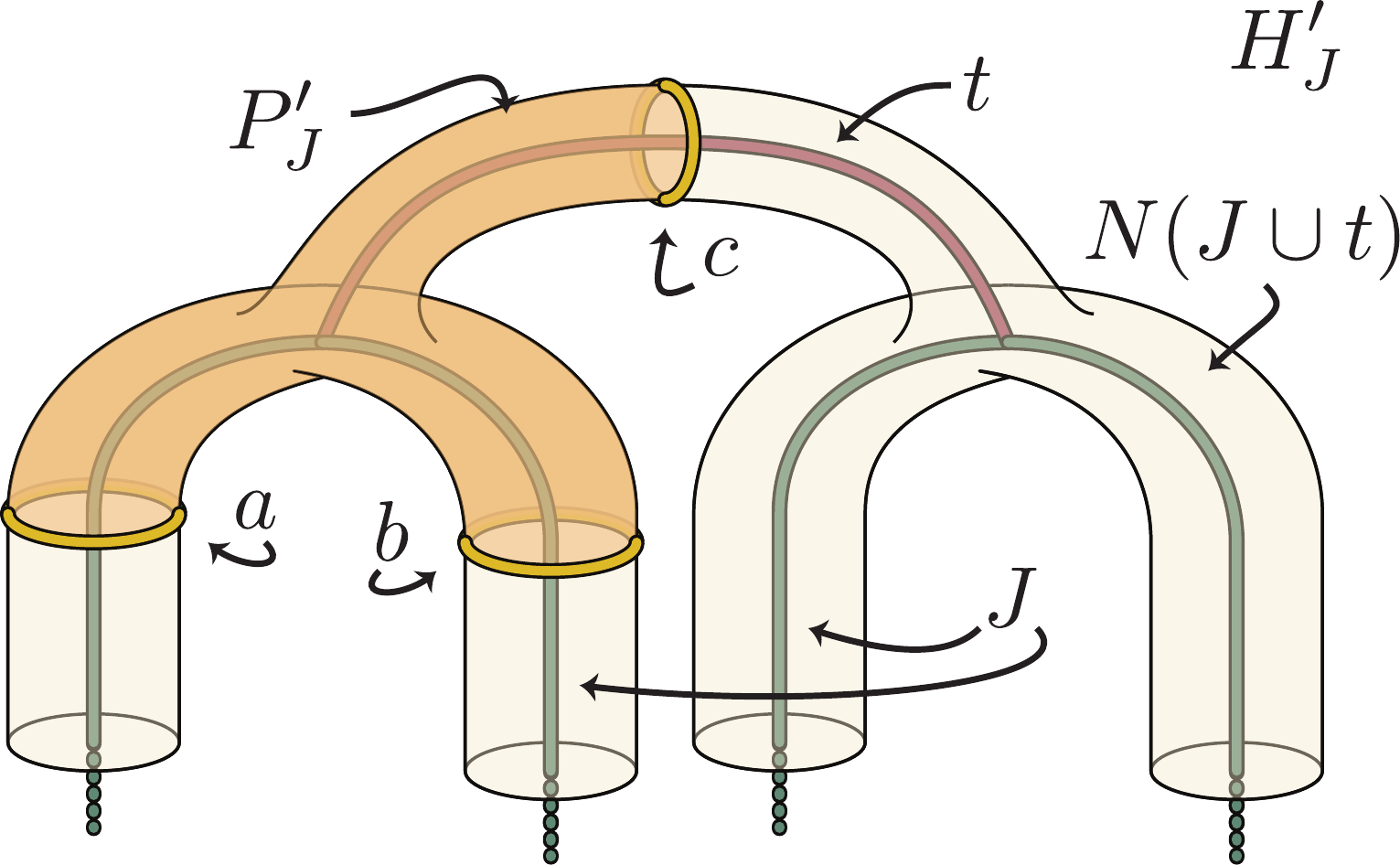}
  \end{center}
  \caption{Part of the handlebody $H_J' = S^3-\nbhd(J \cup t)$ together with $P_J'$}
  \label{fig:Hprime-P}
\end{figure}

\begin{theorem}  Let $H'$ be a genus two handlebody and the pair $(H',P')$ be 
of handlebody type. Then there is a $2$--bridge knot $J$ such that 
$(H',P')$ and $(H_J',P_J')$ are homeomorphic as pairs. 

In particular, let $(H_2,P(2,H,i))$ be as in section~\ref{sec:defKtau}. Then for each $i$ 
there is a $2$--bridge knot $J_i$ such that
$(H_2,P(2,H,i))$ is $(H_{J_i}', P_{J_i}')$. The handlebody/knot pair
$(M(2,H,i),K(\tau(p,q),2,H,i))$ is then the pair $(M,
K_{q,p}^{J_i})$ constructed in~\cite{Bowman13}. 
\end{theorem}

\begin{proof}
First suppose that $(H', P')$ is of handlebody type. Let $\partial P' = a \cup b \cup c$ where $a,b$
are jointly primitive in $H'$.  Attach $2$--handles to $H'$ along
$a$ and $b$, and let $a'$ and $b'$ be the cocores of the
handles.  Because $a,b$ are jointly primitive, the
arcs $a'$ and $b'$ in the resulting ball form a rational tangle.  
Attaching a $2$--handle to $H'$ along $c$ gives the exterior of
$2$--bridge knot, $J$, corresponding to capping the tangle with two unknotted
arcs (bridge arcs) outside the tangle. The cocore of the attached $2$--handle is an 
unknotting tunnel for the exterior of $J$ connecting the two bridge arcs. That is,
$(H',P')$ is homeomorphic to $(H_J',P_J')$.
Having made this identification, the construction of the  pair
$(M(2,H,i),K(\tau(p,q),2,H,i))$ is the same as  for the pair $(M,
K_{q,p}^{J_i})$ in~\cite{Bowman13}. Here we are using the convention that $(\mu,\lambda)$ of $T$ in
section~\ref{section:little-handlebody} corresponds to $(m,l)$ of section 4 in
\cite{Bowman13}.
\end{proof}

The knots $K_{r,s}^L$ in~\cite{Bowman13} are knots in a genus $2$ handlebody 
that admit non-trivial handlebody surgeries. It is  shown there that for certain parameter
values, these knots have bridge number bigger than one in the handlebody.
Below, we find new parameter values for which the corresponding knots have
bridge number bigger than one. Indeed, the bridge numbers get arbitrarily large.

\begin{corollary}\label{seansknots}
    Let $K_{r,s}^{L}$ be the knots in the genus $2$ handlebody $M$ constructed in~\cite{Bowman13}. 
    Let $k$ be a nonzero integer, and fix a positive integer $N$. Let
    $N_{strong}$ be the constant for $g=2$ given by
    Proposition~\ref{prop:strongandhyp}.  For $|i|>N_{strong}$ there are only
    finitely many knots in the families $\set{K_{kn+1,n}^{J_i}}$ and
    $\set{K_{kn-1,n}^{J_i}}$ with bridge number less than $N$ in $M$.
\end{corollary}
\begin{proof}
    The knot $K_{kn+1,n}^{J_i}$ is the knot $K(\tau(n,kn+1),2,H,i)=
    K(\tau(\kappa,\alpha,n),2,H,i)$ where $\kappa=\lambda$ and $\alpha=\tau(1,k)$.
    Similarly, the knot $K_{kn-1,n}^{J_i}$  is the knot
    $K(\tau(n,kn-1),2,H,i)=K(\tau(\kappa,\alpha,n-1),2,H,i)$ where
    $\kappa=\tau(1,k-1)$ and $\alpha=\tau(1,k)$.
    The result now follows from Proposition~\ref{prop:bridge-number-bound}.
\end{proof}

Definition~\ref{associatedpair} gives many examples of pairs $(H',P')$ which are
of strong handlebody or strong Seifert type, and, consequently lots of hyperbolic
knots in a genus $2$ handlebody with boundary-reducing surgeries (see Definition~\ref{def:bighandlebodyM}, and Propositions~\ref{k-has-handlebody-surgery}, ~\ref{prop:exterior-hyperbolic}).

\begin{theorem}
Let $J$ be a nontrivial $2$--bridge knot. The pair $(H_J',P_J')$ is of handlebody type.
If the exterior of $J$ is anannular 
then $(H_J',P_J')$ is of strong handlebody type.

Let $J$ be  a $2$--bridge link of two components. The pair $(H_J',P_J')$ is of Seifert type.
If the exterior of $J$ is irreducible and anannular, 
then $(H_J',P_J')$ is of strong Seifert type.
\end{theorem}

\begin{proof}
Let $J$ be a non-trivial $2$--bridge knot and $(H_J',P_J')$ be as in Definition~\ref{associatedpair}.
Attaching a $2$--handle to $H_J'$ along $a$ or $b$, we obtain a new knot exterior
for which the associated knot has exactly one maximum and one minimum with
respect to $h$.  This must be the unknot, and therefore $a$ and $b$ are
primitive curves in $\boundary H_J'$.  The remaining
components of $\boundary P_J'$ become meridians of this unknot exterior, so $a$
and $b$ are in fact jointly primitive.  Thus $(H_J',P_J')$ is of handlebody type.

When $J$ is a $2$--bridge link, $a$ and $b$ are parallel and $P'$ is
obtained by banding annular neighborhoods of $a$ and $b$ in the complement of
the annulus cobounded by $a$ and $b$. As each component of a $2$--bridge link
is unknotted, $a,b$ are each primitive. Thus $(H_J',P_J')$ is of Seifert type.

Assume that $J$ is a $2$--bridge knot or link whose exterior, $S^3_J$, is irreducible, 
anannular, and has incompressible boundary. 
That exterior is gotten by attaching a $2$--handle along $c$ to $H_J'$.
Assume there is a properly embedded essential disk or annulus
$A$ in $H_J'$ that misses $c$. By the assumption on $S^3_J$, it must be that all components
of $\partial A$ become inessential in $\partial S^3_J$ after adding a $2$--handle along $c$. 
First assume that $A$ is a disk. If $\partial A$ is non-separating in $\partial H_J'$, then it is isotopic
to $c$. But then capping off $\partial A$ with the added $2$--handle gives an essential sphere in 
$S^3_J$, which is irreducible. So $\partial A$ is separating in $\partial H_J'$. Each side of $A$ in
$H_J'$ is a solid torus. In particular, this implies that $S^3_J$ is a solid torus, contradicting
the incompressibility of its boundary. Thus it must be that $A$ is an annulus. If both components are parallel
in $\partial H_J'$, then attaching a $2$--handle along $c$ must give a lens space summand ($A$ is
essential in $H_J'$). But this
is not possible, as such a submanifold lives in the $S^3_J$. Thus it must be that one component
of $\partial A$ is parallel to $c$ and the other bounds a once-punctured torus in $\partial H_J'$ containing
$c$. But then we  can cap off these boundary components in $S^3_J$ to get a non-separating $2$--sphere,
a contradiction. Thus there is no such essential disk or annulus missing $c$. We conclude that when
$J$ is a knot, $(H_J', P_J')$ is of strong handlebody type. When $J$ is a two component link, 
we have that $(H_J',P_J')$ is of strong Seifert type by Lemma~\ref{P-essential-in-Hp} below.
\end{proof}

\begin{lemma}\label{P-essential-in-Hp}
Assume $J$ is a $2$--bridge knot or link which is not the unknot, the Hopf link,
or a split link.  
The surface $P_J'$ is $\boundary$--incompressible in $H_J'$.
\end{lemma}
\begin{proof}
  Suppose that $D$ is a $\boundary$--compressing disk for $P_J'$ in $H_J'$.  If
  $\boundary D$ is non-separating in $P_J'$, then $\boundary D\cap c$ is empty
  since $c$ is not primitive ($J$ is not the unknot).  So $D$ becomes an
  essential disk in the link exterior $S^3_J$, which is impossible by
  hypothesis.

  If $\boundary D$ separates $H_J'$ (so that it also separates $P_J'$), we obtain
  two solid tori after performing the $\boundary$--compression, and each
  contains an annulus coming from $P_J'$ in its boundary.  If $\boundary D$ meets $a$,
  say, then examine the solid torus with the annulus containing $b$
  on its boundary.  Since a primitive curve in a handlebody of genus $g>1$
  remains primitive after cutting along a disk disjoint from the curve, the
  annulus containing $b$ must be longitudinal in the boundary of its
  solid torus.  However, this implies that there is a non-separating
  $\boundary$--compressing disk for $P_J'$ and we apply the argument above.  
  A similar argument works if
  $\boundary D$ meets $b$ or $c$.

  Suppose then that $\boundary D$ separates $P_J'$ but not $\boundary H_J'$.  If
  $\boundary D$ meets $a$ or $b$, we obtain a disk in a link exterior after
  attaching a $2$--handle along $c$.  This disk must be trivial in the exterior
  of $J$ by hypothesis.  Furthermore, since $\boundary D$ does not separate
  $\boundary H_J'$, it must be parallel to $c$.  This gives a non-separating
  sphere in the exterior $S^3_J$, which is impossible. On the other hand, 
  if $\boundary D$ meets $c$, then $D$ extends across $N(t)$ ($t$ the tunnel of Definition~\ref{associatedpair}) to
  an annulus $A$ properly embedded in $S^3_J$ one of whose boundary components
  is a meridian of a component of $J$.  Therefore $A$ is incompressible in
  $S^3_J$.  We may take $t$ to be the cocore of $A$.  Since $t$ is not
  boundary-parallel in $S^3_J$, $A$ is also not boundary-parallel.
  When $J$ is a knot, both boundary components of $A$ are meridians, implying
  that $J$ is not prime.  This is impossible since $J$ is $2$--bridge.  On the
  other hand, when $J$ is a two component link, $A$ shows that $J$ must be the
  Hopf link, contradicting our hypotheses.
\end{proof}

%%%%%%%%%%%%%%%%%%%%%%%%%%%%%%%%%%
%%%%%%%%%%%%%%%%%%%%%%%%%%%%%%%%%%
%%%%%%%%%%%%%%%%%%%%%%%%%%%%%%%%%%

\bibliographystyle{amsalpha}%{plain}
\bibliography{bridge}%

\providecommand{\bysame}{\leavevmode\hbox to3em{\hrulefill}\thinspace}
\providecommand{\MR}{\relax\ifhmode\unskip\space\fi MR }
\providecommand{\MRhref}[2]{%
  \href{http://www.ams.org/mathscinet-getitem?mr=#1}{#2}
}
\providecommand{\href}[2]{#2}
\begin{thebibliography}{CGLS87}

\bibitem[Ber]{Berge}
John Berge, \emph{Knots in handlebodies which can be surgered to produce
  handlebodies}, Unpublished manuscript.

\bibitem[Ber91]{Berge91}
\bysame, \emph{The knots in ${D}^2\times {S}^1$ which have nontrivial {D}ehn
  surgeries that yield ${D}^2\times {S}^1$}, Topology and its Applications
  \textbf{38} (1991), no.~1, 1 -- 19.

\bibitem[BGL13]{BGL}
Kenneth~L. Baker, Cameron Gordon, and John Luecke, \emph{Bridge number and
  integral {D}ehn surgery}, 2013, Alg. Geom. Top., To appear.

\bibitem[Bow13]{Bowman13}
R.~Sean Bowman, \emph{Knots in handlebodies with nontrivial handlebody
  surgeries}, J. Topology \textbf{6} (2013), no.~4, 955--981.

\bibitem[CGLS87]{CGLS}
M.~Culler, C.~McA. Gordon, J.~Luecke, and P.~B. Shalen, \emph{Dehn surgery on
  knots}, Ann. Math. \textbf{125} (1987), 237--300.

\bibitem[FMP03]{frigerio}
Roberto Frigerio, Bruno Martelli, and Carlo Petronio, \emph{Dehn filling of
  cusped hyperbolic {$3$}-manifolds with geodesic boundary}, J. Differential
  Geom. \textbf{64} (2003), no.~3, 425--455. \MR{2032111 (2005d:57024)}

\bibitem[Gab89]{Gabai89}
David Gabai, \emph{Surgery on knots in solid tori}, Topology \textbf{28}
  (1989), no.~1, 1--6. \MR{991095 (90h:57005)}

\bibitem[Gab90]{Gabai90}
\bysame, \emph{$1$--{B}ridge braids in solid tori}, Topology Appl. \textbf{37}
  (1990), no.~3, 221--235.

\bibitem[GL89]{GordonLuecke89}
C.~McA. Gordon and J.~Luecke, \emph{Knots are determined by their complements},
  J. Amer. Math. Soc. \textbf{2} (1989), no.~2, 371--415. \MR{965210
  (90a:57006a)}

\bibitem[Gor87]{Gordon87}
C.~McA. Gordon, \emph{On primitive sets of loops in the boundary of a
  handlebody}, Topology Appl. \textbf{27} (1987), no.~3, 285--299. \MR{918538
  (88k:57013)}

\bibitem[Gor98]{Gordon98}
\bysame, \emph{Boundary slopes of punctured tori in $3$--manifolds}, Trans.
  Amer. Math. Soc. \textbf{350} (1998), no.~5, 1713--1790.

\bibitem[Ni11]{Ni}
Yi~Ni, \emph{Dehn surgeries on knots in product manifolds}, Journal of Topology
  \textbf{4} (2011), no.~4, 799--816.

\bibitem[Sch90]{Scharlemann90}
Martin Scharlemann, \emph{Producing reducible {$3$}-manifolds by surgery on a
  knot}, Topology \textbf{29} (1990), no.~4, 481--500. \MR{1071370 (91i:57003)}

\bibitem[Wu92a]{Wu92b}
Ying~Qing Wu, \emph{A generalization of the handle addition theorem}, Proc.
  Amer. Math. Soc. \textbf{114} (1992), no.~1, 237--242. \MR{1070535
  (92c:57018)}

\bibitem[Wu92b]{Wu92}
\bysame, \emph{Incompressibility of surfaces in surgered {$3$}-manifolds},
  Topology \textbf{31} (1992), no.~2, 271--279. \MR{1167169 (94e:57027)}

\bibitem[Wu93]{Wu93}
\bysame, \emph{{$\partial$}-reducing {D}ehn surgeries and {$1$}-bridge knots},
  Math. Ann. \textbf{295} (1993), no.~2, 319--331. \MR{1202395 (94a:57036)}

\bibitem[Yos14]{Yoshizawa}
Michael Yoshizawa, \emph{High distance {H}eegaard splittings via {D}ehn
  twists}, Algebr. Geom. Topol. \textbf{14} (2014), no.~2, 979--1004.
  \MR{3180825}

\end{thebibliography}

\end{document}